\documentclass[10pt]{article}
\usepackage{amssymb,enumerate}
\usepackage{amsmath,amsfonts}
\usepackage{amsthm}
\usepackage{latexsym}
\usepackage{mathrsfs}
\usepackage{color}
\usepackage{microtype}
\usepackage{geometry}

\usepackage{graphicx}

\usepackage{algorithm}
\usepackage{algorithmicx}
\usepackage{algpseudocode}

\usepackage{multirow}
\allowdisplaybreaks

\DeclareMathOperator{\cK}{\ensuremath{\mathcal{K}}}
\DeclareMathOperator{\cN}{\ensuremath{\mathcal{N}}}
\DeclareMathOperator{\cC}{\ensuremath{\mathcal{C}}}
\DeclareMathOperator{\cO}{\ensuremath{\mathcal{O}}}
\DeclareMathOperator{\cS}{\ensuremath{\mathcal{S}}}
\DeclareMathOperator{\cL}{\ensuremath{\mathcal{L}}}

\DeclareMathOperator{\cA}{\ensuremath{\mathcal{A}}}

\DeclareMathOperator{\bR}{\ensuremath{\mathbb{R}}}
\DeclareMathOperator{\bK}{\ensuremath{\mathbb{K}}}

\DeclareMathOperator{\low}{\ensuremath{\mathrm{low}}}

\DeclareMathOperator{\rmint}{\ensuremath{\mathrm{int}}}

\DeclareMathOperator{\sgn}{\ensuremath{\mathrm{sgn}}}

\DeclareMathOperator{\init}{\ensuremath{\mathrm{init}}}

\DeclareMathOperator{\hi}{\ensuremath{\mathrm{hi}}}

\newtheorem{lemma}{Lemma}[section]

\newtheorem{assumption}{Assumption}[section]
\newtheorem{definition}{Definition}[section]
\newtheorem{theorem}{Theorem}[section]
\newtheorem{remark}{Remark}[section]

\newcommand{\beq}{\begin{equation}}
\newcommand{\eeq}{\end{equation}}
\newcommand{\beqa}{\begin{eqnarray}}
\newcommand{\eeqa}{\end{eqnarray}}
\newcommand{\beqas}{\begin{eqnarray*}}
\newcommand{\eeqas}{\end{eqnarray*}}
\newcommand{\ba}{\begin{array}}
\newcommand{\ea}{\end{array}}
\newcommand{\bi}{\begin{statementize}}
\newcommand{\ei}{\end{statementize}}

\def\bdelta{{\delta}}
\def\bT{{\mathbb{T}}}
\def\cT{{\cal T}}

\def\fh{{f_{\rm hi}}}
\def\fl{{f_{\rm low}}}

\def\phih{{\phi_{\rm hi}}}
\def\phil{{\phi_{\rm low}}}
\def\hd{{\hat d}}
\def\nn{{\nonumber}}
\def\tc{{\tilde c}}
\def\td{{\tilde d}}
\def\tf{{\tilde f}}
\def\tdc{{\tilde \delta_c}}
\def\tdf{{\tilde \delta_f}}

\title{A Newton-CG based barrier-augmented Lagrangian method for general nonconvex conic optimization}

\author{
	Chuan He\thanks{Department of Mathematics, Link\"oping University, Sweden (email: {\tt chuan.he@liu.se}).} \\
	\and
	Heng Huang\thanks{Department of Computer Science, University of Maryland, USA  (email: \texttt{heng@umd.edu}). The work of this author was partially supported by NSF Award IIS-2211492.}\\
	\and
	Zhaosong Lu\thanks{Department of Industrial and Systems Engineering, University of Minnesota, USA (email: {\tt zhaosong@umn.edu}). The work of this author was partially supported by NSF Award IIS-2211491.}
}

\date{April 2, 2023 (Revised: March 4, 2024; August 8, 2024; August 25, 2024)}

\begin{document}
\maketitle
\begin{abstract}
In this paper we consider finding an approximate second-order stationary point (SOSP) of  general nonconvex conic optimization that minimizes a twice differentiable function 
subject to nonlinear equality constraints and also a convex conic constraint. 
In particular, we propose a Newton-conjugate gradient (Newton-CG) based barrier-augmented Lagrangian method for finding an approximate SOSP of this problem. 
Under some mild assumptions, we show that our method enjoys a total inner iteration complexity of $\widetilde{\cO}(\epsilon^{-11/2})$ and an operation complexity of $\widetilde{\cO}(\epsilon^{-11/2}\min\{n,\epsilon^{-5/4}\})$ for finding an $(\epsilon,\sqrt{\epsilon})$-SOSP of general nonconvex conic optimization with high probability. Moreover, under a constraint qualification, these complexity bounds are improved to $\widetilde{\cO}(\epsilon^{-7/2})$ and $\widetilde{\cO}(\epsilon^{-7/2}\min\{n,\epsilon^{-3/4}\})$, respectively. To the best of our knowledge, this is the first study on the complexity of finding an approximate SOSP of general nonconvex conic optimization. Preliminary numerical results are presented to demonstrate superiority of the proposed method over first-order methods in terms of solution quality.
\end{abstract}

\medskip

	
\section{Introduction} \label{intro}
 In this paper we consider the following general nonconvex conic optimization problem:
\begin{equation}\label{model:equa-cnstr}
\min_{x}\{f(x):c(x)=0,\ x\in\cK\},	
\end{equation}
where $\cK\subseteq\bR^{n}$ is a closed and pointed convex cone with a nonempty interior, and $f:\bR^n\to\bR$ and $c:\bR^n\to\bR^{m}$ are continuous in $\cK$ and twice continuously differentiable in the interior of $\cK$. Assume that problem~\eqref{model:equa-cnstr} has at least one optimal solution. 
Our goal is to propose an implementable method with complexity guarantees for finding an approximate second-order stationary point (SOSP) of \eqref{model:equa-cnstr} that will be introduced in Section \ref{sec:opt}.

In recent years, there has been considerable research on designing algorithms with complexity guarantees for finding an approximate SOSP of nonconvex optimization problems. In particular, numerous algorithms were developed for nonconvex unconstrained optimization, such as cubic regularized Newton methods \cite{AABHM17,CD19,CGT11ARC,NP06cubic}, trust-region methods \cite{CRRW21trust,CRS16trust,MR17trust},  quadratic regularization method  \cite{BM17QC}, accelerated gradient method \cite{CDHS17,CDHS18},  second-order line-search method \cite{RW18}, Newton-conjugate gradient (Newton-CG) method \cite{ROW20}, and gradient-based methods with random perturbations \cite{AL17,JNJ18AGD,XJY17AGD}. In addition, several methods with complexity guarantees have also been proposed for nonconvex optimization with relatively simple constraints. For example, interior-point method \cite{BCY2015}, log-barrier method \cite{OW21}, and projected gradient descent method \cite{XW21proj} were proposed for nonconvex optimization with sign constraints. Besides, the interior-point method \cite{BCY2015} was generalized in \cite{HLY19} for nonconvex optimization with sign constraints and additional linear equality constraints. Also, a projected gradient descent method with random perturbations was proposed in \cite{LRYHH20} for nonconvex optimization with linear inequality constraints. Iteration complexity of these methods has been established for finding an approximate SOSP. Besides, operation complexity in terms of the total number of fundamental operations has been studied for the methods \cite{AABHM17,AL17,CD19,CDHS17,CDHS18,CRRW21trust,JNJ18AGD,ROW20,RW18,XJY17AGD}.

Several methods, including trust-region methods \cite{BSS87,CLY02}, sequential quadratic programming method \cite{BL95}, two-phase method \cite{BGMST16,CGT19eq,CM19poly}, penalty method \cite{GEB22cpt}, and augmented Lagrangian (AL) methods \cite{AHRP17,BHR18,HL21al,S19iAL,XW21PAL}, were proposed for finding an approximate SOSP of equality constrained optimization:
\begin{equation}\label{eq-constr}
\min_{x}\{f(x):c(x)=0\},	
\end{equation} 
which is 
special case of \eqref{model:equa-cnstr} with $\cK=\bR^n$. Moreover, total inner iteration complexity and  operation complexity, which are respectively measured by the total number of iterations of the Newton-CG method in \cite{ROW20} and  the total number of gradient evaluations and matrix-vector products performed in  the method, were established in \cite{HL21al,XW21PAL}  for finding an $(\epsilon,\sqrt{\epsilon})$-SOSP $x$ of \eqref{eq-constr} which together with some $\lambda\in\bR^m$ satisfies 
\begin{equation*}
\begin{array}{l}
\|c(x)\|\le\epsilon,\ \|\nabla f(x) + \nabla c(x)\lambda\|\le\epsilon, \\ [8pt]
d^T(\nabla^2f(x)+\sum_{i=1}^m\lambda_i\nabla^2 c_i(x))d \ge - \sqrt{\epsilon} \|d\|^2, \quad  \forall  d \in \{d:\nabla c(x)^Td=0\},
\end{array}    
\end{equation*}
where $\nabla c$ denotes the transpose of the Jacobian of $c$. Specifically, under some suitable assumptions, including a generalized linear independence constraint qualification (GLICQ), the  AL method \cite{XW21PAL} enjoys a total inner iteration complexity of $\widetilde{\cO}(\epsilon^{-11/2})$  and an operation complexity $\widetilde{\cO}(\epsilon^{-11/2}\min\{n,\epsilon^{-3/4}\})$,\footnote{In fact, a total inner iteration complexity of $\widetilde{\cO}(\epsilon^{-7})$  and an operation complexity $\widetilde{\cO}(\epsilon^{-7}\min\{n,\epsilon^{-1}\})$ were established in \cite{XW21PAL} for finding an $(\epsilon,\epsilon)$-SOSP of problem \eqref{model:equa-cnstr} with high probability; see  \cite[Theorem~4(ii), Corollary~3(ii), Theorem~5]{XW21PAL}.  Nevertheless, they can be easily modified to obtain the aforementioned complexity for finding an $(\epsilon,\sqrt{\epsilon})$-SOSP of \eqref{model:equa-cnstr} with high probability.} while the AL method \cite{HL21al} achieves a total inner iteration complexity of $\widetilde{\cO}(\epsilon^{-7/2})$ and an operation complexity of $\widetilde{\cO}(\epsilon^{-7/2}\min\{n,\epsilon^{-3/4}\})$ for finding an $(\epsilon,\sqrt{\epsilon})$-SOSP of problem \eqref{eq-constr} with high probability. On the other hand, when the GLICQ does not hold, the AL method \cite{HL21al} has a total inner iteration complexity of $\widetilde{\cO}(\epsilon^{-11/2})$ and an operation complexity of $\widetilde{\cO}(\epsilon^{-11/2}\min\{n,\epsilon^{-5/4}\})$. Besides, it shall be mentioned that Newton-CG based AL methods were developed for efficiently solving a variety of convex optimization problems (e.g., see \cite{YST15sdpnal,ZST10NCGAL}), though their complexities remain unknown.



In addition, a Newton-CG based barrier method was recently proposed in  \cite{HL21bar} for finding an approximate SOSP of a class of nonconvex conic optimization of the form 
\begin{equation}\label{linear-cone}
\min_{x}\{f(x):Ax-b=0,\ x\in\cK\}
\end{equation}
for some $A\in\bR^{m\times n}$ and $b\in\bR^m$, which is a special case of \eqref{model:equa-cnstr}.  Iteration and operation complexity of this method were established in \cite{HL21bar} for finding an $(\epsilon,\sqrt{\epsilon})$-SOSP $x$ of \eqref{linear-cone} which together with some $\lambda\in\bR^m$ satisfies
\[
\begin{array}{l}
Ax=b,\ x\in\rmint\cK,\ \nabla f(x)+A^T\lambda\in\cK^*,\ \|\nabla f(x)+A^T\lambda\|_x^*\le\epsilon,\\[8pt]
d^T\nabla^2 B(x)^{-1/2}\nabla^2 f(x)\nabla^2 B(x)^{-1/2}d\ge -\sqrt{\epsilon}\|d\|^2,\quad \forall d\in\{d:A\nabla^2 B(x)^{-1/2}d=0\},
\end{array}    
\]
where $\rmint\cK$ and $\cK^*$ are respectively the interior and dual cone of $\cK$, $B$ is a logarithmically homogeneous self-concordant barrier function for $\cK$, and $\|\cdot\|_x^*$ is a local norm induced by $B$ at $x$ (see Section \ref{sec:not-pre} for details). Under some suitable assumptions,  this method achieves an iteration complexity of $\cO(\epsilon^{-3/2})$ and an operation complexity\footnote{The operation complexity of the barrier method \cite{HL21bar} is measured by the amount of fundamental operations consisting of matrix-vector products, matrix multiplications, Cholesky factorizations, and backward or forward substitutions to a triangular linear system.} of $\widetilde{\cO}(\epsilon^{-3/2}\min\{n,\epsilon^{-1/4}\})$ for finding an $(\epsilon,\sqrt{\epsilon})$-SOSP with high probability.  Besides, a Hessian barrier algorithm was proposed in \cite{DS21hba} for finding an approximate SOSP of problem \eqref{linear-cone}. Given that this algorithm requires solving a cubic regularized subproblem exactly per iteration, {it is generally expensive to implement.}



It shall also be mentioned that finding an approximate {\it first-order stationary point} of  \eqref{model:equa-cnstr} with $\cK=\bR^n_+$ was extensively studied in the literature (e.g., \cite{AT02,AO17BAL,AT19BAL,CGT13,CGT97BAL,DDAM08,GPSY99BAL,KB18BAL,
MP03BAL,VS99,WB06,WJJO06}). Notably, a hybrid approach by combining barrier and AL methods was commonly used in \cite{AT02,AO17BAL,AT19BAL,CGT97BAL,DDAM08,GPSY99BAL,KB18BAL,MP03BAL}). However,  
finding an approximate SOSP of \eqref{model:equa-cnstr} by such a hybrid approach has not been considered, even for \eqref{model:equa-cnstr} with $\cK=\bR^n_+$. Inspired by these and \cite{HL21bar,HL21al}, in this paper we propose a Newton-CG based barrier-AL method for finding an approximate SOSP of problem~\eqref{model:equa-cnstr} with high probability. Our main contributions are as follows.

\begin{itemize}
\item We study first- and second-order optimality conditions for problem \eqref{model:equa-cnstr} and introduce an approximate counterpart of them.

\item We propose an implementable Newton-CG based barrier-AL method for finding an approximate SOSP of \eqref{model:equa-cnstr}, whose main operations consist of Cholesky factorizations and other fundamental operations including matrix-vector products and backward or forward substitutions to a triangular linear system.\footnote{The arithmetic complexity of Cholesky factorizations for a positive definite matrix is $\cO(n^3)$ in general, while the arithmetic complexity of matrix-vector products and backward or forward substitutions is at most $\cO(n^2)$, where $n$ is the number of rows of the matrix. 
}

\item We show that under some mild assumptions, our proposed method has a total inner iteration complexity of $\widetilde{\cO}(\epsilon^{-11/2})$ and an operation complexity of $\widetilde{\cO}(\epsilon^{-11/2}\min\{n,\epsilon^{-5/4}\})$ for finding an $(\epsilon,\sqrt{\epsilon})$-SOSP of \eqref{model:equa-cnstr} with high probability. Furthermore, under a constraint qualification, we show that our method achieves an improved total inner iteration complexity of $\widetilde{\cO}(\epsilon^{-7/2})$ and an improved operation complexity of $\widetilde{\cO}(\epsilon^{-7/2}\min\{n,\epsilon^{-3/4}\})$.\footnote{It shall be mentioned that the total numbers of Cholesky factorizations are only $\widetilde{\cO}(\epsilon^{-7/2})$ and $\widetilde{\cO}(\epsilon^{-11/2})$ respectively for the case where constraint qualification holds or not. See Subsections~\ref{subsec:total-complexity} and \ref{sec:AL-modified} for details.} To the best of our knowledge, there was no complexity result for finding an approximate SOSP of problem \eqref{model:equa-cnstr} in the literature before.

\end{itemize}

The rest of this paper is organized as follows. In Section~\ref{sec:not-pre}, we introduce some notation. In Section~\ref{sec:opt}, we study optimality conditions of problem~\eqref{model:equa-cnstr} and introduce an inexact counterpart of them. In Section~\ref{sec:sbpb-solver}, we propose a preconditioned Newton-CG method for solving a barrier problem and study its complexity. We then propose a Newton-CG based barrier-AL method for \eqref{model:equa-cnstr} and study its complexity in Section~\ref{sec:AL-method}. We present in Section~\ref{sec:nr} some preliminary numerical results for the proposed method.  In Section \ref{sec:proof}, we present the proofs of the main results. Finally, we make some concluding remarks in Section~\ref{sec:cr}.

\section{Notation and preliminaries}\label{sec:not-pre}

Throughout this paper,  we let $\bR^n$ denote the $n$-dimensional Euclidean space. The symbol $\|\cdot\|$ stands for the Euclidean norm of a vector or the spectral norm of a matrix. 
The identity matrix is denoted by $I$. We denote by $\lambda_{\min}(H)$ the minimum eigenvalue of a real symmetric matrix $H$. For any two real symmetric matrices $M_1$ and $M_2$, $M_1 \preceq M_2$ means that $M_2-M_1$ is positive semidefinite. For any positive semidefinite matrix $M$, $M^{1/2}$ denotes a positive semidefinite matrix such that $M=M^{1/2}M^{1/2}$. For the closed convex cone $\cK$, its interior and dual cone are respectively denoted by $\rmint\cK$ and $\cK^*$. For any $x\in\cK$, the normal cone and tangent cone  of $\cK$ at $x$ are denoted by $\cN_{\cK}(x)$ and $\cT_{\cK}(x)$, respectively. The Euclidean ball centered at the origin with radius $R\ge0$ is denoted by $\mathcal{B}_R:=\{x:\|x\|\le R\}$, and  we use $\Pi_{\mathcal{B}_R}(v)$ to denote the Euclidean projection of a vector $v$ onto $\mathcal{B}_R$. For a given finite set $\cA$, we let $|\cA|$ denote its cardinality. For any $s\in \mathbb{R}$, we let ${\rm sgn}(s)$ be $1$ if $s \ge 0$ and let it be $-1$ otherwise. In addition, $\widetilde{\cO}(\cdot)$ represents $\cO(\cdot)$ with logarithmic terms omitted.
 
Logarithmically homogeneous self-concordant (LHSC) barrier function is a key ingredient in the development of interior-point methods for convex programming (see the monograph \cite{NN94}). It will also play a crucial role in the design and analysis of Newton-CG based barrier-AL method for solving problem \eqref{model:equa-cnstr}. \emph{Throughout this paper, we assume that the cone $\cK$ is equipped with a $\vartheta$-logarithmically homogeneous self-concordant ($\vartheta$-LHSC) barrier function $B$ for some $\vartheta\ge 1$.} That is, $B:\rmint\cK\to\bR$ satisfies the following conditions:
\begin{enumerate}[{\rm (i)}]
\item $B$ is convex and three times continuously differentiable in $\rmint \cK$, and moreover, $|\psi^{\prime\prime\prime}(0)|\le 2(\psi^{\prime\prime}(0))^{3/2}$ holds for all $x\in\rmint\cK$ and $u\in\bR^n$, where $\psi(t)=B(x+tu)$;
\item $B$ is a barrier function for $\cK$, that is, $B(x)$ goes to infinity as $x$ approaches the boundary of $\cK$;
\item $B$ is logarithmically homogeneous, that is, $B(tx)=B(x)-\vartheta\ln t$ holds for all $x\in\rmint\cK$ and $t>0$. 
\end{enumerate}
For any $x\in\rmint\cK$, the function $B$ induces the following local norms:
\begin{eqnarray}
\|v\|_x&:=& \left(v^T\nabla^2 B(x)v\right)^{1/2}, \quad \forall v\in\bR^{n}, \nonumber\\ [4pt]
\|v\|_x^*&:=& \left(v^T\nabla^2 B(x)^{-1}v\right)^{1/2}, \quad \forall v\in\bR^{n},\nonumber\\ [4pt]
\|M\|^*_x &:=& \max\limits_{\|v\|_x \le 1} \|Mv\|^*_x, \quad \forall M\in \bR^{n\times n}. \label{M-norm}
\end{eqnarray} 
In addition, $\nabla^2 B(x)^{-1}$ is well-defined only in $\rmint\cK$ but undefined on the boundary of $\cK$. To capture the behavior of $\nabla^2 B(x)^{-1}$ as $x$ approaches the boundary of $\cK$, the concept of {\it the limiting inverse of the Hessian of $B$} was recently introduced in \cite{HL21bar}, which can be viewed as a generalization of $[\nabla^2 B]^{-1}$. Specifically,  the limiting inverse of the Hessian of $B$ is defined as follows:
\begin{equation*}
\nabla^{-2} B(x) := \left\{M : M=\lim\limits_{k\to\infty}\nabla^2 B(x^k)^{-1}\text{ for some } \{x^k\}\subset\rmint\cK\text{ with }x^k\to x\text{ as }k\to\infty\right\},\quad \forall x\in\cK.
\end{equation*}
As established in \cite[Theorem~1]{HL21bar}, the inverse of $\nabla^2 B(x)$ is bounded in any nonempty bounded subset of $\rmint\cK$. Consequently, $\nabla^{-2} B(x)\neq \emptyset$ for all $x\in\cK$. Moreover, the following property holds for $\nabla^{-2} B$, whose proof can be found in \cite[Theorem~2]{HL21bar}.

\begin{lemma}\label{lem:ginvH}
For any $x\in\cK$, it holds that 
$\{x+M^{1/2}d:\|d\|<1\}\subseteq\cK$ for all $M\in\nabla^{-2} B(x)$.
\end{lemma}

\section{Optimality conditions}\label{sec:opt}

Classical first- and second-order optimality conditions for nonlinear optimization can be specialized to problem~\eqref{model:equa-cnstr} (e.g., see \cite[Theorems 3.38 and 3.46]{R11nopt}). However, the inexact counterparts of them are not suitable for the design and analysis of a barrier-AL method for solving \eqref{model:equa-cnstr}. In this section we study some alternative first- and second-order optimality conditions for \eqref{model:equa-cnstr} and also introduce an inexact counterpart of them.

Suppose that $x^*$ is a local minimizer of problem~\eqref{model:equa-cnstr}. To derive optimality conditions, one typically needs to impose a constraint qualification (CQ) for $x^*$. The Robinson's CQ, $\{\nabla c(x^*)^Td:d\in \cT_{\cK}(x^*)\}=\bR^m$, is a natural and general one (e.g., see \cite[Section 3.3.2]{R11nopt}). However, verification of Robinson's CQ may not be easy for a general cone $\cK$. Thus, we instead consider a more easily verifiable CQ that $M^{1/2}\nabla c(x^*)$ has full column rank for some $M\in\nabla^{-2}B(x^*)$, which turns out to be stronger than Robinson's CQ. Indeed, suppose that such a CQ holds at $x^*$ for some $M\in\nabla^{-2}B(x^*)$. It then follows from Lemma~\ref{lem:ginvH} that $\{M^{1/2}\td: \|\td\|<1\} \subseteq \cT_{\cK}(x^*)$ and hence $\{M^{1/2}\td: \td\in\bR^n\} \subseteq \cT_{\cK}(x^*)$.  By this and the full column rank of $M^{1/2}\nabla c(x^*)$, one has 
\[
\{\nabla c(x^*)^Td:d\in \cT_{\cK}(x^*)\} \supseteq \{\nabla c(x^*)^TM^{1/2}\td: \td\in\bR^n\} = \bR^m,
\]
and hence Robinson's CQ holds at $x^*$.

We are now ready to establish some first- and second-order optimality conditions for problem~\eqref{model:equa-cnstr} under the aforementioned CQ, whose proof is relegated to Section~\ref{sec:proof-sec3}.

\begin{theorem}[{{\bf first- and second-order optimality conditions}}]\label{thm:1stopt}
Let $x^*$ be a local minimizer of problem~\eqref{model:equa-cnstr}. Suppose that $f$ is twice continuously differentiable at $x^*$ and $M^{1/2}\nabla c(x^*)$ has full column rank for some $M\in\nabla^{-2}B(x^*)$. Then there exists a Lagrangian multiplier $\lambda^*\in\bR^m$ such that 
\begin{eqnarray}
&&\nabla f(x^*) + \nabla c(x^*)\lambda^* \in\cK^*,\label{1stopt-cond-1}\\
&&M^{1/2}(\nabla f(x^*) + \nabla c(x^*)\lambda^*) = 0,\label{1stopt-cond-2}
\end{eqnarray}
and additionally,
\begin{equation}\label{2ndopt-cond}
d^TM^{1/2}\left(\nabla^2 f(x^*)+\sum_{i=1}^m\lambda_i^*\nabla^2 c_i(x^*)\right)M^{1/2}d \ge 0,\quad \forall d\in\{d:\nabla c(x^*)^TM^{1/2}d=0\}. 
\end{equation}
\end{theorem}

\begin{remark}
The relations  \eqref{1stopt-cond-1} and \eqref{1stopt-cond-2} are the first-order optimality conditions of problem~\eqref{model:equa-cnstr}, which are actually equivalent to the classical optimality condition $\nabla f(x^*) + \nabla c(x^*)\lambda^* \in -\cN_{\cK}(x^*)$ (see \cite[Proposition~1]{HL21bar}). 
\end{remark}

Notice that it is generally impossible to find a point exactly satisfying the above first- and second-order optimality conditions. We are instead interested in finding a point satisfying their approximate counterparts. To this end, we next introduce the definition of an approximate first-order stationary point (FOSP) and second-order stationary point (SOSP) of problem \eqref{model:equa-cnstr}.

\begin{definition}[{{\bf $\epsilon_1$-first-order stationary point}}]\label{def-FOSP}
For any $\epsilon_1>0$, a point $x$ is called an $\epsilon_1$-first-order stationary point {\rm($\epsilon_1$-FOSP)} of problem~\eqref{model:equa-cnstr} if it, together with some $\lambda\in\bR^m$, satisfies
\begin{eqnarray}
&&\|c(x)\|\le\epsilon_1,\ x\in\rmint\cK,\label{apx-1st-1}\\
&&\nabla f(x) + \nabla c(x)\lambda \in\cK^*,\label{apx-1st-2}\\
&&\|\nabla f(x) + \nabla c(x)\lambda\|_{x}^*\le\epsilon_1.\label{apx-1st-3}
\end{eqnarray}
\end{definition}

\begin{definition}[{{\bf $(\epsilon_1,\epsilon_2)$-second-order stationary point}}]\label{def-SOSP}
For any $\epsilon_1,\epsilon_2>0$, a point $x$ is called an $(\epsilon_1,\epsilon_2)$-second-order stationary point {\rm($(\epsilon_1,\epsilon_2)$-SOSP)} of problem~\eqref{model:equa-cnstr} if it, together with some $\lambda\in\bR^m$, satisfies \eqref{apx-1st-1}-\eqref{apx-1st-3} and additionally
\begin{equation}\label{apx-2nd}
d^T\nabla^2 B(x)^{-1/2}\left(\nabla^2 f(x) + \sum_{i=1}^m\lambda_i\nabla^2 c_i(x)\right)\nabla^2 B(x)^{-1/2}d\ge-\epsilon_2\|d\|^2,\quad \forall d\in\cC(x),
\end{equation}
where $\cC(\cdot)$ is defined as
\begin{equation}\label{c-cone}
\cC(x):=\{d:\nabla c(x)^T\nabla^2 B(x)^{-1/2}d=0\}.    
\end{equation}
\end{definition}

\begin{remark}
Notice that if the pair $(x,\lambda)$ satisfies \eqref{apx-1st-3} and \eqref{apx-2nd}, then it nearly satisfies \eqref{1stopt-cond-2} and \eqref{2ndopt-cond} with $(x^*,\lambda^*)$ replaced by $(x,\lambda)$. Thus, \eqref{apx-1st-3} and \eqref{apx-2nd} can be viewed as inexact counterparts of \eqref{1stopt-cond-2} and \eqref{2ndopt-cond}. Moreover, the above definitions of $\epsilon_1$-FOSP and $(\epsilon_1,\epsilon_2)$-SOSP are reduced to the ones introduced in \cite{HL21bar} for the case where $c$ is affine.
\end{remark}


\section{A preconditioned Newton-CG method for barrier problems}\label{sec:sbpb-solver}

In this section we propose a preconditioned Newton-CG method in Algorithm~\ref{alg:NCG}, which is a modification of the Newton-CG based barrier method \cite[Algorithm~2]{HL21bar}, for finding an approximate SOSP of the barrier problem
\begin{equation}\label{b-subprob}
\min_{x}\ \{\phi_\mu(x):=F(x)+\mu B(x)\},
\end{equation}
where $F:\bR^n\to\bR$ is twice continuously differentiable in $\rmint\cK$ and $\mu>0$ is a given barrier parameter. Specifically, the proposed method finds an $(\epsilon_g,\epsilon_H)$-SOSP $x$ of problem \eqref{b-subprob} that satisfies
\begin{equation}\label{apx-1st2nd-stat-bsbpb}
\|\nabla \phi_\mu(x)\|_x^*\le\epsilon_g,\quad \lambda_{\min}(\nabla^2B(x)^{-1/2}\nabla^2 \phi_\mu(x)\nabla^2B(x)^{-1/2})\ge-\epsilon_H
\end{equation}
for any prescribed tolerances $\epsilon_g,\epsilon_H\in(0,1)$. It will be used to solve the subproblems arising in the barrier-AL method later.

Our preconditioned Newton-CG method (Algorithm~\ref{alg:NCG}) consists of two main components. The first main component is a modified CG method, referred to as {\it capped CG method}, which was proposed in \cite[Algorithm~1]{ROW20} for solving a possibly indefinite linear system
\begin{equation}\label{indef-sys}
(H+2\varepsilon I)\hat{d}=-g,
\end{equation}
where $0\neq g\in\bR^n$, $\varepsilon>0$, and $H\in\bR^{n\times n}$ is a symmetric matrix. The capped CG method terminates within a finite number of iterations and returns either an approximate solution $\hat{d}$ to \eqref{indef-sys} satisfying $\|(H+2\varepsilon I)\hat{d}+g\|\le\hat{\zeta}\|g\|$ and $\hat{d}^TH\hat{d}\ge-\varepsilon\|\hat{d}\|^2$ for some $\hat{\zeta}\in(0,1)$ or a sufficiently negative curvature direction $\hat{d}$ of $H$ with $\hat{d}^TH\hat{d}<-\varepsilon\|\hat{d}\|^2$. The second main component is a minimum eigenvalue oracle. Given a symmetric matrix $H\in\bR^{n\times n}$ and $\varepsilon>0$, this oracle either produces a sufficiently negative curvature direction $v$ of $H$ with $\|v\|=1$ and  $v^THv\le-\varepsilon/2$ or certifies that $\lambda_{\min}(H)\ge-\varepsilon$ holds with high probability. For ease of reference, we present these two main components in Algorithms~\ref{alg:capped-CG} and \ref{pro:meo} in Appendices~\ref{appendix:capped-CG} and \ref{appendix:meo}, respectively. 

We are now ready to describe our preconditioned Newton-CG method (Algorithm~\ref{alg:NCG}) for solving \eqref{b-subprob}. At iteration $t$, if the first relation in \eqref{apx-1st2nd-stat-bsbpb} is not satisfied at the iterate $x^t$, the capped CG method (Algorithm~\ref{alg:capped-CG}) is invoked to find a descent direction for $\phi_\mu$ by solving the following damped preconditioned Newton system
\begin{equation*}
(M_t^T\nabla^2 \phi_\mu(x^t)M_t + 2\epsilon_H I)\hat{d} = -M_t^T\nabla\phi_\mu(x^t),
\end{equation*}
where $M_t$ is a matrix such that
\begin{equation}\label{Mt}
\nabla^2 B(x^t)^{-1}=M_tM_t^T.
\end{equation}
A line search along this descent direction is then performed to result in a reduction on $\phi_\mu$.  Otherwise, the minimum eigenvalue oracle (Algorithm~\ref{pro:meo}) is invoked. This oracle either produces a sufficiently negative curvature direction of $M_t^T\nabla^2\phi_\mu(x^t)M_t$ along which a line search is performed to result in a reduction on $\phi_\mu$, or certifies that the iterate $x^t$ also satisfies the second relation in \eqref{apx-1st2nd-stat-bsbpb} with high probability and terminates the preconditioned Newton-CG method. The detailed description of our preconditioned Newton-CG method is presented in Algorithm~\ref{alg:NCG}.

\begin{algorithm}[H]
\caption{A preconditioned Newton-CG method for problem~\eqref{b-subprob}}
\label{alg:NCG}
{\footnotesize
\begin{algorithmic}
\State \noindent\textbf{Input}: tolerances $\epsilon_g,\epsilon_H\in(0,1)$, backtracking ratio $\theta\in(0,1)$, starting point $u^0\in\rmint\cK$, CG-accuracy parameter $\zeta\in(0,1)$, maximum step length $\beta\in[\epsilon_H,1)$, line-search parameter $\eta\in(0,1)$, probability parameter $\delta\in(0,1)$;
\State Set $x^0=u^0$;
\For{$t=0,1,2,\ldots$}
\If{$\|\nabla \phi_\mu(x^t)\|_{x^t}^*>\epsilon_g$}
\State Call Algorithm~\ref{alg:capped-CG} (see Appendix~\ref{appendix:capped-CG}) with $H=M_t^T\nabla^2 \phi_\mu(x^t)M_t$, $\varepsilon=\epsilon_H$, $g=M_t^T\nabla \phi_\mu(x^t)$, accuracy parameter $\zeta$,
\State and bound $U=0$ to obtain outputs $\hat{d}^t$, d$\_$type, where $M_t$ is given in \eqref{Mt};
\If{d$\_$type=NC}
\begin{equation}\label{dk-nc}
d^t\leftarrow -\sgn((\hat{d}^t)^TM_t^T\nabla \phi_\mu(x^t))\min\left\{\frac{|(\hat{d}^t)^TM_t^T\nabla^2 \phi_\mu(x^t)M_t \hat{d}^t|}{\|\hat{d}^t\|^3},\frac{\beta}{\|\hat{d}^t\|}\right\}\hat{d}^t;
\end{equation}
\Else\ \{d$\_$type=SOL\}
\begin{equation}\label{dk-sol}
d^t\leftarrow \min\left\{1,\frac{\beta}{\|\hat{d}^t\|}\right\}\hat{d}^t;
\end{equation}
\EndIf
\State Go to {\bf Line Search};
\Else
\State Call Algorithm~\ref{pro:meo} (see Appendix~\ref{appendix:meo}) with $H=M_t^T\nabla^2 \phi_\mu(x^t)M_t$, $\varepsilon=\epsilon_H$, and probability parameter $\delta$;
\If{Algorithm~\ref{pro:meo} certifies that $\lambda_{\min}(M_t^T\nabla^2 \phi_\mu(x^t)M_t)\ge-\epsilon_H$}
\State Output $x^t$ and terminate;
\Else\ \{Sufficiently negative curvature direction $v$ returned by Algorithm~\ref{pro:meo}\}
\State Set d$\_$type=NC and
\begin{equation}\label{dk-2nd-nc}
d^t\leftarrow -\sgn(v^TM_t^T\nabla \phi_\mu(x^t))\min\{|v^TM_t^T\nabla^2 \phi_\mu(x^t)M_tv|,\beta\}v;
\end{equation}
\State Go to {\bf Line Search};
\EndIf
\EndIf
\State{\bf Line Search:}
\If{d$\_$type=SOL}
\State Find $\alpha_t=\theta^{j_t}$, where $j_t$ is the smallest nonnegative integer $j$ such that
\begin{equation}\label{ls-sol}
\phi_\mu(x^t+\theta^jM_td^t)<\phi_\mu(x^t)-\eta\epsilon_H\theta^{2j}\|d^t\|^2;
\end{equation}
\Else\ \{d$\_$type=NC\}
\State Find $\alpha_t=\theta^{j_t}$, where $j_t$ is the smallest nonnegative integer $j$ such that
\begin{equation}\label{ls-nc}
\phi_\mu(x^t+\theta^jM_td^t)<\phi_\mu(x^t)-\eta\theta^{2j}\|d^t\|^3/2;
\end{equation}
\EndIf
\State $x^{t+1}=x^t+\alpha_tM_td^t$;
\EndFor
\end{algorithmic}
}
\end{algorithm}

\subsection{Iteration and operation complexity of Algorithm~\ref{alg:NCG}}
\label{complex-alg1}

In this subsection we study iteration and operation complexity of Algorithm~\ref{alg:NCG}. To proceed, we make the following assumptions on problem~\eqref{b-subprob}.

\begin{assumption}\label{asp:NCG-cmplxity}
\begin{enumerate}[{\rm (a)}]
\item There exists a finite $\phi_{\low}$ such that 
\begin{eqnarray}
&&\phi_\mu(x)\ge\phi_{\low},\quad \forall x\in\rmint\cK,\label{lwbd-b-subpb}\\
&&\cS=\{x\in\rmint\cK:\phi_\mu(x)\le\phi_\mu(u^0)\}\text{ is bounded},\label{b-subpb-set}
\end{eqnarray}
where $u^0\in\rmint\cK$ is the initial point of Algorithm~\ref{alg:NCG} and $\phi_\mu$ is given in \eqref{b-subprob}.
\item There exists $L_H^F>0$ such that 
\begin{equation*}\label{F-Hess-Lip}
\|\nabla^2 F(y)-\nabla^2 F(x)\|_x^*\le L_{H}^F\|y-x\|_x,\quad \forall x,y\in\Omega\text{ with }\|y-x\|_x\le\beta,
\end{equation*}
where $\Omega \subset \rmint\cK$ is an open bounded convex neighborhood of $\cS$ and $\beta\in(0,1)$ is an input of Algorithm~\ref{alg:NCG}.
\item The quantities $U_g^F$ and $U_H^F$ are finite, where
\begin{equation}\label{Fbd-b-subpb}
U_g^F:=\sup_{x\in\cS}\|\nabla F(x)\|_x^*,\quad U_H^F:=\sup_{x\in\cS}\|\nabla^2 F(x)\|_x^*.
\end{equation}
\end{enumerate}
\end{assumption}

Before establishing operation complexity of Algorithm~\ref{alg:NCG}, let us make some observations on its fundamental operations. Firstly, at iteration $t$, the main effort of Algorithm~\ref{alg:NCG} is on the execution of Algorithm~\ref{alg:capped-CG} or \ref{pro:meo} with $H=M_t^T\nabla^2 \phi_\mu(x^t)M_t$. Secondly, the main computational cost of Algorithms~\ref{alg:capped-CG} and \ref{pro:meo} per iteration is on the product of $H$ and a vector $v$. Consequently, it suffices to focus on computing $Hv$. Indeed, notice from \eqref{b-subprob} and \eqref{Mt} that
\[
Hv=M_t^T\nabla^2 \phi_\mu(x^t)M_tv=M^T_t\nabla^2 F(x^t)M_tv+\mu v.
\]
Thus, computing $Hv$ consists of one Hessian-vector product of $F$ and two matrix-vector products involving $M_t$ and $M_t^T$, respectively. We next discuss how to efficiently compute the product of $M_t$ or $M_t^T$ and a vector.

\begin{itemize}
\item When $\cK$ is the nonnegative orthant, its associated barrier function is $B(x)=-\sum_{i=1}^n\ln x_i$. Notice that $\nabla^2 B(x)$ is a diagonal matrix and so is $M_t$. As a result, the operation cost for computing the product of $M_t$ or $M_t^T$ and a vector is $\cO(n)$, which is typically cheaper than the Hessian-vector product of $F$.
\item When $\cK$ is a general cone,  directly computing $M_t$ may be too expensive. In view of  $\nabla^2 B(x^t)=M_t^{-T}M_t^{-1}$ (see \eqref{Mt}), one can instead choose $M_t^{-T}$ as the Cholesky factor of $\nabla^2 B(x^t)$, which is computed only once in each iteration of Algorithm~\ref{alg:NCG}. Once $M_t^{-T}$ is available, the product of $M_t$ or $M_t^T$ and a vector can be computed by performing backward or forward substitution to a linear system with coefficient matrix $M_t^{-1}$ or $M_t^{-T}$. 
\end{itemize}
Based on the above discussion, we conclude that: (i) when $\cK$ is the nonnegative orthant, the fundamental operations of Algorithm~\ref{alg:NCG} consist only of the Hessian-vector products of $F$; (ii) when $\cK$ is a general cone, the fundamental operations of Algorithm~\ref{alg:NCG} consist of the Hessian-vector products of $F$, Cholesky factorizations of $\nabla^2 B$, and backward or forward substitutions to a triangular linear system.

The following theorem states the iteration and operation complexity of Algorithm~\ref{alg:NCG}, whose proof is deferred to Section~\ref{sec:proof-sec4}.

\begin{theorem}[{{\bf Complexity of Algorithm~\ref{alg:NCG}}}]\label{thm:NCG-iter-oper-cmplxity}
Suppose that Assumption~\ref{asp:NCG-cmplxity} holds. Let
\begin{equation}
T_1=\left\lceil\frac{\phi_{\hi}-\phi_{\low}}{\min\{c_{\rm sol},c_{\rm nc}\}}\max\{\epsilon_g^{-2}\epsilon_H,\epsilon_H^{-3}\}\right\rceil+\left\lceil\frac{\phih-\phil}{c_{\rm nc}}\epsilon_H^{-3}\right\rceil+1,\
T_2=\left\lceil\frac{\phih-\phil}{c_{\rm nc}}\epsilon_H^{-3}\right\rceil+1,\label{T1}
\end{equation}
where $\phi_{\hi}=\phi_\mu(u^0)$, $\phi_{\low}$ is given in \eqref{lwbd-b-subpb}, and
\begin{eqnarray}
&c_{\rm sol}=\eta\min\left\{\left[\frac{4(1-\beta)}{4+\zeta+\sqrt{(4+\zeta)^2+8[(1-\beta)L^F_H+\mu(2-\beta)/(1-\beta)]}}\right]^2,\left[\frac{\min\{6(1-\eta),2\}\theta}{L^F_H+\mu(2-\beta)/(1-\beta)^2}\right]^2\right\},\label{csol}\\
&c_{\rm nc} = \frac{\eta}{16} \min\left\{1,\left[\frac{\min\{3(1-\eta),1\}\theta}{L^F_H+\mu(2-\beta)/(1-\beta)^2}\right]^2\right\}.\label{cnc}
\end{eqnarray}
Then the following statements hold.
\begin{enumerate}[{\rm (i)}]
\item The total number of calls of Algorithm~\ref{pro:meo} in Algorithm~\ref{alg:NCG} is at most $T_2$.	
\item The total number of calls of Algorithm~\ref{alg:capped-CG} in Algorithm~\ref{alg:NCG} is at most $T_1$.
\item {\rm ({\bf iteration complexity})} Algorithm~\ref{alg:NCG} terminates in at most $T_1+T_2$ iterations with
\begin{equation}\label{NCG-iter}
T_1+T_2=\cO((\phih-\phil)(L_H^F)^2\max\{\epsilon_g^{-2}\epsilon_H,\epsilon_H^{-3}\}).
\end{equation}
Moreover, its output $x^t$ satisfies the first relation in \eqref{apx-1st2nd-stat-bsbpb} deterministically and 
the second relation in \eqref{apx-1st2nd-stat-bsbpb} with probability at least $1-\delta$ for 
some $0 \le t \le T_1+T_2$.
\item {\rm ({\bf operation complexity})} The total numbers of Cholesky factorizations and other fundamental operations 
 consisting of the Hessian-vector products of $F$ and backward or forward substitutions to a triangular linear system
 required by Algorithm~\ref{alg:NCG} are at most $T_1+T_2$ and
\begin{equation*}\label{NCG-oper}
\widetilde{\cO}((\phih-\phil)(L_H^F)^2\max\{\epsilon_g^{-2}\epsilon_H,\epsilon_H^{-3}\}\min\{n,(U_H^F/\epsilon_H)^{1/2}\}),
\end{equation*} 
respectively, where $U_H^F$ is given in \eqref{Fbd-b-subpb}.
\end{enumerate}
\end{theorem}

\section{A Newton-CG based barrier-AL method for problem \eqref{model:equa-cnstr}}
\label{sec:AL-method}
In this section we propose a Newton-CG based barrier-AL method for finding a stochastic $(\epsilon,\sqrt{\epsilon})$-SOSP of problem~\eqref{model:equa-cnstr} for any prescribed tolerance $\epsilon\in(0,1)$. 

 Recall that $B$ is the $\vartheta$-LHSC barrier function associated with $\cK$ for some $\vartheta\ge 1$.  We now make the following additional assumptions on problem~\eqref{model:equa-cnstr}.

\begin{assumption}\label{asp:lowbd-knownfeas}		
\begin{enumerate}[{\rm (a)}]
\item An $\epsilon/2$-approximately strictly feasible point $z_{\epsilon}$ of problem~\eqref{model:equa-cnstr}, namely satisfying $z_{\epsilon}\in\rmint\cK$ and $\|c(z_{\epsilon})\|\le\epsilon/2$, is known.
\item There exist constants $\fh,\fl\in\bR$ and  $\gamma, \bdelta_f, \bdelta_c>0$, independent of $\epsilon$, such that
\begin{eqnarray}
&&f(z_{\epsilon})+\mu B(z_{\epsilon})\le\fh,\quad \forall \mu\in(0,\mu_0], \label{ubd}\\
&&f(x)+\mu B(x)+\gamma\|c(x)\|^2/2\ge \fl,\quad \forall \mu\in(0,\mu_0], x\in\rmint\cK,\label{ineq:lower-bound-penalty}\\
&&\cS(\bdelta_f,\bdelta_c):=\bigcup_{\mu\in(0,\mu_0]}\{x\in\rmint\cK:f(x)+\mu B(x)\le \fh+\bdelta_f,\|c(x)\|\le 1+\bdelta_c\}\text{ is bounded,}\label{nearly-feas-level-set}
\end{eqnarray}
where $\mu_0=1/(2\vartheta^{1/2}+2)$ and $z_{\epsilon}$ is given in (a).
\item There exist $L_H^f, L_H^c>0$ and $\beta\in(0,1)$ such that
\begin{equation}\label{Hes-loc-Lip}
\begin{array}{l}
\|\nabla^2 f(y)-\nabla^2 f(x)\|_x^*\le L_H^f\|y-x\|_x,\quad  \forall x,y\in\Omega(\delta_f,\delta_c)\text{ with }\|y-x\|_x\le\beta,\\[5pt]
\|\nabla^2 c_i(y)-\nabla^2 c_i(x)\|_x^*\le L_H^c\|y-x\|_x,\quad \forall  x,y\in\Omega(\delta_f,\delta_c)\text{ with }\|y-x\|_x\le\beta,\ 1\le i\le m,
\end{array}    
\end{equation}
where $\Omega(\delta_f,\delta_c)\subset \rmint\cK$ is an open bounded convex neighborhood of $\cS(\delta_f,\delta_c)$.
\item The quantities $U_g^f$, $U_g^c$, $U_H^f$ and $U_H^c$ are finite, where
\begin{eqnarray}
&U_g^f=\sup_{x\in\Omega(\delta_{f},\delta_c)}\|\nabla f(x)\|_x^*,\quad U_g^c=\sup_{x\in\Omega(\delta_{f},\delta_c)}\max_{1\le i\le m}\|\nabla c_i(x)\|_x^*,\label{g-loc-upbd}\\
&U_H^f=\sup_{x\in\Omega(\delta_{f},\delta_c)}\|\nabla^2 f(x)\|_x^*,\quad U_H^c=\sup_{x\in\Omega(\delta_{f},\delta_c)}\max_{1\le i\le m}\|\nabla^2 c_i(x)\|_x^*.\label{H-loc-upbd}
\end{eqnarray}
\end{enumerate}
\end{assumption}

We next make some remarks about Assumption~\ref{asp:lowbd-knownfeas}.

\begin{remark}
\begin{enumerate}[{\rm (i)}]

\item A similar assumption as Assumption~\ref{asp:lowbd-knownfeas}(a) was  considered in the study of AL methods for nonconvex equality constrained optimization (e.g., see \cite{CGLY17,GY21,HL21al,LZ12,XW21PAL}). By imposing Assumption~\ref{asp:lowbd-knownfeas}(a), we restrict our study on problem~\eqref{model:equa-cnstr} for which an $\epsilon/2$-approximately strictly feasible point $z_{\epsilon}$ can be found by an inexpensive procedure. Assumption~\ref{asp:lowbd-knownfeas}(a) often holds in practice. For example,  when the constraints of \eqref{model:equa-cnstr} consist of sphere and nonnegative orthant constraints, a strictly feasible point is readily available.  Also, when $c$ is an affine mapping and $\cK$ is the nonnegative orthant,  a strictly feasible point of \eqref{model:equa-cnstr} can be found using interior point methods. In addition,  when the generalized LICQ condition $\lambda_{\min}(\nabla c(x)^T\nabla^2 B(x)^{-1}\nabla c(x))\ge\sigma^2$ (see Assumption~\ref{asp:LICQ} below) holds on a level set of $\|c(x)\|^2+\mu B(x)$ 
for $\mu=\sigma\epsilon/(2\vartheta^{1/2})$ and some constant $\sigma>0$ and $F(x)=\|c(x)\|^2$ satisfies Assumption~\ref{asp:NCG-cmplxity}, the point $z_{\epsilon}$ can be found by applying our preconditioned Newton-CG method (Algorithm~\ref{alg:NCG}) to the barrier problem $\min_x \|c(x)\|^2 +\mu B(x)$ to find $z_{\epsilon}$ satisfying $\|\nabla(\|c(z_{\epsilon})\|^2 +\mu B(z_{\epsilon}))\|_{z_{\epsilon}}^*\le\sigma\epsilon/2$. It can be verified that such  $z_{\epsilon}$ satisfies Assumption~\ref{asp:lowbd-knownfeas}(a).  As observed from Theorem~\ref{thm:NCG-iter-oper-cmplxity}, the resulting iteration and operation complexity for finding such $z_{\epsilon}$ are respectively $\cO(\epsilon^{-3/2})$ and $\widetilde{\cO}(\epsilon^{-3/2}\min\{n,\epsilon^{-1/4}\})$, which are negligible compared with those of our barrier-AL method (see Theorems~\ref{thm:total-iter-cmplxity} and \ref{thm:total-iter-cmplxity2} below).  
\item Assumption~\ref{asp:lowbd-knownfeas}(b) is mild. In particular, the assumption in \eqref{ubd} holds if $f(x)+\mu_0[B(x)]_+$ is bounded above for all $x\in\rmint\cK$ with $\|c(x)\|\le 1$. Besides, the function $f(x)+\mu B(x)+\gamma\|c(x)\|^2/2$ is a barrier-quadratic penalty function of problem~\eqref{model:equa-cnstr} and is typically bounded below on $\rmint\cK$. In addition, letting $z^0$ be an arbitrary point in $\rmint\cK$,
it can be shown that $\cS(\delta_f,\delta_c)\subseteq\cS_1\cup \cS_2$, where
\begin{equation*}
\begin{array}{l}
\cS_1=\left\{x\in\rmint\cK:f(x)\le f_{\hi}+\delta_f+\mu_0+\mu_0[B(z^0)]_+,B(x)\ge-1-[B(z^0)]_+,\|c(x)\|\le1+\delta_c\right\},\\[5pt]
\cS_2=\left\{x\in\rmint\cK:\frac{f(x)}{-B(x)}\le\frac{[f_{\hi}+\delta_f]_+}{1+[B(z^0)]_+}+\mu_0,B(x)\le-1-[B(z^0)]_+,\|c(x)\|\le1+\delta_c\right\},
\end{array}    
\end{equation*}
and $t_+=\max\{0,t\}$ for all $t\in\bR$.
Thus, the assumption in \eqref{nearly-feas-level-set} holds if $\cS_1$ and $\cS_2$ are bounded. The latter holds, for example, for the problem with $f(x)=\ell(x) +\sum_{i=1}^n x_i^p$, $B(x) = -\sum_{i=1}^n\ln x_i$ and $\cK=\bR^n_+$ studied in \cite{HLY19}, where $\ell:\bR^n\to\bR_+$ is a loss function and $p > 0$.

\item Assumptions~\ref{asp:lowbd-knownfeas}(c) means that $\nabla^2 f$ and $\nabla^2 c_i$, $1\le i\le m$, are locally Lipschitz continuous in $\Omega(\delta_f,\delta_c)$ with respect to the local norms. As pointed out in \cite[Section~5]{HL21bar}, such local Lipschitz continuity is weaker than the global Lipschitz continuity of $\nabla^2 f$ and $\nabla^2 c_i$, $1\le i\le m$, in $\Omega(\delta_f,\delta_c)$. Besides, Assumption~\ref{asp:lowbd-knownfeas}(d) holds if 
$f$ and $c$ are twice continuously differentiable in an open set containing $\cK$. 
\end{enumerate}
\end{remark}

\subsection{A Newton-CG based barrier-AL method}\label{sbsc:NCGBAL}

We now describe our Newton-CG based barrier-AL method (Algorithm~\ref{alg:2nd-order-AL-nonconvex}) for finding a stochastic $(\epsilon,\sqrt{\epsilon})$-SOSP of problem~\eqref{model:equa-cnstr} for a prescribed tolerance $\epsilon\in(0,1)$. Instead of solving \eqref{model:equa-cnstr} directly, our method solves a sequence of perturbed equality constrained barrier problems
\begin{equation}\label{md:eq-bar-pb}
\min_x\{f(x)+\mu_k B(x):\tilde{c}(x)=0\},
\end{equation}
where $\mu_k$ is given in Algorithm~\ref{alg:2nd-order-AL-nonconvex}, $z_\epsilon$ is given in Assumption~\ref{asp:lowbd-knownfeas}(a), and 
\begin{equation}
\tilde{c}(x):=c(x)-c(z_\epsilon).\label{per-eq-constr}
\end{equation}
It follows a similar AL framework as the one proposed in \cite{HL21al}. In particular, at the $k$th iteration, it first applies the preconditioned Newton-CG method (Algorithm~\ref{alg:NCG}) to find an approximate stochastic SOSP $x^{k+1}$  of the subproblem:
\begin{equation}\label{def:bar-AL-func}
\min_{x}\left\{\cL_{\mu_k}(x,\lambda^k;\rho_k):=f(x)+\mu_k B(x) + (\lambda^k)^T\tilde{c}(x) + \frac{\rho_k}{2}\|\tilde{c}(x)\|^2\right\},
\end{equation} 
which is an AL subproblem associated with \eqref{md:eq-bar-pb}. 
Then the standard multiplier estimate $\tilde{\lambda}^{k+1}$ is updated by the classical scheme (see step 3 of Algorithm~\ref{alg:2nd-order-AL-nonconvex}), and the truncated Lagrangian multiplier $\lambda^{k+1}$ is updated by projecting $\tilde{\lambda}^{k+1}$ onto a Euclidean ball (see step 5 of Algorithm~\ref{alg:2nd-order-AL-nonconvex}).\footnote{
The $\lambda^{k+1}$ is also called a safeguarded Lagrangian multiplier, which has been used in the literature for designing some AL methods (e.g., see \cite{ABM08,BM14,HL21al,KS17example}). It has been shown to enjoy many practical and theoretical advantages (e.g., see \cite{BM14}).} Finally, the penalty parameter $\rho_{k+1}$ is adaptively updated according to the improvement on constraint violation (see step~\ref{altstep:penalty} of Algorithm~\ref{alg:2nd-order-AL-nonconvex}). This update scheme is very practical and widely used in AL type methods (e.g., see \cite{ABM08,B97,CGLY17}).  

\begin{algorithm}[h]
\caption{A Newton-CG based barrier-AL method for problem~\eqref{model:equa-cnstr}}
{\small
\label{alg:2nd-order-AL-nonconvex}
Let $\gamma$ and $\mu$ be given in Assumption \ref{asp:lowbd-knownfeas}.\\
\noindent\textbf{Input}: $\epsilon\in(0,1)$, $\Lambda>0$, $x^0\in\rmint\cK$, $\lambda^0 \in \mathcal{B}_{\Lambda}$, $\rho_0>2\gamma$, $\alpha\in(0,1)$, ${r}>1$, $\delta\in(0,1)$, $z_\epsilon$ given in Assumption~\ref{asp:lowbd-knownfeas}(a), and $\mu_k=\max\{\epsilon,r^{k\log\epsilon/\log 2}\}/(2\vartheta^{1/2}+2)$ for all $k\ge0$.
\begin{algorithmic}[1]
\State Set $k=0$. 
\State Call Algorithm \ref{alg:NCG} with $\epsilon_g=\mu_k$, $\epsilon_H=\sqrt{\mu_k}$ and $u^0=x^{k}_{\init}$ to find an approximate solution $x^{k+1}\in\rmint\cK$ to $\min_x\mathcal{L}_{\mu_k}(x,\lambda^k;\rho_k)$ such that
\begin{eqnarray}
&&\cL_{\mu_k}(x^{k+1},\lambda^k;\rho_k)\le f(z_\epsilon) + \mu_k B(z_\epsilon),\quad  \|\nabla_x \cL_{\mu_k}(x^{k+1},\lambda^k;\rho_k)\|_{x^{k+1}}^*\le\mu_k,\label{algstop:1st-order}\\
&&\lambda_{\min}(M_{k+1}^T\nabla^2_{xx}\mathcal{L}_{\mu_k}(x^{k+1},\lambda^k;\rho_k)M_{k+1})\ge-\sqrt{\mu_k}\  \text{with probability at least } 1-\delta, \label{algstop:2nd-order}
\end{eqnarray}
where $M_{k+1}$ is defined as in \eqref{Mt} and
\begin{equation}\label{def:initial-iterate-subprob}
x^k_{\init}=\left\{\begin{array}{ll}
z_\epsilon&\text{if }\cL_{\mu_k}(x^k,\lambda^k;\rho_k)>f(z_\epsilon) + {\mu_k} B(z_\epsilon),\\
x^{k}&\text{otherwise},
\end{array}\right.\quad\text{for }k\ge 0.
\end{equation}
\State Set $\tilde{\lambda}^{k+1}=\lambda^k+\rho_k\tilde{c}(x^{k+1})$.
\State If $\mu_k\le\epsilon/(2\vartheta^{1/2}+2)$ and $\|c(x^{k+1})\|\le\epsilon$, then output $(x^{k+1},\tilde{\lambda}^{k+1})$ and terminate.\label{algstep:stop}
\State Set $\lambda^{k+1}=\Pi_{\mathcal{B}_\Lambda}(\tilde{\lambda}^{k+1})$.\label{algstep:proj-multiplier}
\State If $k=0$ or $\|\tilde{c}(x^{k+1})\|>\alpha\|\tilde{c}(x^k)\|$, set $\rho_{k+1}={r}\rho_k$. Otherwise, set $\rho_{k+1}=\rho_k$.\label{altstep:penalty}
\State Set $k\leftarrow k+1$, and go to step 2.
\end{algorithmic}
}
\end{algorithm}

\begin{remark}\label{rmk:alg2}
\begin{enumerate}[{\rm (i)}]
	
\item Notice that the starting point $x^0_{\init}$ of Algorithm~\ref{alg:2nd-order-AL-nonconvex} can be different from $z_\epsilon$ and it may be rather infeasible, though $z_\epsilon$ is a nearly feasible point of \eqref{model:equa-cnstr}. Besides, $z_\epsilon$ is used to monitor convergence of Algorithm~\ref{alg:2nd-order-AL-nonconvex}. Specifically, if the algorithm runs into a ``poorly infeasible point" $x^k$, namely satisfying $\cL_{\mu_k}(x^k,\lambda^k;\rho_k)>f(z_\epsilon)+\mu_k B(z_\epsilon)$, it will be superseded by $z_\epsilon$ (see \eqref{def:initial-iterate-subprob}), which prevents the iterates $\{x^k\}$ from converging to an infeasible point. Yet, $x^k$ may be rather infeasible when $k$ is not large. Thus, Algorithm~\ref{alg:2nd-order-AL-nonconvex} substantially differs from a funneling or two-phase type algorithm, in which a nearly feasible point is found in Phase 1, and then approximate stationarity is sought while near feasibility is maintained throughout Phase 2 (e.g., see \cite{BGMST16,bueno2020complexity,cartis2013evaluation,cartis2014complexity,CGT14,
cartis2015evaluation,cartis2019evaluation,curtis2018complexity}).
	
\item The choice of $\rho_0$ in Algorithm \ref{alg:2nd-order-AL-nonconvex} is mainly for the simplicity of complexity analysis. Yet, it may be overly large and lead to highly ill-conditioned AL subproblems in practice. To make Algorithm \ref{alg:2nd-order-AL-nonconvex} practically more efficient, one can possibly modify it by choosing a relatively small initial penalty parameter, then solving the subsequent AL subproblems by a first-order method until an $\epsilon_1$-first-order stationary point $\hat x$ of \eqref{md:eq-bar-pb} along with a Lagrangian multiplier $\hat \lambda$ is found,  and finally performing the steps described in Algorithm \ref{alg:2nd-order-AL-nonconvex} but with $x^0=\hat{x}$ and $\lambda^0=\Pi_{\mathcal{B}_\Lambda}(\hat \lambda)$. 


\item Algorithm~\ref{alg:2nd-order-AL-nonconvex} can be easily extended to find an $(\epsilon,\sqrt{\epsilon})$-SOSP of a more general conic optimization problem of the form $\min_{x,y} \{\tilde{f}(x,y):\tilde{c}(x,y)=0,y\in\cK\}$. Indeed, one can follow almost the same framework as  Algorithm~\ref{alg:2nd-order-AL-nonconvex}, except that the associated subproblems are solved by a preconditioned Newton-CG method, which is a slight modification of Algorithm \ref{alg:NCG} by choosing the preconditioning matrix $\widetilde{M}_k$ as the one satisfying
\[
\begin{bmatrix}
I& 0\\
0 &\nabla^2B(y^k)
\end{bmatrix}^{-1} =  \widetilde{M}_k\widetilde{M}_k^T.
\]
\item {It is worth mentioning that Algorithm~\ref{alg:2nd-order-AL-nonconvex} shares some similarities with classical primal interior point methods for convex conic optimization. Specifically, Algorithm~\ref{alg:2nd-order-AL-nonconvex} applies a damped Newton's method to solve a sequence of barrier-AL subproblems, while the primal interior point method applies a projected Newton's method to solve a sequence of constrained barrier subproblems.}
\end{enumerate}
\end{remark}

Before analyzing the complexity of Algorithm \ref{alg:2nd-order-AL-nonconvex}, we first argue that it is well-defined if $\rho_0$ is suitably chosen. Specifically, we will show that when $\rho_0$ is sufficiently large, one can apply Algorithm \ref{alg:NCG} to the subproblem $\min_x\cL_{\mu_k}(x,\lambda^k;\rho_k)$ with $x^k_{\init}$ as the initial point to find an $x^{k+1}$ satisfying \eqref{algstop:1st-order} and \eqref{algstop:2nd-order}. 
To this end, we start by noting from \eqref{ubd}, \eqref{per-eq-constr}, \eqref{def:bar-AL-func} and \eqref{def:initial-iterate-subprob} that
\begin{equation}\label{L-xinit}
\cL_{\mu_k}(x^{k}_{\init},\lambda^k;\rho_k) \overset{\eqref{def:initial-iterate-subprob}}{\le} \max\{\cL_{\mu_k}(z_\epsilon,\lambda^k;\rho_k),f(z_\epsilon) + \mu_kB(z_\epsilon)\}\overset{\eqref{per-eq-constr}\eqref{def:bar-AL-func}}{=}f(z_\epsilon) + \mu_k B(z_\epsilon)\overset{\eqref{ubd}}{\le}\fh.
\end{equation}
Based on this observation, we show in the next lemma that when $\rho_0$ is sufficiently large, $\cL_{\mu_k}(\cdot,\lambda^k;\rho_k)$ is bounded below and its certain level set is bounded, whose proof is deferred to Section~\ref{sec:proof-sec4}.

\begin{lemma}[{{\bf Properties of {$\cL_{\mu_k}(\cdot,\lambda^k;\rho_k)$} and $\cL(\cdot,\lambda^k;\rho_k)$}}]\label{lem:level-set-augmented-lagrangian-func}
Suppose that Assumption~\ref{asp:lowbd-knownfeas} holds. Let $(\lambda^k, \rho_k)$ be generated at the $k$th iteration of Algorithm~\ref{alg:2nd-order-AL-nonconvex} for some $k\ge0$, and 
\begin{equation}\label{def:ALfunc}
\cL(x,\lambda^k;\rho_k):=f(x)+(\lambda^k)^T\tilde{c}(x)+\frac{\rho_k}{2}\|\tilde{c}(x)\|^2.
\end{equation}
Let $\cS(\bdelta_f,\bdelta_c)$ and $x_{\init}^k$ be respectively defined in \eqref{nearly-feas-level-set} and \eqref{def:initial-iterate-subprob}, {$\mu_k$ be given in Algorithm~\ref{alg:2nd-order-AL-nonconvex},} and let $\bdelta_f$, $\bdelta_c$, $\fh$, $\fl$, $L_H^f$, $L^c_H$, $U^f_H$, $U^c_g$, $U^c_H$ and $\Omega(\bdelta_f,\bdelta_c)$ be given in Assumption \ref{asp:lowbd-knownfeas}. Suppose that $\rho_0$ is sufficiently large such that $\delta_{f,1} \le \bdelta_f$ and $\delta_{c,1} \le \bdelta_c$, where 
\begin{equation}\label{def:delta0c-rhobar1xxx}
\delta_{f,1}:=\Lambda^2/(2\rho_0)\ \text{ and }\ \delta_{c,1}:=\sqrt{\frac{2(\fh-\fl+\gamma)}{\rho_0-2\gamma}+\frac{\Lambda^2}{(\rho_0-2\gamma)^2}}+\frac{\Lambda}{\rho_0-2\gamma}.
\end{equation} 
Then the following statements hold.
\begin{enumerate}[{\rm (i)}]
\item {$\{x\in\rmint\cK:\cL_{\mu_k}(x,\lambda^k;\rho_k)\le \cL_{\mu_k}(x^{k}_{\init},\lambda^k;\rho_k)\}\subseteq \cS(\bdelta_f,\bdelta_c)$}.
\item {$\inf_{x\in\rmint\cK} \cL_{\mu_k}(x,\lambda^k;\rho_k) \ge\fl-\gamma-\Lambda\bdelta_c$}.
\item $\|\nabla^2_{xx}\cL(y,\lambda^k;\rho_k)-\nabla^2_{xx}\cL(x,\lambda^k;\rho_k)\|_x^*\le L_{k,H}\|y-x\|_x$ for all $x,y\in\Omega(\delta_f,\delta_c)\text{ with }\|y-x\|_x\le\beta$, where 
\beq \label{LkH}
L_{k,H}:=L^f_H+ \|\lambda^k\|_1 L^c_H+\rho_k m\left[(1+U^c)L_H^c+\frac{U^c_gU_H^c}{1-\beta}+\frac{(2-\beta)U_g^cU_H^c}{(1-\beta)^3}\right], \ U^c:=\sup\limits_{z\in\Omega(\delta_f,\delta_c)}\|c(z)\|.
\eeq
\item The quantities $U_{k,g}$ and $U_{k,H}$ are finite, where 
\[
U_{k,g}:=\sup_{x\in\cS(\delta_f,\delta_c)}\|\nabla_x\cL(x,\lambda^k;\rho_k)\|_{x}^*,\quad U_{k,H}:=\sup_{x\in\cS(\delta_f,\delta_c)}\|\nabla^2_{xx}\cL(x,\lambda^k;\rho_k)\|_{x}^*.
\]
Moreover, {$U_{k,g}  \leq U_g^f + \|\lambda^k\|_1U_g^c + \rho_k\sqrt{m}(2+\delta_c)U_g^c$} and $U_{k,H}  \leq U^f_H+\|\lambda^k\|_1 U^c_H+\rho_k(m(U^c_g)^2+\sqrt{m}(2 + \delta_c)U^c_H)$.
\end{enumerate}
\end{lemma}

In view of \eqref{nearly-feas-level-set} and Lemma~\ref{lem:level-set-augmented-lagrangian-func}(i) and (ii), one can see that the level set {$\{x\in\rmint\cK:\cL_{\mu_k}(x,\lambda^k;\rho_k) \le \cL_{\mu_k}(x_{\init}^k,\lambda^k;\rho_k)\}$} is bounded and {$\cL_{\mu_k}(x,\lambda^k;\rho_k)$} is bounded below for all $x\in\rmint\cK$. By these and  Lemma~\ref{lem:level-set-augmented-lagrangian-func}(iii) and (iv), one can further see that Assumption~\ref{asp:NCG-cmplxity} holds for $F(\cdot)=\cL(\cdot,\lambda^k;\rho_k)$ and $u^0=x_{\init}^k$. Based on this and the discussion in Section \ref{sec:sbpb-solver}, we can conclude that Algorithm~\ref{alg:NCG}, starting with $u^0=x^{k}_{\init}$, is applicable to the subproblem {$\min_x\cL_{\mu_k}(x,\lambda^k;\rho_k)$}. Moreover, it follows from Theorem \ref{thm:NCG-iter-oper-cmplxity} that this algorithm with {$\epsilon_g=\mu_k$ and $\epsilon_H=\sqrt{\mu_k}$} can produce a point $x^{k+1}$ satisfying \eqref{algstop:2nd-order} and also the second relation in \eqref{algstop:1st-order}. In addition, since this algorithm
is descent and its starting point is $x^{k}_{\init}$, its output $x^{k+1}$ must satisfy {$\cL_{\mu_k}(x^{k+1},\lambda^k;\rho_k) \le \cL_{\mu_k}(x^{k}_{\init},\lambda^k;\rho_k)$}, which along with \eqref{L-xinit} implies that {$\cL_{\mu_k}(x^{k+1},\lambda^k;\rho_k) \le f(z_\epsilon)+\mu_k B(z_\epsilon)$} and thus $x^{k+1}$ also satisfies the first relation in \eqref{algstop:1st-order}. 


The above discussion leads to the following conclusion concerning the well-definedness of Algorithm \ref{alg:2nd-order-AL-nonconvex}.         

\begin{theorem}[{{\bf Well-definedness of Algorithm \ref{alg:2nd-order-AL-nonconvex}}}]\label{subprob-solver}
Under the same settings as in Lemma \ref{lem:level-set-augmented-lagrangian-func}, the preconditioned Newton-CG method (Algorithm \ref{alg:NCG}), when applied to the subproblem {$\min_x \cL_{\mu_k}(x,\lambda^k;\rho_k)$} with $u^0=x^{k}_{\init}$, can find a point $x^{k+1}$ satisfying \eqref{algstop:1st-order} and \eqref{algstop:2nd-order}.
\end{theorem}

The following theorem characterizes the output of Algorithm \ref{alg:2nd-order-AL-nonconvex}, whose proof is deferred to Section~\ref{sec:proof-sec5}.

\begin{theorem}[{{\bf Output of Algorithm \ref{alg:2nd-order-AL-nonconvex}}}]\label{thm:output-alg1}
Suppose that Assumption~\ref{asp:lowbd-knownfeas} holds and that $\rho_0$ is sufficiently large such that $\delta_{f,1}\le\delta_f$ and $\delta_{c,1}\le\delta_c$, where $\delta_{f,1}$ and $\delta_{c,1}$ are defined in \eqref{def:delta0c-rhobar1xxx}. If Algorithm \ref{alg:2nd-order-AL-nonconvex} terminates at some iteration $k$, then its output $x^{k+1}$ is a deterministic $\epsilon$-FOSP of problem~\eqref{model:equa-cnstr}, and moreover, it is an $(\epsilon,\sqrt{\epsilon})$-SOSP of \eqref{model:equa-cnstr} with probability at least $1-\delta$.
\end{theorem}
	
\begin{remark}\label{only-1st}
As seen from Theorem~\ref{thm:output-alg1}, the output of Algorithm~\ref{alg:2nd-order-AL-nonconvex} is a stochastic $(\epsilon,\sqrt{\epsilon})$-SOSP of problem~\eqref{model:equa-cnstr}. 
On the other hand, this algorithm can be easily modified to find other approximate solutions of \eqref{model:equa-cnstr} as well. For example, if only an $\epsilon$-FOSP of \eqref{model:equa-cnstr} is to be sought, one can remove the condition \eqref{algstop:2nd-order} from Algorithm~\ref{alg:2nd-order-AL-nonconvex}. In addition, if one aims to find a deterministic $(\epsilon,\sqrt{\epsilon})$-SOSP of \eqref{model:equa-cnstr}, 
one can replace the condition \eqref{algstop:2nd-order} and Algorithm~\ref{alg:NCG} by {$\lambda_{\min}(M_{k+1}^T\nabla^2_{xx}\mathcal{L}_{\mu_k}(x^{k+1},\lambda^k;\rho_k)M_{k+1})\ge-\sqrt{\mu_k}$} and a deterministic counterpart, respectively.
\end{remark}

\subsection{Outer iteration complexity of Algorithm \ref{alg:2nd-order-AL-nonconvex}}\label{sbsc:outiter-bal}

In this subsection we establish \emph{outer iteration complexity} of Algorithm \ref{alg:2nd-order-AL-nonconvex}, which measures the number of its outer iterations. Notice that {$\mu_k$} can be rewritten as
\begin{equation}\label{omega-tolerance}
{
\mu_k=\max\{\epsilon,\omega^k\}/(2\vartheta^{1/2}+2)\; \text{with}\ \omega:={r}^{\log\epsilon/\log 2},\quad \forall k\ge 0,
}
\end{equation}
where $r$ is an input of Algorithm~\ref{alg:2nd-order-AL-nonconvex}. By {$\epsilon\in(0,1)$} and $r>1$, one has $\omega\in(0,1)$. For notational convenience, we introduce the following quantity that will be frequently used later:
\begin{equation}
{
K_{\epsilon}:=\left\lceil\min\{k\ge 0: \omega^{k}\le\epsilon\}\right\rceil. \label{T-epsilon-g}}
\end{equation}
In view of this and \eqref{omega-tolerance}, we obtain that {$\mu_k=\epsilon/(2\vartheta^{1/2}+2)$} for all $k\ge K_{\epsilon}$. This along with the termination criterion of Algorithm \ref{alg:2nd-order-AL-nonconvex} implies that it runs for at least $K_{\epsilon}$ iterations and  terminates once $\|c(x^{k+1})\|\le\epsilon$ for some $k\ge K_{\epsilon}$. Consequently, to establish outer iteration complexity of Algorithm \ref{alg:2nd-order-AL-nonconvex}, it suffices to bound such $k$. The resulting outer iteration complexity is presented below, whose proof is deferred to Section~\ref{sec:proof-sec5}.

\begin{theorem}[{{\bf Outer iteration complexity of Algorithm \ref{alg:2nd-order-AL-nonconvex}}}]\label{thm:out-itr-cmplxity-1}
Suppose that Assumption~\ref{asp:lowbd-knownfeas} holds and that $\rho_0$ is sufficiently large such that $\delta_{f,1}\le\delta_f$ and $\delta_{c,1}\le\delta_c$, where $\delta_{f,1}$ and $\delta_{c,1}$ are defined in \eqref{def:delta0c-rhobar1xxx}. Let
\begin{eqnarray}
&&{\rho}_{\epsilon,1}:=\max\left\{8(\fh-\fl+\gamma)\epsilon^{-2}+4\Lambda \epsilon^{-1}+2\gamma,2\rho_0\right\},\label{def:delta0c-rhobar1}\\
&&\overline{K}_{\epsilon}:=\inf\{k\ge K_{\epsilon}: \|c(x^{k+1})\|\le\epsilon\},\label{number-outer-iteration}
\end{eqnarray}
where $K_{\epsilon}$ is defined in \eqref{T-epsilon-g},  and $\gamma$, $\fh$ and $\fl$ are given in Assumption \ref{asp:lowbd-knownfeas}. Then $\overline{K}_{\epsilon}$ is finite, and Algorithm~\ref{alg:2nd-order-AL-nonconvex} terminates at iteration $\overline{K}_{\epsilon}$ with
\begin{equation}\label{outer-iteration-cmplxty}
\overline{K}_{\epsilon}\le\left(\frac{\log({\rho}_{\epsilon,1}\rho_0^{-1})}{\log{r}}+1\right)\left(\left|\frac{\log(\epsilon(2\delta_{c,1})^{-1})}{\log \alpha}\right|+2\right)+1.
\end{equation}
Moreover, $\rho_k\le{r}{\rho}_{\epsilon,1}$ holds for $0 \le k \le \overline{K}_{\epsilon}$.
\end{theorem}

\begin{remark}[{{\bf Upper bounds for $\overline{K}_{\epsilon}$ and $\{\rho_k\}$}}]\label{order-out-itera-penalty}
As seen from Theorem \ref{thm:out-itr-cmplxity-1}, the number of outer iterations of Algorithm \ref{alg:2nd-order-AL-nonconvex} for finding a stochastic $(\epsilon,\sqrt{\epsilon})$-SOSP of problem~\eqref{model:equa-cnstr} is at most of $\cO(|\log\epsilon|^2)$. In addition, the penalty parameters $\{\rho_k\}$ generated by this algorithm are at most of $\cO(\epsilon^{-2})$.
\end{remark}


\subsection{Total inner iteration and operation complexity of Algorithm \ref{alg:2nd-order-AL-nonconvex}}\label{subsec:total-complexity}\label{sbsc:inner-bal}

In this subsection we present the \emph{total inner iteration and operation complexity} of Algorithm \ref{alg:2nd-order-AL-nonconvex}, which measures the total number of iterations and fundamental operations performed by Algorithm~\ref{alg:NCG} in Algorithm~\ref{alg:2nd-order-AL-nonconvex}. Its proof is deferred to Section~\ref{sec:proof-sec5}.


\begin{theorem}[{\bf Total inner iteration and operation complexity of Algorithm \ref{alg:2nd-order-AL-nonconvex}}]\label{thm:total-iter-cmplxity}
Suppose that Assumption~\ref{asp:lowbd-knownfeas} holds and that $\rho_0$ is sufficiently large
such that $\delta_{f,1} \le \bdelta_f$ and  $\delta_{c,1} \le \bdelta_c$, where $\delta_{f,1}$ and $\delta_{c,1}$ are defined in \eqref{def:delta0c-rhobar1xxx}. Then the following statements hold.
\begin{enumerate}[{\rm (i)}]
\item The total number of inner iterations of Algorithm \ref{alg:2nd-order-AL-nonconvex}, namely, the total number of iterations of Algorithm~\ref{alg:NCG} performed in Algorithm \ref{alg:2nd-order-AL-nonconvex}, is at most $\widetilde{\cO}(\epsilon^{-11/2})$. If $c$ is further assumed to be affine, it is at most $\widetilde{\cO}(\epsilon^{-3/2})$.
\item The total numbers of {Cholesky factorizations and other fundamental operations} consisting of the Hessian-vector products of $f$ and $c$ and backward or forward substitutions to a triangular linear system
required by Algorithm~\ref{alg:NCG} in Algorithm~\ref{alg:2nd-order-AL-nonconvex} are at most $\widetilde{\cO}(\epsilon^{-11/2})$ and $\widetilde{\cO}(\epsilon^{-11/2}\min\{n,\epsilon^{-5/4}\})$, respectively. If $c$ is further assumed to be affine, they are at most $\widetilde{\cO}(\epsilon^{-3/2})$ and $\widetilde{\cO}(\epsilon^{-3/2}\min\{n,\epsilon^{-5/4}\})$, respectively.
\end{enumerate}
\end{theorem}
	
\begin{remark}
\begin{enumerate}[{\rm (i)}]
\item It is worth mentioning that the above complexity results are established without assuming any constraint qualification. Moreover, when $\cK$ is the nonnegative orthant, these results match the best-known ones achieved by a Newton-CG based AL method \cite{HL21al} for nonconvex equality constrained optimization without imposing a constraint qualification.
\item { For the unconstrained case, where $\cK = \bR^n$ and $c \equiv 0$, we can eliminate the barrier function for handling $\cK$ and the AL function for handling $c$ from Algorithm~\ref{alg:2nd-order-AL-nonconvex}. Consequently, Algorithm~\ref{alg:2nd-order-AL-nonconvex} reduces to the Newton-CG method presented in \cite[Algorithm 1]{HL21al} and \cite[Algorithm 3]{ROW20}. The corresponding operation complexity bound can be improved to $\widetilde{\cO}(\epsilon^{-3/2} \min{n, \epsilon^{-1/4}})$, which matches the results provided in \cite[Theorem 3.1]{HL21al} and \cite[Theorem 4 and Corollary 2]{ROW20} for unconstrained optimization.}
\end{enumerate}
\end{remark}

\subsection{Enhanced complexity of Algorithm \ref{alg:2nd-order-AL-nonconvex} under constraint qualification}\label{sec:AL-modified}

In this subsection we study complexity of Algorithm~\ref{alg:2nd-order-AL-nonconvex} under one additional assumption that a generalized linear independence constraint qualification (GLICQ) holds for problem \eqref{model:equa-cnstr}, which is introduced below. In particular, under GLICQ we will obtain an enhanced total inner iteration complexity of $\widetilde{\cO}(\epsilon^{-7/2})$ and an enhanced operation complexity of $\widetilde{\cO}(\epsilon^{-7/2}\min\{n,\epsilon^{-3/4}\})$ for Algorithm~\ref{alg:2nd-order-AL-nonconvex} when the equality constraints in problem~\eqref{model:equa-cnstr} are nonlinear, which are significantly better than the ones in Theorem \ref{thm:total-iter-cmplxity}. We now introduce the GLICQ assumption for \eqref{model:equa-cnstr}.

  

\begin{assumption}[{{\bf GLICQ}}]\label{asp:LICQ}
There exists some $\sigma>0$ such that
\begin{equation}\label{sigma}
\lambda_{\min}(\nabla c(x)^T\nabla^2 B(x)^{-1}\nabla c(x))\ge\sigma^2,\quad \forall x\in\cS(\delta_f,\delta_c),
\end{equation}
where $\cS(\bdelta_f,\bdelta_c)$ is defined in \eqref{nearly-feas-level-set}.
\end{assumption}


The following theorem shows that under Assumption \ref{asp:LICQ}, the total inner iteration and operation complexity results presented in Theorem \ref{thm:total-iter-cmplxity} can be significantly improved, whose proof is deferred to Section~\ref{sec:proof-sec5}.

\begin{theorem}[{\bf Enhanced total inner iteration and operation complexity of Algorithm \ref{alg:2nd-order-AL-nonconvex}}]\label{thm:total-iter-cmplxity2} 
Suppose that Assumptions \ref{asp:lowbd-knownfeas} and \ref{asp:LICQ} hold and that $\rho_0$ is sufficiently large 
such that $\delta_{f,1} \le \bdelta_f$ and  $\delta_{c,1} \le \bdelta_c$, where $\delta_{f,1}$ and $\delta_{c,1}$ are defined in \eqref{def:delta0c-rhobar1xxx}. Then the following statements hold.
\begin{enumerate}[{\rm (i)}]
\item The total number of inner iterations of Algorithm \ref{alg:2nd-order-AL-nonconvex}, namely, the total number of iterations of Algorithm~\ref{alg:NCG} performed in Algorithm~\ref{alg:2nd-order-AL-nonconvex}, is at most $\widetilde{\cO}(\epsilon^{-7/2})$. If $c$ is further assumed to be affine, it is at most $\widetilde{\cO}(\epsilon^{-3/2})$.
\item The total numbers of {Cholesky factorizations and other fundamental operations}
 consisting of the Hessian-vector products of $f$ and $c$ and backward or forward substitutions to a triangular linear system
required by Algorithm~\ref{alg:NCG} in Algorithm~\ref{alg:2nd-order-AL-nonconvex} are at most $\widetilde{\cO}(\epsilon^{-7/2})$ and $\widetilde{\cO}(\epsilon^{-7/2}\min\{n,\epsilon^{-3/4}\})$, respectively. If $c$ is further assumed to be affine, they are at most $\widetilde{\cO}(\epsilon^{-3/2})$ and $\widetilde{\cO}(\epsilon^{-3/2}\min\{n,\epsilon^{-3/4}\})$, respectively.
\end{enumerate}
\end{theorem}
	
\begin{remark}
As seen from Theorem~\ref{thm:total-iter-cmplxity2}, under GLICQ and some other suitable assumptions, Algorithm \ref{alg:2nd-order-AL-nonconvex} achieves significantly better complexity bounds than the ones in Theorem \ref{thm:total-iter-cmplxity} when the equality constraints in \eqref{model:equa-cnstr} are nonlinear. Moreover, when $\cK$ is the nonnegative orthant, the complexity results in Theorem~\ref{thm:total-iter-cmplxity2} match the best-known ones achieved by a Newton-CG based AL method \cite{HL21al} for nonconvex equality constrained optimization under the constraint qualification that is obtained from the above GLICQ by replacing $\nabla^2 B(x)$ by the identity matrix.
\end{remark}

\section{Numerical results}\label{sec:nr}

In this section we conduct some preliminary numerical experiments to test the performance of our Newton-CG based barrier-AL method (Algorithm~\ref{alg:2nd-order-AL-nonconvex}) for solving a low-rank matrix recovery problem, a simplex-constrained nonnegative matrix factorization problem, {and a sphere-constrained nonnegative matrix factorization problem.} In our experiments, all the algorithms are coded in Matlab and all the computations are performed on a desktop with a 3.79 GHz AMD 3900XT 12-Core processor and 32 GB of RAM.

\subsection{Low-rank matrix recovery} \label{rrr}
In this subsection we consider a low-rank matrix recovery problem (e.g., see \cite{BNS16global,CLC19,PKCS17non})
\begin{equation}\label{ms}
\min_{U\in\bR^{n\times l}}\ \left\{\frac{1}{2}\|\cA(UU^T)-y\|^2: \|U\|_F^2\le b\right\},
\end{equation}
where $\cA:\bR^{n\times n}\to\bR^m$ is a linear operator and $\|\cdot\|_F$ is the Frobenius norm. 

For each triple $(n,l,m)$, we randomly generate 10 instances of problem~\eqref{ms} in a similar manner as described in \cite{BNS16global}. In particular, we first randomly generate a linear operator $\cA$ by setting $\cA(\cdot)=A(\mathrm{vec}(\cdot))$, where  $A$ is an $m\times n^2$ matrix with all entries chosen from the normal distribution {with mean zero and standard deviation $1/\sqrt{m}$}, and $\mathrm{vec}(\cdot)$ is the vectorization of the associated matrix.\footnote{The vectorization of a matrix is the column vector obtained by stacking the columns of the matrix on top of one another.} Then we randomly generate the ground-truth low-rank matrix $X^*=\widetilde{U}\widetilde{U}^T$ with all entries of $\widetilde{U}$ chosen from the standard normal distribution. We finally set {$b=\|\widetilde{U}\|^2_F$} and $y = \cA(X^*)+e$, where $e_i$, $1\le i\le m$, is generated according to the normal distribution with mean zero and standard deviation $0.01$. 

Observe that problem \eqref{ms} is equivalent to
\beq \label{ms-reform}
\min_{U, s}\ \left\{\frac{1}{2}\|\cA(UU^T)-y\|^2: \|U\|_F^2+s= b,s\ge 0\right\}.
\eeq 
In this experiment, we apply Algorithm~\ref{alg:2nd-order-AL-nonconvex} to find a $(10^{-4},10^{-2})$-SOSP of \eqref{ms-reform} and hence of \eqref{ms}. To ensure that the output of Algorithm~\ref{alg:2nd-order-AL-nonconvex} is a deterministic approximate second-order stationary point, we use a minimum eigenvalue oracle that returns a deterministic output in Algorithm~\ref{alg:2nd-order-AL-nonconvex} instead, which calls the Matlab subroutine \textsf{[v,$\lambda$] = eigs(H,1,'smallestreal')} to find the minimum eigenvalue $\lambda$ and its associated unit eigenvector $v$ of a real symmetric matrix $H$. Besides, we apply \cite[Algorithm SpaRSA]{WNF09sparse}, which is a nonmonotone proximal gradient method, to find a $10^{-4}$-FOSP of \eqref{ms} by generating a sequence $\{U^t\}$ according to 
\[
U^t = \arg\min_U  \{\|U-U^{t-1}+\nabla f(U^{t-1})/\alpha_{t-1}\|_F: \|U\|_F^2\le b\},
\]
 where $f$ is the objective function of \eqref{ms} and $\alpha_{t-1}$ is chosen by a backtracking line search scheme such that 
$f(U^t) \leq \max_{[t-M-1]_+ \leq i \leq t-1} f(U^i) - \sigma \alpha_{t-1}\|U^t-U^{t-1}\|^2_F/2$ for some $\sigma\in (0,1)$ and a positive integer $M$ (see \cite{WNF09sparse} for details). We terminate  SpaRSA  once the condition
\[
\|\alpha_{t-1}(U^{t}-U^{t-1})+\nabla f(U^{t-1})-\nabla f(U^t)\|_F\le10^{-4}
\]
is met. It can be verified that such $U^t$ is a $10^{-4}$-FOSP of \eqref{ms}. We choose the initial point $U^0$ with all entries equal {to} $\sqrt{b/(2nl)}$ for both methods, $s^0=b/2$ for Algorithm~\ref{alg:2nd-order-AL-nonconvex}, and set  
\begin{itemize}
\item $(\Lambda,\rho_0,\lambda^0,\alpha,r)= (10^3,10^2,0,0.25,1.5)$ for Algorithm~\ref{alg:2nd-order-AL-nonconvex}, and $(\theta,\zeta,\eta,\beta)=(0.5,0.5,0.01,0.9)$ for Algorithm~\ref{alg:NCG};
\item $(\sigma,M,\alpha_{\min},\alpha_{\max},\eta)=(0.01,5,10^{-30}, 10^{30},2)$ for SpaRSA \cite{WNF09sparse}.
\end{itemize}

\begin{table}
\centering
\begin{tabular}{ccc||ll||ll}
\hline
& & &\multicolumn{2}{c||}{Relative error} & \multicolumn{2}{c}{Objective value} \\
$n$ & $l$ & $m$ & Algorithm 2 & SpaRSA & Algorithm 2 & SpaRSA  \\ \hline
20 & 1 & 40 & 6.3$\times10^{-4}$ &  6.3$\times10^{-4}$ & 9.9$\times10^{-4}$ &  9.8$\times10^{-4}$\\ 
20 & 2 & 80 & 3.3$\times10^{-4}$ &  0.60 & 2.0$\times10^{-3}$& 7.8$\times10^{3}$\\
40 & 2 & 160 &1.7$\times10^{-4}$& 0.66 & 4.2$\times10^{-3}$ & 7.1$\times10^{4}$ \\ 
40 & 4 & 320 & 1.2$\times$10$^{-4}$ & 0.81 & 8.0$\times10^{-3}$& 5.5$\times10^{5}$\\ 
60 & 3 & 360 &9.2$\times$10$^{-5}$ & 0.78 & 9.3$\times10^{-3}$& 8.4$\times10^{5}$\\ 
60 & 6 & 720 & 6.3$\times 10^{-5}$  &  0.85  & 1.9$\times10^{-2}$& 5.1$\times10^{6}$\\ 
80 & 4 & 640 & 5.8$\times 10^{-5}$ &0.83& 1.6$\times10^{-2}$& 4.4$\times10^{6}$ \\ 
80 & 8 & 1280 & 3.9$\times 10^{-5}$ & 0.90  & 3.3$\times10^{-2}$&2.5$\times10^{7}$ \\
100 & 5& 1000  &4.2$\times 10^{-5}$ & 0.89   & 2.6$\times10^{-2}$&1.6$\times10^{7}$ \\ 
100 & 10& 2000 &2.8$\times 10^{-5}$&  0.92 & 5.2$\times10^{-2}$ &8.1$\times10^{7}$ \\
\hline 
\end{tabular}
\caption{Numerical results for problem~\eqref{ms}}
\label{table:ms}
\end{table}

\begin{figure}[!ht]
\centering
\includegraphics[width=0.48\linewidth]{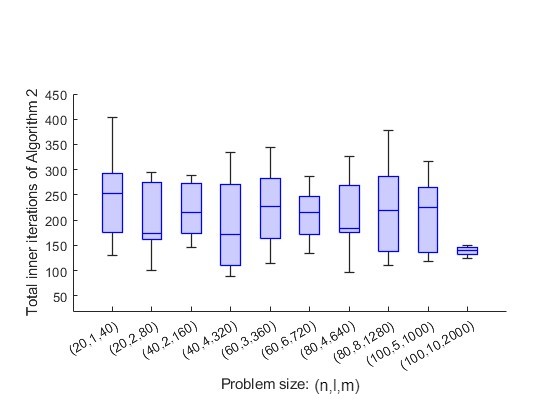}
\includegraphics[width=0.48\linewidth]{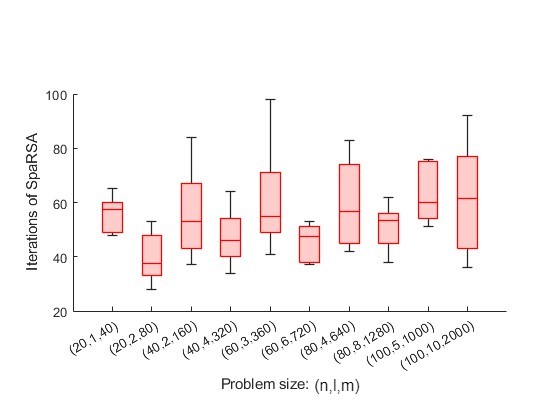}
\caption{{Left: The total number of inner iterations of Algorithm~\ref{alg:2nd-order-AL-nonconvex} for finding a $(10^{-4},10^{-2})$-SOSP of \eqref{ms-reform} for each problem size over 10 random instances. Right: The number of iterations of SpaRSA for finding a $10^{-4}$-FOSP of \eqref{ms} for each problem size over 10 random instances.}}
\label{fig:tt-ms}
\end{figure}
\begin{figure}[!ht]
\centering
\includegraphics[width=0.48\linewidth]{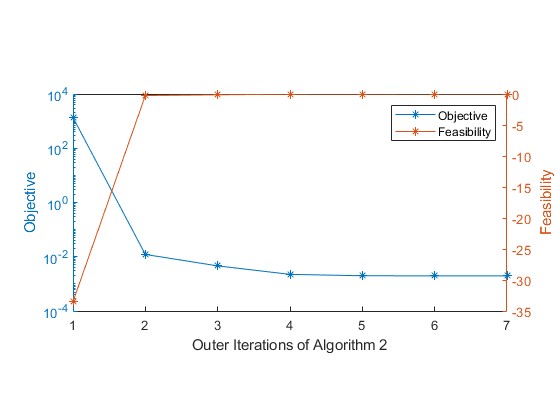}
\includegraphics[width=0.48\linewidth]{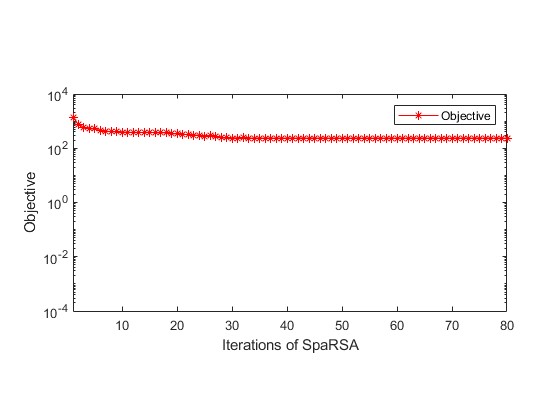}
\caption{{Numerical results of Algorithm~\ref{alg:2nd-order-AL-nonconvex} and SpaRSA on a single random instance of problem~\eqref{ms} with $(n,l,m)=(20,2,80)$. These two figures illustrate the convergence behavior of both methods in terms of objective value $\frac{1}{2}\|\cA(U^t(U^t)^T)-y\|^2$ and feasibility $[\|U^t\|_F^2-b]_+$.}}
\label{fig:ms-beh}
\end{figure}

Notice that the approximate solution obtained by SpaRSA must be a feasible point of \eqref{ms}, while the one found by Algorithm~\ref{alg:2nd-order-AL-nonconvex} may not be a feasible point of \eqref{ms}. For a fair comparison, we project the latter one into the feasible region of  \eqref{ms} to obtain a feasible approximate solution.  Then we compare the quality of these feasible approximate solutions in terms of objective value and relative error defined as $\|UU^T-X^*\|_F/\|X^*\|_F$ for a given $U$. The computational results of Algorithm~\ref{alg:2nd-order-AL-nonconvex} and SpaRSA for the instances randomly generated above are presented in Table~\ref{table:ms}. In detail, the values of $n$, $l$ and $m$ are listed in the first three columns,  respectively. For each triple $(n,l,m)$, the average relative error and the average objective value of the feasible approximate solutions found by each method over 10 random instances are given in the rest columns. 
One can observe that the approximate SOSP found by Algorithm~\ref{alg:2nd-order-AL-nonconvex} has significantly lower relative error and objective value than the approximate FOSP obtained by SpaRSA, {except for the instances with $(n,l,m)=(20,1,40)$.}

For each triple $(n, l, m)$, we use box charts in Figure~\ref{fig:tt-ms} to present the total number of inner iterations of Algorithm~\ref{alg:2nd-order-AL-nonconvex} for finding a $(10^{-4},10^{-2})$-SOSP of problem~\eqref{ms-reform} and the number of iterations of SpaRSA for finding a $10^{-4}$-FOSP of problem~\eqref{ms} over 10 random instances. We observe that the total number of inner iterations of Algorithm~\ref{alg:2nd-order-AL-nonconvex} remains at a similar level when the problem size becomes larger. In addition, Figure~\ref{fig:ms-beh} illustrates the convergence behavior of both methods for solving a single random instance of \eqref{ms} with $(n,l,m)=(20,2,80)$. 
{We observe that SpaRSA converges in fewer (inner) iterations than Algorithm~\ref{alg:2nd-order-AL-nonconvex}. However, it converges to a suboptimal solution with a significantly larger objective value compared to the solution found by Algorithm~\ref{alg:2nd-order-AL-nonconvex}.}

\subsection{A simplex-constrained nonnegative matrix factorization}
In this subsection we consider a simplex-constrained nonnegative matrix factorization (e.g., see \cite{HG10,LXWZ10,MQ07,TGL22}) in the form of
\begin{equation}\label{nmf}
\min_{U\in\bR^{n\times l},V\in\bR^{l\times m}}\ \left\{\frac{1}{2}\|X-UV\|_F^2+\gamma(\|U\|_F^2+\|V\|_F^2):V^Te_l = e_m, U\ge0, V\ge0\right\}
\end{equation}
for some $\gamma>0$, where $\|\cdot\|_F$ is the Frobenius norm and $e_d$ is the $d$-dimensional all-ones vector for any $d\ge 1$.

\begin{table}
\centering
\begin{tabular}{ccc||ll||ll}
\hline
& & &\multicolumn{2}{c||}{Relative error} & \multicolumn{2}{c}{Objective value} \\
$n$ & $l$ & $m$ & Algorithm 2 & SpaRSA & Algorithm 2  & SpaRSA  \\ \hline
20 & 2 & 10  & 4.8$\times 10^{-3}$ &   0.15  &  0.30 &  3.1  \\
20 & 2 & 20  & 3.6$\times 10^{-3}$ &   0.16  &  0.35 &  6.4  \\
20 & 2 & 30  & 3.2$\times 10^{-3}$ &   0.16 &  0.39 &  9.7   \\
30 & 3 & 15  & 5.8$\times 10^{-3}$  &  0.16   & 0.62  & 7.6    \\ 
30 & 3 & 30  & 4.3$\times 10^{-3}$  &  0.17   & 0.70  & 16.1 \\ 
30 & 3 & 45  & 3.6$\times 10^{-3}$  &  0.17   & 0.76  & 23.8  \\
40 & 4 & 20  & 6.3$\times 10^{-3}$  &  0.15    & 1.0  & 11.1  \\ 
40 & 4 & 40  & 4.6$\times 10^{-3}$  &  0.15    & 1.2  & 21.7  \\
40 & 4 & 60  & 4.0$\times 10^{-3}$  &  0.15    & 1.2  & 31.8  \\
50 & 5 & 25  & 6.8$\times 10^{-3}$  &  0.14    & 1.6  & 15.1  \\
50 & 5 & 50  & 5.0$\times 10^{-3}$  &  0.14   & 1.8 &  29.8  \\
50 & 5 & 75  & 4.3$\times 10^{-3}$  &  0.14   & 1.9 &  43.9  \\
\hline 
\end{tabular}
\caption{Numerical results for problem~\eqref{nmf}}
\label{table:nmf}
\end{table}

\begin{figure}[!ht]
\centering
\includegraphics[width=0.48\linewidth]{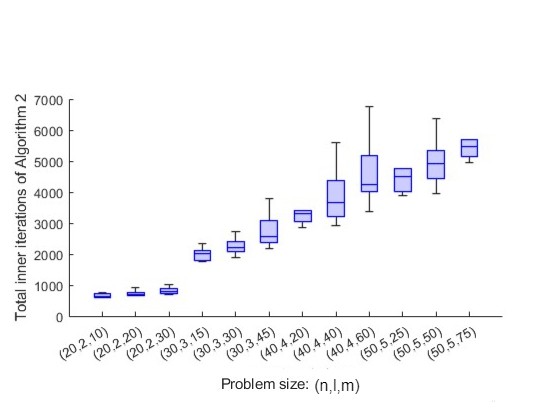}
\includegraphics[width=0.48\linewidth]{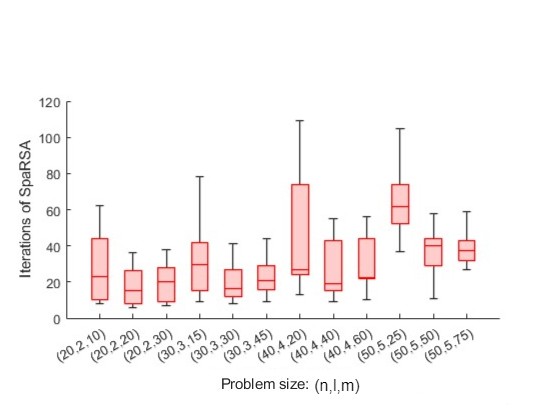}
\caption{{Left: The total inner iterations of Algorithm~\ref{alg:2nd-order-AL-nonconvex} before finding a $(10^{-4},10^{-2})$-SOSP of \eqref{nmf} for each problem size over 10 random instances. Right: The number of iterations of SpaRSA before finding a $10^{-4}$-FOSP of \eqref{nmf} for each problem size over 10 random instances.}}
\label{fig:tt-nmf}
\end{figure}

\begin{figure}[!ht]
\centering
\includegraphics[width=0.48\linewidth]{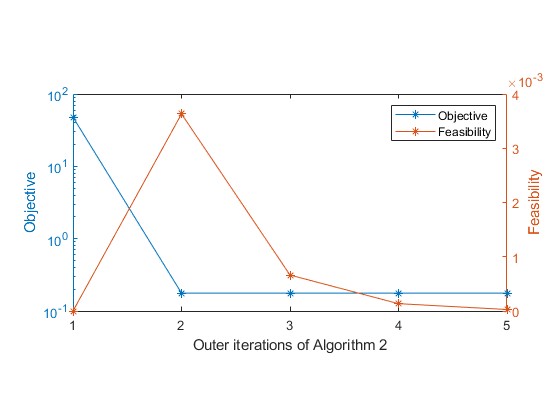}
\includegraphics[width=0.48\linewidth]{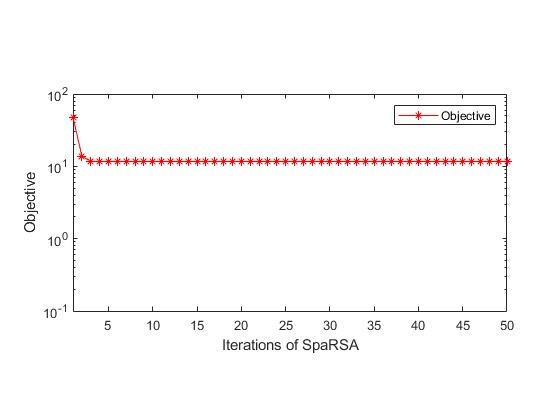}
\caption{Numerical results of Algorithm~\ref{alg:2nd-order-AL-nonconvex} and SpaRSA on a single random instance of problem~\eqref{nmf} with $(n,l,m)=(20,2,20)$. These two figures illustrate the convergence behavior of both methods in terms of objective value $\frac{1}{2}\|X-U^tV^t\|_F^2+\gamma(\|U^t\|_F^2+\|V^t\|_F^2)$ and feasibility $\|(V^t)^Te_l-e_m\|$.}
\label{fig:nmf-beh}
\end{figure}

For each triple $(n,l,m)$, we randomly generate 10 instances of problem~\eqref{nmf}. In particular, we first randomly generate $U^*$ with all entries chosen from the uniform distribution over $[0,2]$. We next randomly generate $\widetilde{V}$ with all entries chosen from the standard uniform distribution and set $V^*=\widetilde{V}D$, where $D$ is a diagonal matrix such that $(V^*)^Te_l=e_m$. In addition, we set $\gamma=0.005$ and $X=U^*V^*+E$, where the entries of $E$ follow the normal distribution with mean zero and standard deviation $0.01$. 

Our aim is to apply Algorithm~\ref{alg:2nd-order-AL-nonconvex} and SpaRSA \cite{WNF09sparse} to solve \eqref{nmf} and compare the solution quality of these methods in terms of objective value and relative error defined as $\|UV-U^*V^*\|_F/\|U^*V^*\|_F$. In particular, we first apply Algorithm~\ref{alg:2nd-order-AL-nonconvex} to find a $(10^{-4},10^{-2})$-SOSP of \eqref{nmf}, in which a minimum eigenvalue oracle that returns a deterministic output, namely the Matlab subroutine \textsf{[v,$\lambda$] = eigs(H,1,'smallestreal')} is used.  Given that the obtained approximate SOSP may not be a feasible point of \eqref{nmf}, we post-multiply it by a suitable diagonal matrix to obtain a feasible approximate solution 
of \eqref{nmf}. In addition, we apply SpaRSA \cite{WNF09sparse} to find a $10^{-4}$-FOSP of \eqref{nmf} by generating a sequence $\{(U^t,V^t)\}$ according to 
\[
(U^t,V^t)= \arg\min_{U ,V} \left\{\|(U,V)-(U^{t-1},V^{t-1})+\nabla f(U^{t-1},V^{t-1})/\alpha_{t-1}\|_F: V^Te_l = e_m, U\ge0, V\ge0\right\},
\]
where $f$ is the objective function of \eqref{nmf} and $\alpha_{t-1}$ is chosen by a backtracking line search scheme such that 
$f(U^t,V^t) \leq \max_{[t-M-1]_+ \leq i \leq t-1} f(U^i,V^i) - \sigma \alpha_{t-1}\|(U^t,V^t)-(U^{t-1},V^{t-1})\|^2_F/2$ for some $\sigma\in (0,1)$ and a positive integer $M$ (see \cite{WNF09sparse} for details). We terminate  SpaRSA  once the condition
\[
\|\alpha_{t-1}((U^t,V^t)-(U^{t-1},V^{t-1}))+\nabla f(U^{t-1},V^{t-1})-\nabla f(U^t,V^t)\|_F \le10^{-4}
\]
is met. It can be verified that such $(U^t,V^t)$ is a $10^{-4}$-FOSP of \eqref{nmf}.  In addition, we choose the initial point $U^0$ and $V^0$ with all entries equal $1$ and $1/k$ respectively for all the methods. We set the parameters for Algorithm~\ref{alg:2nd-order-AL-nonconvex}  as $(\Lambda,\rho_0,\alpha,r)= (10^3,10^2,0.25,1.5)$ and $\lambda^0=(0,\ldots,0)^T$, and choose the same parameters for Algorithm~\ref{alg:NCG} and SpaRSA as the ones described in Subsection~\ref{rrr}.

The computational results of Algorithm~\ref{alg:2nd-order-AL-nonconvex} and  SpaRSA \cite{WNF09sparse}  for the instances randomly generated above are presented in Table~\ref{table:nmf}. In detail, the values of $n$, $k$ and $m$ are listed in the first three columns,  respectively. For each triple $(n,k,m)$, the average relative error and the average objective value of the feasible approximate solutions found by each method over 10 random instances are given in the rest columns. One can observe that the approximate SOSP found by Algorithm~\ref{alg:2nd-order-AL-nonconvex} has significantly lower relative error and objective value than the approximate FOSP obtained by SpaRSA. 

For each triple $(n, l, m)$, we use box charts in Figure~\ref{fig:tt-nmf} to present the total number of inner iterations of Algorithm~\ref{alg:2nd-order-AL-nonconvex} for finding a $(10^{-4},10^{-2})$-SOSP of problem~\eqref{nmf} and the number of iterations of SpaRSA for finding a $10^{-4}$-FOSP of problem~\eqref{nmf} over 10 random instances. We observe that the total number of inner iterations of Algorithm~\ref{alg:2nd-order-AL-nonconvex} increases when the problem size becomes larger. In addition, Figure~\ref{fig:nmf-beh} illustrates the convergence behavior of both methods for solving a single random instance of \eqref{nmf} with $(n,l,m)=(20,2,20)$. 
{We observe that SpaRSA converges in fewer (inner) iterations than Algorithm~\ref{alg:2nd-order-AL-nonconvex}, but it converges to a poorer approximate solution. Specifically, it quickly gets stuck at a suboptimal point with a significantly larger objective value compared to the solution found by Algorithm~\ref{alg:2nd-order-AL-nonconvex}.}



\subsection{A sphere-constrained nonnegative matrix factorization}
In this subsection we consider a sphere-constrained nonnegative matrix factorization in the form of
\begin{equation}\label{sph-nmf}
\min_{U\in\bR^{n\times l},V\in\bR^{l\times m}}\ \left\{\frac{1}{2}\|X-UV\|_F^2+\gamma(\|U\|_F^2+\|V\|_F^2):\|V\|_F^2 = m, U\ge0, V\ge0\right\},
\end{equation}
where $\|\cdot\|_F$ is the Frobenius norm.

For each triple $(n,l,m)$, we randomly generate 10 instances of problem~\eqref{sph-nmf}. In particular, we first randomly generate $U^*$ with all entries chosen from the uniform distribution over $[0,2]$. We next randomly generate $\widetilde{V}$ with all entries chosen from the standard uniform distribution and set $V^*=\sqrt{m}\widetilde{V}/\|\widetilde{V}\|_F$. In addition, we set $\gamma=0.005$ and $X=U^*V^*+E$, where the entries of $E$ follow the normal distribution with mean zero and standard deviation $0.01$. 

Our aim is to apply Algorithm~\ref{alg:2nd-order-AL-nonconvex} and SpaRSA \cite{WNF09sparse} to solve \eqref{sph-nmf} and compare the solution quality of these methods in terms of objective value and relative error defined as $\|UV-U^*V^*\|_F/\|U^*V^*\|_F$.  In particular, we first apply Algorithm~\ref{alg:2nd-order-AL-nonconvex} to find a $(10^{-4},10^{-2})$-SOSP of \eqref{sph-nmf}, in which a minimum eigenvalue oracle that returns a deterministic output, namely the Matlab subroutine \textsf{[v,$\lambda$] = eigs(H,1,'smallestreal')} is used.  Given that the obtained approximate SOSP may not be a feasible point of \eqref{sph-nmf}, we post-multiply it by a suitable diagonal matrix to obtain a feasible approximate solution 
of \eqref{sph-nmf}. In addition, we apply SpaRSA \cite{WNF09sparse} to find a $10^{-4}$-FOSP of \eqref{nmf} by generating a sequence $\{(U^t,V^t)\}$ according to 
\[
(U^t,V^t)= \arg\min_{U ,V} \left\{\|(U,V)-(U^{t-1},V^{t-1})+\nabla f(U^{t-1},V^{t-1})/\alpha_{t-1}\|_F: \|V\|^2_F=m, U\ge0, V\ge0\right\},
\]
where $f$ is the objective function of \eqref{sph-nmf} and $\alpha_{t-1}$ is chosen by a backtracking line search scheme such that 
$f(U^t,V^t) \leq \max_{[t-M-1]_+ \leq i \leq t-1} f(U^i,V^i) - \sigma \alpha_{t-1}\|(U^t,V^t)-(U^{t-1},V^{t-1})\|^2_F/2$ for some $\sigma\in (0,1)$ and a positive integer $M$ (see \cite{WNF09sparse} for details). We terminate  SpaRSA  once the condition
\[
\|\alpha_{t-1}((U^t,V^t)-(U^{t-1},V^{t-1}))+\nabla f(U^{t-1},V^{t-1})-\nabla f(U^t,V^t)\|_F \le10^{-4}
\]
is met. It can be verified that such $(U^t,V^t)$ is a $10^{-4}$-FOSP of \eqref{sph-nmf}.  Furthermore, we choose the initial point $U^0$ and $V^0$ with all entries equal to $1$ and $1/l$ respectively for all the methods. We set the parameters for Algorithm~\ref{alg:2nd-order-AL-nonconvex}  as $(\Lambda,\rho_0,\alpha,r)= (10^3,10^2,0.25,1.5)$ and $\lambda^0=0$, and choose the same parameters for Algorithm~\ref{alg:NCG} and SpaRSA as the ones described in Subsection~\ref{rrr}.

\begin{table}
\centering
\begin{tabular}{ccc||ll||ll}
\hline
& & &\multicolumn{2}{c||}{Relative error} & \multicolumn{2}{c}{Objective value}  \\
$n$ & $l$ & $m$ & Algorithm 2 & SpaRSA & Algorithm 2  & SpaRSA  \\ \hline
20 & 2 & 5  & 5.5$\times 10^{-3}$ & 0.10 & 0.29 & 2.3  \\
20 & 2 & 10 & 3.9$\times 10^{-3}$ & 0.12 & 0.33 & 5.5  \\
20 & 2 & 15 & 3.4$\times 10^{-3}$ & 0.13 & 0.36 & 8.6  \\
20 & 2 & 20 & 3.1$\times 10^{-3}$ & 0.12 & 0.39 & 10.9 \\ 
20 & 2 & 25 & 2.9$\times 10^{-3}$ & 0.12 & 0.42 & 13.5  \\ 
20 & 2 & 30 & 2.8$\times 10^{-3}$ & 0.12 & 0.45 & 16.5 \\
40 & 4 & 10 & 4.5$\times 10^{-3}$ & 0.12 & 1.1 & 18.1 \\
40 & 4 & 20 & 3.1$\times 10^{-3}$ & 0.12 & 1.2 & 36.8  \\
40 & 4 & 30 & 2.8$\times 10^{-3}$ & 0.12 & 1.3 & 54.6 \\
40 & 4 & 40 & 2.6$\times 10^{-3}$ & 0.13 & 1.3 & 75.0 \\
40 & 4 & 50 & 2.2$\times 10^{-3}$ & 0.13 & 1.4 & 94.5  \\ 
40 & 4 & 60 & 2.3$\times 10^{-3}$ & 0.13 & 1.5 & 113.0  \\
\hline 
\end{tabular}
\caption{Numerical results for problem~\eqref{sph-nmf}}
\label{table:snmf}
\end{table}

\begin{figure}[!ht]
\centering
\includegraphics[width=0.48\linewidth]{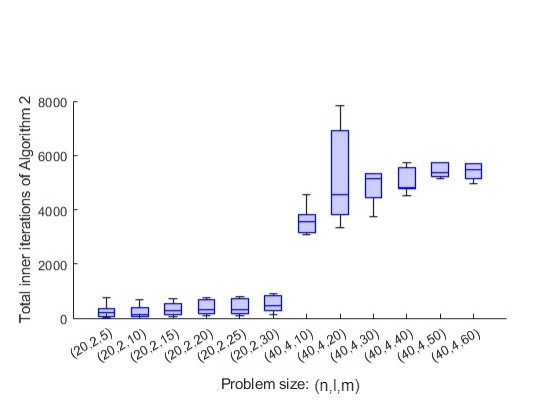}
\includegraphics[width=0.48\linewidth]{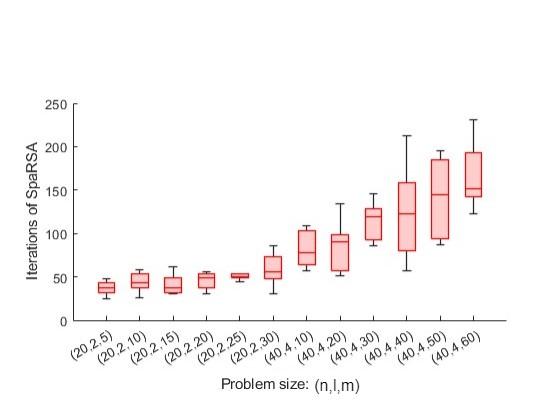}
\caption{{ Left: The total inner iterations of Algorithm~\ref{alg:2nd-order-AL-nonconvex} before finding a $(10^{-4},10^{-2})$-SOSP of \eqref{sph-nmf} for each problem size over 10 random instances. Right: The number of iterations of SpaRSA before finding a $10^{-4}$-FOSP of \eqref{sph-nmf} for each problem size over 10 random instances.}}
\label{fig:tt-snmf}
\end{figure}

\begin{figure}[!ht]
\centering
\includegraphics[width=0.48\linewidth]{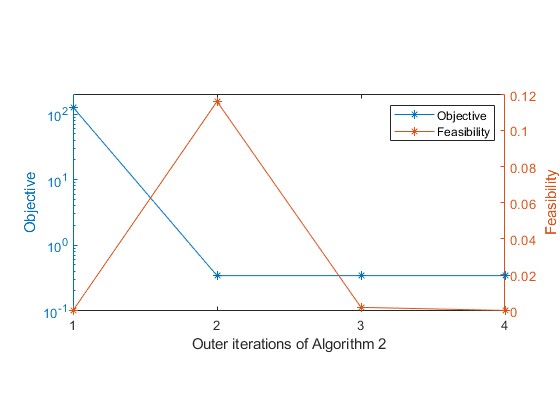}
\includegraphics[width=0.48\linewidth]{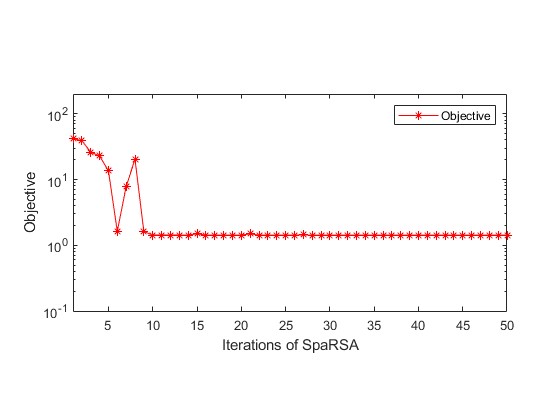}
\caption{{Numerical results of Algorithm~\ref{alg:2nd-order-AL-nonconvex} and SpaRSA on a single random instance of problem~\eqref{sph-nmf} with $(n,l,m)=(20,2,5)$. These two figures illustrate the convergence behavior of both methods in terms of objective value $\frac{1}{2}\|X-U^tV^t\|_F^2+\gamma(\|U^t\|_F^2+\|V^t\|_F^2)$ and feasibility $|\|V^t\|^2_F-m|$.}}
\label{fig:snmf-beh}
\end{figure}

The computational results of Algorithm~\ref{alg:2nd-order-AL-nonconvex} and SpaRSA \cite{WNF09sparse}  for the instances randomly generated above are presented in Table~\ref{table:snmf}. In detail, the values of $n$, $l$ and $m$ are listed in the first three columns,  respectively.  For each triple $(n,l,m)$, the average relative error and the average objective value of the feasible approximate solutions over 10 random instances are given in the rest columns. One can observe that the approximate SOSP found by Algorithm~\ref{alg:2nd-order-AL-nonconvex} has significantly lower relative error and objective value than the approximate FOSP obtained by SpaRSA. 

For each triple $(n, l, m)$, we use box charts in Figure~\ref{fig:tt-snmf} to present the total number of inner iterations of Algorithm~\ref{alg:2nd-order-AL-nonconvex} for finding a $(10^{-4},10^{-2})$-SOSP of problem~\eqref{sph-nmf} and the number of iterations of SpaRSA for finding a $10^{-4}$-FOSP of problem~\eqref{sph-nmf} over 10 random instances. We observe that the total number of inner iterations of Algorithm~\ref{alg:2nd-order-AL-nonconvex} increases when the problem size becomes larger. In addition, Figure~\ref{fig:snmf-beh} illustrates the convergence behavior of both methods for solving a single random instance of \eqref{sph-nmf} with $(n,l,m)=(20,2,5)$. 
{We observe that SpaRSA converges in fewer (inner) iterations than Algorithm~\ref{alg:2nd-order-AL-nonconvex}. However, it converges to a suboptimal solution with a much larger objective value compared to the solution found by Algorithm~\ref{alg:2nd-order-AL-nonconvex}.}

\section{Proof of the main results} \label{sec:proof}

In this section we provide a proof of our main results presented in Sections~\ref{sec:opt}, \ref{sec:sbpb-solver}, and \ref{sec:AL-method}, which are, particularly, Theorems~\ref{thm:1stopt} and \ref{thm:NCG-iter-oper-cmplxity}, Lemma~\ref{lem:level-set-augmented-lagrangian-func}, and Theorems~\ref{thm:output-alg1}, \ref{thm:out-itr-cmplxity-1}, \ref{thm:total-iter-cmplxity}, and \ref{thm:total-iter-cmplxity2}. 

Let us start with the following lemma concerning some properties of the $\vartheta$-LHSC barrier function.

\begin{lemma}\label{lem:ppty-bar}
Let $x\in\rmint\cK$ and $\beta\in(0,1)$ be given. Then the following statements hold for the $\vartheta$-LHSC barrier function $B$.
\begin{enumerate}[{\rm (i)}]
\item $(\|\nabla B(x)\|_x^*)^2=-x^T\nabla B(x)=\|x\|_x^2=\vartheta$.
\item $-\nabla B(x)\in\rmint\cK^*$.
\item $\{y:\|y-x\|_x<1\}\subset\rmint\cK$.
\item For any $y$ satisfying $\|y-x\|_x\le\beta$, it holds that
\beqa
(1-\beta)\|v\|_x \le\|v\|_y\le(1-\beta)^{-1}\|v\|_x,\quad \forall v\in\bR^n, \label{ineq:2-local-pnorm} \\
(1-\beta)\|v\|_x^*\le\|v\|_y^*\le(1-\beta)^{-1}\|v\|_x^*,\quad \forall v\in\bR^n. \label{ineq:2-local-norm}
\eeqa
\item $\{s:\|s+\nabla B(x)\|_x^*\le 1\}\subseteq\cK^*$.
\item $\|\nabla^2 B(y)-\nabla^2 B(x)\|_x^*\le \frac{2-\beta}{(1-\beta)^2}\|y-x\|_x$ holds for all $y$ with $\|y-x\|_x\le\beta$. 
\end{enumerate}
\end{lemma}

\begin{proof}
The proof of statements (i)-(v) can be found in \cite[Lemma~1]{HL21bar}.

We next prove statement (vi). Let $y$ be such that $\|y-x\|_x\le\beta$. It follows from \cite[Theorem 2.1.1]{NN94} that
\begin{equation}\label{loc-Lip-bar}
(1-\|y-x\|_x)^2 I \preceq \nabla^2 B(x)^{-1/2} \nabla^2 B(y) \nabla^2 B(x)^{-1/2} \preceq(1-\|y-x\|_x)^{-2} I.
\end{equation}
By \eqref{M-norm}, \eqref{loc-Lip-bar}, and $\|y-x\|_x\le\beta$, one has
\[
\begin{array}{rcl}
\|\nabla^2 B(y)-\nabla^2 B(x)\|_x^*
&\overset{\eqref{M-norm}}{=}&\max_{\|u\|\le1}\|\nabla^2 B(x)^{-1/2}(\nabla^2 B(y)-\nabla^2 B(x))\nabla^2 B(x)^{-1/2}u\|\\[6pt]
&=&\|\nabla^2 B(x)^{-1/2}\nabla^2 B(y)\nabla^2 B(x)^{-1/2}-I\| \\[4pt]
&\overset{\eqref{loc-Lip-bar}}{\le} & \max\{1-(1-\|y-x\|_x)^2,(1-\|y-x\|_x)^{-2}-1\} = (1-\|y-x\|_x)^{-2}-1\\[4pt]
&=&\frac{2-\|y-x\|_x}{(1-\|y-x\|_x)^{2}}\|y-x\|_x \ \le\  \frac{2-\beta}{(1-\beta)^2}\|y-x\|_x,
\end{array}
\]
where the last inequality is due to $\|y-x\|_x\le\beta$. 
Hence, statement (vi) holds as desired.
%
\end{proof}

\subsection{Proof of the main results in Section~\ref{sec:opt}}\label{sec:proof-sec3}
In this subsection we provide a proof of Theorems~\ref{thm:1stopt}.

\begin{proof}[Proof of Theorem \ref{thm:1stopt}]
By $M\in\nabla^{-2}B(x^*)$, the full column rank of $M^{1/2}\nabla c(x^*)$, and also the discussion in Section \ref{sec:opt}, one knows that Robinson's constraint qualification holds at $x^*$. Since $x^*$ is a local minimizer of \eqref{model:equa-cnstr}, it then follows from \cite[Theorem 3.38]{R11nopt} that there exists some $\lambda^*\in\bR^m$ such that 
\begin{equation}\label{class-1stopt}
\nabla f(x^*) + \nabla c(x^*) \lambda^* \in-\cN_{\cK}(x^*).
\end{equation}
Further, by \cite[Proposition 1]{HL21bar}, one knows that \eqref{class-1stopt} holds if and only if \eqref{1stopt-cond-1} and \eqref{1stopt-cond-2} hold.  Consequently, \eqref{1stopt-cond-1} and \eqref{1stopt-cond-2} hold as desired.

We next prove \eqref{2ndopt-cond}. It follows from Lemma \ref{lem:ginvH} that $\{x^*+M^{1/2}d:\|d\|<1\}\subseteq\cK$. Using this and the fact that $x^*$ is a local minimizer of \eqref{model:equa-cnstr}, we see that $d^*=0$ is a local minimizer of the problem
\begin{equation}\label{piece-res-pb}
\min_{d}\ \left\{f(x^*+M^{1/2}d):c(x^*+M^{1/2}d)=0\right\}.
\end{equation}
In addition, since $M^{1/2}\nabla c(x^*)$ has full column rank,  it is clear that LICQ holds at $d^*=0$ for \eqref{piece-res-pb}. By the first- and second-order optimality conditions of \eqref{piece-res-pb} at $d^*=0$, there exists some $\tilde{\lambda}^*\in\bR^m$ such that
\begin{eqnarray}
&&M^{1/2}(\nabla f(x^*) + \nabla c(x^*)\tilde{\lambda}^*)=0,\label{piece-1stopt}\\
&&d^TM^{1/2}\left(\nabla^2 f(x^*) + \sum_{i=1}^m\tilde{\lambda}^*_i\nabla^2 c_i(x^*)\right)M^{1/2}d\ge 0,\quad \forall d\in\{d:\nabla c(x^*)^TM^{1/2}d=0\}.\label{piece-2ndopt}
\end{eqnarray}
In view of \eqref{1stopt-cond-2}, \eqref{piece-1stopt}, and the fact that $M^{1/2}\nabla c(x^*)$ has full column rank, one can see that $\tilde{\lambda}^*=\lambda^*$. Using this and \eqref{piece-2ndopt}, we conclude that \eqref{2ndopt-cond} holds. 
\end{proof}

\subsection{Proof of the main results in Section~\ref{sec:sbpb-solver}}\label{sec:proof-sec4}

In this subsection we first establish several technical lemmas and then use them to prove Theorem~\ref{thm:NCG-iter-oper-cmplxity}. 

{To proceed, by \eqref{Mt} and the definitions of local norms, one can verify that}
\begin{equation}\label{Mt-loc-norm}
{\|d\| = \|M_t d\|_{x^t},\quad \|d\|_{x^t}^* = \|M_t d\|,\quad \|H\|_{x^t}^*=\|M_t H M_t\|,\quad \forall d\in\bR^n, H\in\bR^{n\times n}.}
\end{equation}
{ In addition, as} a consequence of Assumption~\ref{asp:NCG-cmplxity}(b) and Lemma~\ref{lem:ppty-bar}(vi), one can observe that $\phi_\mu$ is locally Lipschitz continuous in $\Omega$ with respect to the local norms, i.e.,
\begin{equation}\label{phimu-Lip}
\|\nabla^2 \phi_\mu(y)-\nabla^2 \phi_\mu(x)\|_x^*\le L_{H}^\phi\|y-x\|_x,\quad \forall x,y\in\Omega\text{ with }\|y-x\|_x\le\beta,
\end{equation}
where
\begin{equation}\label{bar-Hes-Lip}
L_H^\phi := L_H^F + \mu(2-\beta)/(1-\beta)^2.
\end{equation}
The following lemma directly follows from \eqref{phimu-Lip}. Its proof can be found in \cite[Lemma~3]{HL21bar}.

\begin{lemma}
Under Assumption~\ref{asp:NCG-cmplxity}(b), the following inequalities hold:
\begin{equation}
\|\nabla \phi_\mu(y)-\nabla \phi_\mu(x)-\nabla^2 \phi_\mu(x)(y-x)\|_x^*\le\frac{1}{2}L_H^\phi\|y-x\|_x^2,\  \forall x,y\in\Omega\text{ with }\|y-x\|_x\le\beta, \label{apx-nxt-grad}
\end{equation}
\begin{equation}
\phi_\mu(y)\le \phi_\mu(x) + \nabla \phi_\mu(x)^T(y-x) + \frac{1}{2}(y-x)^T\nabla^2 \phi_\mu(x) (y-x) + \frac{1}{6}L_H^\phi\|y-x\|_x^3,\  \forall x,y\in\Omega\text{ with }\|y-x\|_x\le\beta, \label{desc-ineq}
\end{equation}
where $\Omega$ and $L_H^\phi$ are given in Assumption~\ref{asp:NCG-cmplxity}(b) and \eqref{bar-Hes-Lip}, respectively.
\end{lemma}

The following lemma shows that all iterates generated by Algorithm~\ref{alg:NCG} lie in $\cS$. 
\begin{lemma}\label{lem:NCG-set}
Suppose that Assumption~\ref{asp:NCG-cmplxity} holds. Let $\{x^t\}_{t\in\bT}$ be all the iterates generated by Algorithm~\ref{alg:NCG}, where $\bT$ is a subset of consecutive nonnegative integers starting from $0$. Then $\{x^t\}_{t\in\bT}\subset\cS$, where $\cS$ is defined in \eqref{b-subpb-set}.
\end{lemma}
\begin{proof}
We first prove $\{x^t\}_{t\in\bT}\subset\rmint\cK$ by induction. Observe from Algorithm~\ref{alg:NCG} that $x^0=u^0\in\rmint\cK$. Suppose that $x^t\in\rmint\cK$ is generated at iteration $t$ of Algorithm~\ref{alg:NCG} and $x^{t+1}$ is generated at iteration $t+1$. We next prove $x^{t+1}\in\rmint\cK$. Indeed, observe from Algorithm~\ref{alg:NCG} that $x^{t+1}=x^t + \alpha_tM_td^t$ with $\alpha_t\in(0,1]$ and $d^t$ given in one of \eqref{dk-nc}-\eqref{dk-2nd-nc}. It follows from \eqref{dk-nc}-\eqref{dk-2nd-nc} that $\|d^t\|\le\beta$. By these and and {the first relation in \eqref{Mt-loc-norm}}, one has
\begin{equation}\label{local-dist-twoiter}
\|x^{t+1}-x^t\|_{x^t}=\alpha_t\|M_td^t\|_{x^t}\le\|M_td^t\|_{x^t}\overset{\eqref{Mt-loc-norm}}{=}\|d^t\|\le\beta,
\end{equation}
which, along with $x^t\in\rmint\cK$, $\beta<1$ and Lemma~\ref{lem:ppty-bar}(iii), implies that $x^{t+1}\in\rmint\cK$. Hence, the induction is completed, and we have $\{x^t\}_{t\in\bT}\subset\rmint\cK$. 

In addition, observe from Algorithm~\ref{alg:NCG} that $\{\phi_\mu(x^t)\}_{t\in\bT}$ is descent. By this, $x^0=u^0$, $\{x^t\}_{t\in\bT}\subset\rmint\cK$, and \eqref{b-subpb-set}, one can see that $\{x^t\}_{t\in\bT}\subset\cS$.
\end{proof}

The following lemma provides some properties of the output of Algorithm~\ref{alg:capped-CG}, whose proof is similar to the ones of \cite[Lemma~3]{ROW20} and \cite[Lemma~7]{OW21} and thus omitted here.

\begin{lemma}\label{lem:SOL-NC-ppty}
Suppose that Assumption~\ref{asp:NCG-cmplxity} holds and the direction $d^t$ results from the output $\hat{d}^t$ of Algorithm \ref{alg:capped-CG} with a type specified in d$\_$type  at some iteration $t$ of Algorithm \ref{alg:NCG}. Let $M_t$ be given in \eqref{Mt} and $\gamma_t:=\max\{\|\hat{d}^t\|/\beta,1\}$. Then the following statements hold.
\begin{enumerate}[{\rm (i)}]
\item If d$\_$type=SOL, then  $d^t$ satisfies
\beqa
&\epsilon_H\|d^t\|^2\le (d^t)^T\left(M_t^T\nabla^2 \phi_\mu(x^t)M_t+2\epsilon_HI\right)d^t,\label{SOL-ppty-1}\\
&\|d^t\|\le1.1 \epsilon_H^{-1}\|M_t^T\nabla \phi_\mu(x^t)\|,\label{SOL-ppty-2}\\
&(d^t)^TM_t^T\nabla \phi_\mu(x^t)=-\gamma_t(d^t)^T\left(M_t^T\nabla^2\phi_\mu(x^t)M_t+2\epsilon_HI\right)d^t. \label{SOL-ppty-3}
\eeqa
If $\|\hat{d}^t\|\le\beta$, then $d^t$ also satisfies
\beq
\|(M_t^T\nabla^2 \phi_\mu(x^t)M_t+2\epsilon_HI)d^t+ M_t^T\nabla\phi_\mu(x^t)\|\le \epsilon_H\zeta\|d^t\|/2.\label{SOL-ppty-4}
\eeq
\item If d$\_$type=NC, then  $d^t$ satisfies $(d^t)^TM_t^T\nabla\phi_\mu(x^t)\le0$ and
\begin{equation}\label{NC-ppty}
\frac{(d^t)^TM_t^T\nabla^2\phi_\mu(x^t)M_td^t}{\|d^t\|^2} \le -\|d^t\|\le-\epsilon_H.
\end{equation}
\end{enumerate}
\end{lemma}

The following lemma shows that when the search direction $d^t$ in Algorithm \ref{alg:NCG} is of type `SOL', the line search step results in a sufficient reduction on $\phi_\mu$. 

\begin{lemma}\label{lem:sol}
Suppose that Assumption~\ref{asp:NCG-cmplxity} holds and the direction $d^t$ results from the output $\hat{d}^t$ of Algorithm~\ref{alg:capped-CG} with d$\_$type=SOL at some iteration $t$ of Algorithm~\ref{alg:NCG}. Then the following statements hold.
\begin{enumerate}[{\rm (i)}]
\item The step length $\alpha_t$ is well-defined, and moreover,
\begin{equation}\label{lwbd-ak-sol}
\alpha_t\ge\min\left\{1,\sqrt{\frac{\min\{6(1-\eta),2\}}{1.1 [L_H^F+\mu(2-\beta)/(1-\beta)^2] (U_g^F+\mu\sqrt{\vartheta})}}\theta\epsilon_H\right\}.
\end{equation}
\item The next iterate $x^{t+1}=x^t+\alpha_tM_td^t$ satisfies
\begin{equation}\label{suff-desc-sol}
\phi_\mu(x^t)-\phi_\mu(x^{t+1})\ge c_{\rm sol}\min\{(\|\nabla \phi_\mu(x^{t+1})\|_{x^{t+1}}^*)^2\epsilon_H^{-1},\epsilon_H^3\},
\end{equation}
where $M_t$ and $c_{\rm sol}$ are given in \eqref{Mt} and \eqref{csol}, respectively.
\end{enumerate}
\end{lemma}

\begin{proof}
Notice from Lemma~\ref{lem:NCG-set} that $x^t\in\cS$, that is, $x^t\in\rmint\cK$ and $\phi_\mu(x^t)\le\phi_\mu(u^0)$. It then follows from \eqref{Mt}, \eqref{Fbd-b-subpb} and Lemma~\ref{lem:ppty-bar}(i) that
\begin{equation}
\|M_t^T\nabla\phi_\mu(x^t)\|=\|\nabla\phi_\mu(x^t)\|^*_{x^t}\le\|\nabla F(x^t)\|^*_{x^t}+\mu\|\nabla B(x^t)\|^*_{x^t}\le  U_g^F+\mu\sqrt{\vartheta}.\label{Ug-bound}   
\end{equation}
Since $d^t$ results from the output of Algorithm~\ref{alg:capped-CG} with d$\_$type=SOL, one can see that $\|M_t^T\nabla\phi_\mu(x^t)\|>\epsilon_g$ and the relations 
\eqref{SOL-ppty-1}-\eqref{SOL-ppty-3} hold. Also, one can observe from Algorithm~\ref{alg:capped-CG} that its output $\hd^t$ satisfies
\[
\|(M_t^T\nabla^2 \phi_\mu(x^t)M_t+2\epsilon_HI)\hd^t+ M_t^T\nabla\phi_\mu(x^t)\|\le \hat\zeta\|M_t^T\nabla\phi_\mu(x^t)\|
\] 
for some $\hat\zeta\in(0,1/6)$, which together with $\|M_t^T\nabla\phi_\mu(x^t)\|>\epsilon_g$ implies that $\hd^t\neq 0$. It then follows from this and \eqref{dk-sol} that $d^t \neq 0$.

We first prove statement (i). If \eqref{ls-sol} holds for $j=0$, then $\alpha_t=1$, which clearly implies that \eqref{lwbd-ak-sol} holds. We now suppose that \eqref{ls-sol} fails for $j=0$. Claim that for all $j\ge0$ that violate \eqref{ls-sol}, it holds that
\begin{equation}\label{lwbd-sol-ak}
\theta^{2j}\ge\min\{6(1-\eta),2\}\epsilon_H(L^\phi_H)^{-1}\|d^t\|^{-1},
\end{equation}
where $L_H^\phi$ is defined in \eqref{bar-Hes-Lip}. Indeed, we suppose that \eqref{ls-sol} is violated by some $j\ge 0$. We next show that \eqref{lwbd-sol-ak} holds for such $j$ by considering two separate cases below.
	
Case 1) $\phi_\mu(x^t+\theta^j M_td^t)>\phi_\mu(x^t)$. Let $\varphi(\alpha)=\phi_\mu(x^t+\alpha M_t d^t)$. Then $\varphi(\theta^j)>\varphi(0)$. In addition, by \eqref{SOL-ppty-1}, \eqref{SOL-ppty-3}, $\gamma_t=\max\{\|\hat{d}^t\|/\beta,1\}\ge 1$, and $d^t\neq0$, one has
\[
\varphi^\prime(0)= (d^t)^TM_t^T\nabla \phi_\mu(x^t) \overset{\eqref{SOL-ppty-3}}{=} -\gamma_t(d^t)^T(M_t^T\nabla^2 \phi_\mu(x^t)M_t+2\epsilon_HI)d^t\overset{\eqref{SOL-ppty-1}}{\le}-\gamma_t\epsilon_H\|d^t\|^2< 0.
\]
In view of these, we can observe that there exists a local minimizer $\alpha^*\in(0,\theta^j)$ of $\varphi$ such that $\varphi(\alpha^*)<\varphi(0)$ and
\begin{equation}\label{eq:gv-as}
\varphi^\prime(\alpha^*)=\nabla \phi_\mu(x^t+\alpha^*M_td^t)^TM_td^t=0.    
\end{equation}
By $\phi_\mu(x^t)\le\phi_\mu(u^0)$ and $\varphi(\alpha^*)<\varphi(0)$, one has $\phi_\mu(x^t+\alpha^*M_td^t)< \phi_\mu(u^0)$. In addition, using \eqref{local-dist-twoiter} and $0<\alpha^*<\theta^j\le 1$, we have $\|\alpha^*M_td^t\|_{x^t}\le \|M_td^t\|_{x^t}\le\beta$. Hence, \eqref{apx-nxt-grad} holds for $x=x^t$ and $y=x^t+\alpha^*M_td^t$. By this, \eqref{Mt-loc-norm}, \eqref{SOL-ppty-1}, \eqref{SOL-ppty-3}, \eqref{eq:gv-as}, $0<\alpha^*< 1$ and $\gamma_t\ge 1$, one has
\[
\begin{array}{rcl}
\frac{(\alpha^*)^2L_H^\phi}{2}\|d^t\|^3&\overset{\eqref{Mt-loc-norm}}{=}&\frac{(\alpha^*)^2L_H^\phi}{2}\|d^t\|\|M_td^t\|_{x^t}^2\overset{\eqref{apx-nxt-grad}}{\ge}\|d^t\|\|\nabla \phi_\mu(x^t+\alpha^*M_td^t)-\nabla \phi_\mu(x^t)-\alpha^*\nabla^2 \phi_\mu(x^t)M_td^t\|_{x^t}^*\\[6pt]
&\ge&(d^t)^T(M_t^T\nabla \phi_\mu(x^t+\alpha^*M_td^t)-M_t^T\nabla \phi_\mu(x^t)-\alpha^*M_t^T\nabla^2 \phi_\mu(x^t)M_td^t)\\[6pt]
&\overset{\eqref{eq:gv-as}}{=}&-(d^t)^T M_t^T\nabla \phi_\mu(x^t)-\alpha^*(d^t)^T M_t^T\nabla^2 \phi_\mu(x^t)M_t d^t\\[4pt]
&\overset{\eqref{SOL-ppty-3}}{=}&(\gamma_t-\alpha^*)(d^t)^T(M_t^T\nabla^2 \phi_\mu(x^t)M_t+2\epsilon_HI)d^t +2 \alpha^*\epsilon_H\|d^t\|^2\\[4pt]
&\overset{\eqref{SOL-ppty-1}}{\ge}&(\gamma_t-\alpha^*)\epsilon_H\|d^t\|^2+2 \alpha^*\epsilon_H\|d^t\|^2 =  (\gamma_t+\alpha^*)\epsilon_H\|d^t\|^2\ge \epsilon_H \|d^t\|^2,
\end{array}    
\]
which along with $d^t\neq 0$ implies that $(\alpha^*)^2\ge 2\epsilon_H(L_H^\phi)^{-1}\|d^t\|^{-1}$. Using this and $\theta^{j}>\alpha^*$, we conclude that \eqref{lwbd-sol-ak} holds in this case.
	
Case 2) $\phi_\mu(x^t+\theta^jM_td^t)\le \phi_\mu(x^t)$. By this and $\phi_\mu(x^t)\le\phi_\mu(u^0)$, one has $\phi_\mu(x^t+\theta^jM_td^t)\le \phi_\mu(u^0)$. Also, using \eqref{local-dist-twoiter} and $\theta\in(0,1)$, we have $\|\theta^jM_td^t\|_{x^t}\le \|M_td^t\|_{x^t}\le\beta$. Hence, \eqref{desc-ineq} holds for $x=x^t$ and $y=x^t+\theta^jM_td^t$. Using this, \eqref{Mt-loc-norm}, \eqref{SOL-ppty-1}, \eqref{SOL-ppty-3} and the fact that $j$ violates \eqref{ls-sol}, we obtain that
\begin{align}
-\eta\epsilon_H\theta^{2j}\|d^t\|^2&\le \phi_\mu(x^t+\theta^jM_td^t)-\phi_\mu(x^t)\nn \\
&\overset{\eqref{desc-ineq}}{\le} \theta^j \nabla \phi_\mu(x^t)^TM_td^t + \frac{\theta^{2j}}{2}(d^t)^TM_t^T\nabla^2 \phi_\mu(x^t)M_t d^t + \frac{L_H^\phi}{6}\theta^{3j}\|M_td^t\|_{x^t}^3 \nn \\
&\overset{\eqref{Mt-loc-norm}\eqref{SOL-ppty-3}}{=}-\theta^j\gamma_t (d^t)^T(M_t^T\nabla^2 \phi_\mu(x^t)M_t +2\epsilon_H I)d^t + \frac{\theta^{2j}}{2}(d^t)^TM_t^T\nabla^2 \phi_\mu(x^t)M_t d^t + \frac{L_H^\phi}{6}\theta^{3j}\|d^t\|^3 \nn \\
&=-\theta^j\left(\gamma_t-\frac{\theta^j}{2}\right)(d^t)^T(M_t^T\nabla^2 \phi_\mu(x^t)M_t +2\epsilon_H I)d^t - \theta^{2j} \epsilon_H\|d^t\|^2 + \frac{L_H^\phi}{6}\theta^{3j}\|d^t\|^3 \nn \\
&\overset{\eqref{SOL-ppty-1}}{\le} -\theta^j\left(\gamma_t-\frac{\theta^j}{2}\right) \epsilon_H\|d^t\|^2 - \theta^{2j}\epsilon_H\|d^t\|^2 + \frac{L_H^\phi}{6}\theta^{3j}\|d^t\|^3 \le -\theta^j \epsilon_H\gamma_t\|d^t\|^2 + \frac{L_H^\phi}{6}\theta^{3j}\|d^t\|^3, \label{ineq:desc-ls-sol}
\end{align}
where the first inequality is due to the violation of \eqref{ls-sol} by such $j$. Recall that $d^t \neq 0$. Dividing both sides of \eqref{ineq:desc-ls-sol} by $L_H^\phi\theta^j\|d^t\|^3/6$ and using $\eta, \theta\in(0,1)$ and $\gamma_t\ge 1$, we have
\[
\theta^{2j}\ge 6(\gamma_t-\eta\theta^j)\epsilon_H(L_H^\phi)^{-1}\|d^t\|^{-1}\ge6(1-\eta)\epsilon_H(L_H^\phi)^{-1}\|d^t\|^{-1}.
\]
Hence, \eqref{lwbd-sol-ak} also holds in this case.
	
Combining the above two cases, we conclude that  \eqref{lwbd-sol-ak} holds for any $j\ge0$ violating \eqref{ls-sol}. By this and $\theta\in(0,1)$, one can see that all $j\ge0$ that violate \eqref{ls-sol} must be bounded above. It then follows that the step length $\alpha_t$ associated with \eqref{ls-sol} is well-defined. We next prove \eqref{lwbd-ak-sol}. Observe from the definition of $j_t$ in Algorithm~\ref{alg:NCG} that $j=j_t-1$ violates \eqref{ls-sol} and hence \eqref{lwbd-sol-ak} holds for $j=j_t-1$. Then, by \eqref{bar-Hes-Lip}, \eqref{lwbd-sol-ak} with $j=j_t-1$, and $\alpha_t = \theta^{j_t}$, one has
\beqa
\alpha_t &=& \theta^{j_t} \ge\sqrt{\min\{6(1-\eta),2\}\epsilon_H(L^\phi_H)^{-1}} \theta\|d^t\|^{-1/2}\nn \\[4pt]
&=&\sqrt{\min\{6(1-\eta),2\}\epsilon_H[L_H^F+\mu(2-\beta)/(1-\beta)^2]^{-1}} \theta\|d^t\|^{-1/2}, \label{lwbd-sol-ak-1}
\eeqa
which along with \eqref{SOL-ppty-2} and \eqref{Ug-bound} implies that \eqref{lwbd-ak-sol} holds.
	
We next prove statement (ii), particularly,  \eqref{suff-desc-sol} by considering three separate cases below.
	
Case 1) $\alpha_t=1$ and $\|\hat{d}^t\|\ge\beta$. It then follows from \eqref{dk-sol} that $d^t=\beta \hat{d}^t/\|\hat{d}^t\|$. Notice from Algorithm~\ref{alg:NCG} that $\beta\ge\epsilon_H$. Using this and $d^t=\beta \hat{d}^t/\|\hat{d}^t\|$, we see that $\|d^t\|=\beta\ge\epsilon_H$, which together with \eqref{ls-sol} and $\alpha_t=1$ implies that \eqref{suff-desc-sol} holds.
	
Case 2) $\alpha_t=1$ and $\|\hat{d}^t\|<\beta$. Notice from $\alpha_t=1$ that $j=0$ is accepted by \eqref{ls-sol}. Then one can see that $\phi_\mu(x^t+M_td^t)\le \phi_\mu(x^t)\le \phi_\mu(u^0)$. Also, observe from \eqref{local-dist-twoiter} that $\|M_td^t\|_{x^t}\le\beta$. Hence, \eqref{apx-nxt-grad} holds for $x=x^t$ and $y=x^t+M_td^t$.  By these, \eqref{ineq:2-local-norm} and \eqref{SOL-ppty-4}, one has
\begin{align*}
(1-\beta)\|\nabla \phi_\mu(x^{t+1})\|_{x^{t+1}}^*&\overset{\eqref{ineq:2-local-norm}}{\le}\|\nabla \phi_\mu(x^{t+1})\|_{x^t}^*=\|\nabla \phi_\mu(x^{t}+M_td^t)\|_{x^{t}}^*\\
&\le\|\nabla \phi_\mu(x^t+M_td^t)-\nabla \phi_\mu(x^t)-\nabla^2 \phi_\mu(x^t)M_td^t\|_{x^t}^*+\|\nabla \phi_\mu(x^t)+\nabla^2 \phi_\mu(x^t)M_td^t\|_{x^t}^*\\
&=\|\nabla \phi_\mu(x^t+M_td^t)-\nabla \phi_\mu(x^t)-\nabla^2 \phi_\mu(x^t)M_td^t\|_{x^t}^*+\|M_t^T(\nabla \phi_\mu(x^t)+\nabla^2 \phi_\mu(x^t)M_td^t)\|\\
&\le\|\nabla \phi_\mu(x^t+M_td^t)-\nabla \phi_\mu(x^t)-\nabla^2 \phi_\mu(x^t)M_td^t\|_{x^t}^*\\
&\quad+\|(M_t^T\nabla^2 \phi_\mu(x^t)M_t+2\epsilon_HI)d^t+M_t^T\nabla \phi_\mu(x^t)\|+2\epsilon_H\|d^t\|\\
&\overset{\eqref{apx-nxt-grad}\eqref{SOL-ppty-4}}\le L_H^\phi\|M_td^t\|_{x^t}^2/2 + (4+\zeta)\epsilon_H\|d^t\|/2\overset{\eqref{Mt-loc-norm}}= L_H^\phi\|d^t\|^2/2 + (4+\zeta)\epsilon_H\|d^t\|/2,
\end{align*}
where the second inequality is due to the triangle inequality, and the second equality follows from \eqref{M-norm} and {the second relation in \eqref{Mt-loc-norm}}. Solving the above inequality for $\|d^t\|$ and using \eqref{bar-Hes-Lip} and the fact that $\|d^t\|>0$, we obtain that
\[
\begin{array}{rcl}
\|d^t\|&\ge&\frac{-(4+\zeta)\epsilon_H + \sqrt{(4+\zeta)^2\epsilon_H^2+8(1-\beta)L_H^\phi\|\nabla \phi_\mu(x^{t+1})\|_{x^{t+1}}^*}}{2L_H^\phi}\\ [6pt]
&\ge&\frac{-(4+\zeta)\epsilon_H + \sqrt{(4+\zeta)^2\epsilon_H^2+8(1-\beta)L_H^\phi\epsilon_H^2}}{2L_H^\phi}\min\{\|\nabla \phi_\mu(x^{t+1})\|_{x^{t+1}}^*\epsilon_H^{-2},1\}\\[6pt]
&=&\frac{4(1-\beta)}{4+\zeta+\sqrt{(4+\zeta)^2+8(1-\beta)L_H^\phi}}\min\{\|\nabla \phi_\mu(x^{t+1})\|_{x^{t+1}}^*\epsilon_H^{-1},\epsilon_H\}\\[6pt]
&\overset{\eqref{bar-Hes-Lip}}{=}&\frac{4(1-\beta)}{4+\zeta+\sqrt{(4+\zeta)^2+8[(1-\beta)L_H^F + \mu(2-\beta)/(1-\beta)]}}\min\{\|\nabla \phi_\mu(x^{t+1})\|_{x^{t+1}}^*\epsilon_H^{-1},\epsilon_H\},
\end{array}
\]
where the second inequality follows from the inequality $-a+\sqrt{a^2+bs}\ge(-a+\sqrt{a^2+b})\min\{s,1\}$ for all $a,b,s\ge0$, which can be easily verified by performing a rationalization to the terms $-a+\sqrt{a^2+bs}$ and $-a+\sqrt{a^2+b}$, respectively.  In view of this, $\alpha_t=1$, \eqref{ls-sol} and \eqref{csol}, one can see that \eqref{suff-desc-sol} holds.

Case 3) $\alpha_t<1$. By this, one has that $j=0$ violates \eqref{ls-sol} and hence \eqref{lwbd-sol-ak} holds for $j=0$. Letting $j=0$ in \eqref{lwbd-sol-ak}, we obtain that $\|d^t\|\ge \min\{6(1-\eta),2\}\epsilon_H/L_H^\phi$, which along with \eqref{ls-sol}, \eqref{bar-Hes-Lip} and \eqref{lwbd-sol-ak-1} implies that
\[
\phi_\mu(x^t)-\phi_\mu(x^{t+1}) \overset{\eqref{ls-sol}}{\ge}\eta \epsilon_H\theta^{2j_t}\|d^t\|^2 \ge \eta\left[\frac{\min\{6(1-\eta),2\}\theta}{L_H^\phi}\right]^2\epsilon_H^3\overset{\eqref{bar-Hes-Lip}}{=}\eta\left[\frac{\min\{6(1-\eta),2\}\theta}{L_H^F + \mu(2-\beta)/(1-\beta)^2}\right]^2\epsilon_H^3.
\]
By this and \eqref{csol}, one can immediately see that \eqref{suff-desc-sol} also holds in this case.
\end{proof}

The next lemma shows that when the search direction $d^t$ in Algorithm \ref{alg:NCG} is of type `NC', the line search step results in a sufficient reduction on $\phi_\mu$ as well.

\begin{lemma}\label{lem:nc}
Suppose that Assumption \ref{asp:NCG-cmplxity} holds and the direction $d^t$ results from either the output $\hat{d}^t$ of Algorithm~\ref{alg:capped-CG} with d$\_$type=NC or the output $v$ of Algorithm~\ref{pro:meo} at some iteration $t$ of Algorithm~\ref{alg:NCG}. Then the following statements hold.
\begin{enumerate}[{\rm (i)}]
\item The step length $\alpha_t$ is well-defined, and moreover,
\begin{equation}\label{lwbd-ak-nc}
\alpha_t\ge \min\left\{1,\frac{\min\{1,3(1-\eta)\}\theta}{L_H^F+\mu(2-\beta)/(1-\beta)^2}\right\}.  
\end{equation}
\item The next iterate $x^{t+1}=x^t+\alpha_tM_td^t$ satisfies $\phi_\mu(x^t)-\phi_\mu(x^{t+1})\ge c_{\rm nc}\epsilon_H^3$, where $M_t$ and $c_{\rm nc}$ are given in \eqref{Mt} and \eqref{cnc}, respectively.
\end{enumerate}
\end{lemma}

\begin{proof}
It follows from Lemma~\ref{lem:NCG-set} that $x^t\in\cS$, that is, $x^t\in\rmint\cK$ and $\phi_\mu(x^t)\le\phi_\mu(u^0)$. By the assumption on $d^t$, one can see from Algorithm \ref{alg:NCG} that $d^t$ is a negative curvature direction given in \eqref{dk-nc} or \eqref{dk-2nd-nc} and thus $d^t\neq 0$. Also, the vector $v$ satisfies $\|v\|=1$ whenever it is returned from Algorithm~\ref{pro:meo}.  By these, Lemma~\ref{lem:SOL-NC-ppty}(ii), \eqref{dk-nc} and \eqref{dk-2nd-nc}, one has
\begin{equation}\label{dk-Hk-dk3}
\nabla \phi_\mu(x^t)^TM_td^t\le 0, \quad (d^t)^TM_t^T\nabla^2 \phi_\mu(x^t)M_td^t \leq -\|d^t\|^3<0.
\end{equation}
	
We first prove statement (i). If \eqref{ls-nc} holds for $j=0$, then $\alpha_t=1$, which clearly implies that \eqref{lwbd-ak-nc} holds. We now suppose that \eqref{ls-nc} fails for $j=0$. Claim that for all $j\ge0$ that violate \eqref{ls-nc}, it holds that
\begin{equation}\label{lwbd-at-nc}
\theta^j\ge \min\{1,3(1-\eta)\}/L_H^\phi,
\end{equation}
where $L_H^\phi$ is defined in \eqref{bar-Hes-Lip}. Indeed, suppose that \eqref{ls-nc} is violated by some $j\ge0$. We now prove that \eqref{lwbd-at-nc} holds for such $j$ by considering two separate cases below.
	
Case 1) $\phi_\mu(x^t+\theta^jM_td^t)>\phi_\mu(x^t)$. Let $\varphi(\alpha)=\phi_\mu(x^t+\alpha M_t d^t)$. Then $\varphi(\theta^j)>\varphi(0)$. Also, by \eqref{dk-Hk-dk3}, one has
\[
\varphi^\prime(0)=\nabla \phi_\mu(x^t)^TM_td^t\le 0,\quad \varphi^{\prime\prime}(0)=(d^t)^TM_t^T\nabla^2 \phi_\mu(x^t)M_td^t<0.
\]
From these, we can observe that there exists a local minimizer $\alpha^*\in(0,\theta^j)$ of $\varphi$ such that $\varphi(\alpha^*)<\varphi(0)$. By the second-order necessary optimality condition of $\varphi$ at $\alpha^*$, one has
\begin{equation}\label{varphi-pp}
\varphi^{\prime\prime}(\alpha^*)=(d^t)^TM_t^T\nabla^2 \phi_\mu(x^t+\alpha^*M_td^t)M_td^t\ge 0.    
\end{equation}
In addition, by $\phi_\mu(x^t)\le\phi_\mu(u^0)$ and $\varphi(\alpha^*)<\varphi(0)$, one has $\phi_\mu(x^t+\alpha^*M_td^t)<\phi_\mu(u^0)$. Using \eqref{local-dist-twoiter} and $0<\alpha^*<\theta^j\le 1$, we see that $\|\alpha^*M_td^t\|_{x^t}\le \|M_t d^t\|_{x^t}\le\beta$. Hence, \eqref{phimu-Lip} holds for $x=x^t$ and $y=x^t+\alpha^*M_td^t$. Using this, \eqref{M-norm}, \eqref{Mt-loc-norm}, \eqref{phimu-Lip},  \eqref{dk-Hk-dk3} and \eqref{varphi-pp}, we obtain that
\begin{align}
L_H^\phi\alpha^*\|d^t\|^3 & \overset{\eqref{Mt-loc-norm}}{=}L_H^\phi\alpha^*\|d^t\|^2\|M_td^t\|_{x^t}\overset{\eqref{phimu-Lip}}{\ge}\|d^t\|^2\|\nabla^2 \phi_\mu(x^t+\alpha^*M_td^t)-\nabla^2 \phi_\mu(x^t)\|_{x^t}^*\nn \\
& \overset{\eqref{M-norm}}{=}\|d^t\|^2\|M_t^T(\nabla^2 \phi_\mu(x^t+\alpha^*M_td^t)-\nabla^2 \phi_\mu(x^t))M_t\| \ge(d^t)^TM_t^T(\nabla^2 \phi_\mu(x^t+\alpha^*M_td^t)-\nabla^2 \phi_\mu(x^t))M_td^t \nn \\ 
&\overset{\eqref{varphi-pp}}{\geq} -(d^t)^TM_t^T \nabla^2\phi_\mu(x^t)M_td^t \overset{\eqref{dk-Hk-dk3}}{\ge}\|d^t\|^3. \nn
\end{align}
It then follows from this and $d^t\neq 0$ that $\alpha^*\ge 1/L_H^\phi$, which along with $\theta^j>\alpha^*$ implies that \eqref{lwbd-at-nc} holds in this case.

Case 2) $\phi_\mu(x^t+\theta^jM_td^t)\le \phi_\mu(x^t)$. By this and $\phi_\mu(x^t)\le \phi_\mu(u^0)$, one has $\phi_\mu(x^t+\theta^jM_td^t)\le \phi_\mu(u^0)$. In addition, it follows from \eqref{local-dist-twoiter} and $\theta\in(0,1)$ that $\|\theta^jM_td^t\|_{x^t}\le\|M_td^t\|_{x^t}\le\beta$. Hence, \eqref{desc-ineq} holds for $x=x^t$ and $y=x^t+\theta^jM_td^t$. By this, \eqref{Mt-loc-norm}, \eqref{dk-Hk-dk3} and the fact that $j$ violates \eqref{ls-nc}, one has
\begin{equation*}
\begin{array}{ll}
-\frac{\eta}{2}\theta^{2j}\|d^t\|^3&\le \ \phi_\mu(x^t+\theta^jM_t d^t) - \phi_\mu(x^t) \overset{\eqref{desc-ineq}}{\le} \theta^j \nabla \phi_\mu(x^t)^TM_td^t + \frac{\theta^{2j}}{2}(d^t)^TM_t^T\nabla^2 \phi_\mu(x^t) M_t d^t + \frac{L_H^\phi}{6}\theta^{3j}\|M_td^t\|_{x^t}^3 \\[6pt] 
&\overset{\eqref{Mt-loc-norm}\eqref{dk-Hk-dk3}}{\le}-\frac{\theta^{2j}}{2}\|d^t\|^3 + \frac{L_H^\phi}{6}\theta^{3j}\|d^t\|^3,
\end{array}
\end{equation*}
where the first inequality is due to the violation of \eqref{ls-nc} by such $j$. 
Using this and $d^t\neq 0$, we see that $\theta^j\ge 3(1-\eta)/L_H^\phi$. Hence, \eqref{lwbd-at-nc} also holds in this case.
	
Combining the above two cases, we conclude that \eqref{lwbd-at-nc} holds for all $j\ge0$ violating \eqref{ls-nc}. By this and $\theta\in(0,1)$, one can see that all $j\ge0$ that violate \eqref{ls-nc} must be bounded above. It then follows that the step length $\alpha_t$ associated with \eqref{ls-nc} is well-defined. We next derive a lower bound for $\alpha_t$. Notice that $j=j_t-1$ violates \eqref{ls-nc} and hence \eqref{lwbd-at-nc} holds for $j=j_t-1$. Then by \eqref{lwbd-at-nc} with $j=j_t-1$ and $\alpha_t = \theta^{j_t}$, one can observe that $\alpha_t=\theta^{j_t} \ge \min\{1,3(1-\eta)\}\theta/L_H^\phi$, which along with \eqref{bar-Hes-Lip} yields \eqref{lwbd-ak-nc} as desired.

We next prove statement (ii) by considering two separate cases below.

Case 1) $d^t$ results from the output $\hat{d}^t$ of Algorithm~\ref{alg:capped-CG} with d$\_$type=NC. By this and \eqref{NC-ppty}, one has $\|d^t\|\ge\epsilon_H$, which along with statement (i) and \eqref{ls-nc} implies that statement (ii) holds.

Case 2) $d^t$ results from the output $v$ of Algorithm~\ref{pro:meo}. Notice from Algorithm~\ref{pro:meo} that $\|v\|=1$ and $v^TM_t^T\nabla^2\phi_\mu(x^t)M_tv\le-\epsilon_H/2$. It then follows from \eqref{dk-2nd-nc} and $\beta\ge\epsilon_H$ that $\|d^t\|\ge\epsilon_H/2$. Using this, \eqref{ls-nc} and statement (i), we see that statement (ii) also holds in this case.
\end{proof}

We are now ready to prove Theorem~\ref{thm:NCG-iter-oper-cmplxity}.

\begin{proof}[{Proof of Theorem~\ref{thm:NCG-iter-oper-cmplxity}}]	
Notice from Lemma~\ref{lem:NCG-set} that all the iterates generated by Algorithm~\ref{alg:NCG} lie in $\cS$. By this, \eqref{M-norm}, \eqref{Mt-loc-norm} and \eqref{Fbd-b-subpb}, one has
\begin{equation}
\|M_t^T\nabla^2\phi_\mu(x^t)M_t\|\overset{\eqref{M-norm}\eqref{Mt-loc-norm}}{=}\|\nabla^2\phi_\mu(x^t)\|^*_{x^t}\le\|\nabla^2 F(x^t)\|^*_{x^t}+\mu\|\nabla^2 B(x^t)\|^*_{x^t}\le  U_H^F+\mu, \label{UH-bound}    
\end{equation}
where the last inequality follows from \eqref{Fbd-b-subpb} and the fact that $\|\nabla^2 B(x^t)\|^*_{x^t}=1$ due to \eqref{M-norm} and \eqref{Mt}.

(i) Suppose for contradiction that the total number of calls of Algorithm~\ref{pro:meo} in Algorithm~\ref{alg:NCG} is more than $T_2$. Observe from Algorithm~\ref{alg:NCG} and Lemma~\ref{lem:nc}(ii) that each of these calls, except the last one, returns a sufficiently negative curvature direction, and each of them results in a reduction on $\phi_\mu$ at least by $c_{\rm nc} \epsilon_H^3$. Using this and the fact that $x^0=u^0$, we obtain that
\[
T_2 c_{\rm nc} \epsilon_H^3 \le \sum_{t\in\bT} [\phi_\mu(x^t)-\phi_\mu(x^{t+1})] \le \phi_\mu(x^0)-\phil = \phih-\phil, 
\]
where $\bT$ is given in Lemma~\ref{lem:NCG-set}. This contradicts with the definition of $T_2$ given in \eqref{T1}.

(ii) Suppose for contradiction that the total number of calls of Algorithm~\ref{alg:capped-CG} in Algorithm~\ref{alg:NCG} is more than $T_1$. Note that if Algorithm~\ref{alg:capped-CG} is called at some iteration $t$ and generates $x^{t+1}$ satisfying $\|\nabla \phi_\mu(x^{t+1})\|_{x^{t+1}}^*\le\epsilon_g$, then Algorithm~\ref{pro:meo} must be called at the next iteration $t+1$. Using this and statement (i), we see that the total number of such iterations $t$ is at most $T_2$. Hence, the total number of iterations $t$ of Algorithm~\ref{alg:NCG} at which Algorithm~\ref{alg:capped-CG} is called and generates  $x^{t+1}$ satisfying $\|\nabla \phi_\mu(x^{t+1})\|_{x^{t+1}}^*>\epsilon_g$ is at least $T_1-T_2+1$. Moreover, for each of such iterations $t$, it follows from Lemmas~\ref{lem:sol}(ii) and \ref{lem:nc}(ii) that $\phi_\mu(x^t)-\phi_\mu(x^{t+1})\ge \min\{c_{\rm sol},c_{\rm nc}\}\min\{\epsilon_g^2\epsilon_H^{-1},\epsilon_H^3\}$. Thus, one has
\[
(T_1-T_2+1) \min\{c_{\rm sol},c_{\rm nc}\}\min\{\epsilon_g^2\epsilon_H^{-1},\epsilon_H^3\} \le \sum_{t\in\bT}[\phi_\mu(x^t)-\phi_\mu(x^{t+1})]\le \phih-\phil, 
\]
where $\bT$ is given in Lemma~\ref{lem:NCG-set}. This contradicts the definitions of $T_1$ and $T_2$ given in \eqref{T1}.
	
(iii) Notice that either Algorithm~\ref{alg:capped-CG} or Algorithm~\ref{pro:meo} is called at each iteration of Algorithm~\ref{alg:NCG}. By this and statements (i) and (ii), one has that the total number of iterations of Algorithm~\ref{alg:NCG} is at most $T_1+T_2$. In addition, the relation \eqref{NCG-iter} follows from \eqref{T1}, \eqref{csol} and \eqref{cnc}. It is also not hard to see that the output $x^t$ of Algorithm~\ref{alg:NCG} satisfies $\|\nabla \phi_\mu(x^t)\|_{x^t}^*\le \epsilon_g$ deterministically and $\lambda_{\min}(M_t^T\nabla^2\phi_\mu(x^t)M_t)\ge-\epsilon_H$ with probability at least $1-\delta$ for some $0 \le t \le T_1+T_2$, where the probability is due to Algorithm~\ref{pro:meo}. 
Hence, statement (iii) holds.

(iv) Recall that each iteration of Algorithm~\ref{alg:NCG} requires Cholesky factorization of $\nabla^2 B$ at one point only. This together with statement (iii) implies that the total number of Cholesky factorizations required by Algorithm~\ref{alg:NCG} is at most $T_1+T_2$. By \eqref{UH-bound} and Theorem~\ref{lem:capped-CG} with $(H,\varepsilon)=(M_t^T\nabla^2\phi_\mu(x^t)M_t,\epsilon_H)$, one can see that the number of products of $H$ and a vector $v$ required by each call of Algorithm~\ref{alg:capped-CG} is at most $\widetilde{\cO}(\min\{n,[(U_H^F+\mu)/\epsilon_H]^{1/2}\})$. In addition, by \eqref{UH-bound}, Theorem~\ref{rand-Lanczos} with $(H,\varepsilon)=(M_t^T\nabla^2\phi_\mu(x^t)M_t,\epsilon_H)$, and the fact that each iteration of the Lanczos method requires only one product of $H$ and a vector $v$, one can observe that the number of products of $H$ and a vector $v$ required by each call of Algorithm~\ref{pro:meo} is also at most $\widetilde{\cO}(\min\{n,[(U_H^F+\mu)/\epsilon_H]^{1/2}\})$. 
Recall from Section \ref{complex-alg1} that the product of $H$ and a vector $v$ requires at most three fundamental operations, which are one Hessian-vector product of $F$, one backward and forward substitutions to a triangular linear system. Hence, each call of Algorithm~\ref{alg:capped-CG} or \ref{pro:meo} requires at most $\widetilde{\cO}(\min\{n,[(U_H^F+\mu)/\epsilon_H]^{1/2}\})$ fundamental operations. By these observations and statement (iii), we conclude that statement (iv) holds.
\end{proof}

\subsection{Proof of the main results in Section~\ref{sec:AL-method}}\label{sec:proof-sec5}

In this subsection we provide a proof of Lemma~\ref{lem:level-set-augmented-lagrangian-func} and Theorems~\ref{thm:output-alg1}, \ref{thm:out-itr-cmplxity-1}, \ref{thm:total-iter-cmplxity} and \ref{thm:total-iter-cmplxity2}. 

Before proceeding,  we recall that $\|c(z_{\epsilon})\|\le \epsilon/2<1$. Using this, \eqref{ineq:lower-bound-penalty} and  {\eqref{per-eq-constr}}, we obtain that
\begin{equation}\label{balp-lbd}
{
f(x) + \mu B(x) + \gamma \|\tilde{c}(x)\|^2 \ge f(x) + \mu B(x) + \frac{\gamma}{2}\|c(x)\|^2 - \gamma\|c(z_\epsilon)\|^2 \ge \fl - \gamma,\quad \forall x\in\rmint\cK, \mu\in(0,\mu_0],
}
\end{equation}
{where the first inequality follows from \eqref{per-eq-constr} and $\|a-b\|^2\ge\|a\|^2/2-\|b\|^2$ for all $a,b\in\bR^n$.} In addition, by \eqref{ubd}, the first relation in \eqref{algstop:1st-order}, {and the fact that $0<\mu_k\le\mu_0$ (see Algorithm~\ref{alg:2nd-order-AL-nonconvex}),} one has
\beq \label{Lxk-uppbnd}
{
\cL_{\mu_k}(x^{k+1},\lambda^k;\rho_k)\le \fh \ \mbox{whenever} \ x^{k+1} \ \mbox{is generated}.}
\eeq 
We next present an auxiliary lemma that will be frequently used later. Its proof is identical to the one of \cite[Lemma~5.4]{HL21al} with $f$ replaced by $f+\mu B$, and thus omitted here.

\begin{lemma}\label{tech-1} 
Suppose that Assumption \ref{asp:lowbd-knownfeas} holds. 
Let $\gamma$, {$\mu_0$}, $\fh$ and $\fl$ be given in Assumption \ref{asp:lowbd-knownfeas}.  Assume that $\rho>2\gamma$, {$\mu\in(0,\mu_0]$}, $\lambda\in\bR^m$ and $x\in\rmint\cK$ satisfy $\cL_{\mu}(x,\lambda;\rho)\le \fh$, where $\cL_\mu$ is defined in \eqref{def:bar-AL-func}. Then the following statements hold.
\begin{enumerate}[{\rm (i)}]
\item $f(x)+\mu B(x)\le \fh+\|\lambda\|^2/(2\rho)$.
\item $\|\tilde{c}(x)\|\le\sqrt{2(\fh-\fl+\gamma)/(\rho-2\gamma)+\|\lambda\|^2/(\rho-2\gamma)^2}+\|\lambda\|/(\rho-2\gamma)$.
\item If $\rho\ge\|\lambda\|^2/(2\tdf)$ for some $\tdf>0$, then $f(x)+\mu B(x)\le \fh+\tdf$.
\item If $\rho\ge 2(\fh-\fl+\gamma)\tdc^{-2}+ 2\|\lambda\|\tdc^{-1}+2\gamma$ for some $\tdc>0$, then $\|\tilde{c}(x)\|\le\tdc$.
\end{enumerate}
\end{lemma}

The following lemma establishes the local Lipschitz continuity of $c_i$ and $\nabla c_i$ with respect to the local norm.

\begin{lemma}
Under Assumption~\ref{asp:lowbd-knownfeas}, the following inequalities hold:
\begin{eqnarray}
&&|c_i(y)-c_i(x)|\le \frac{U_g^c}{1-\beta}\|y-x\|_x,\quad \forall x,y\in\Omega(\delta_f,\delta_c)\text{ with }\|y-x\|_x\le\beta,\ 1\le i\le m,\label{ci-loc-Lip}\\
&&\|\nabla c_i(y) - \nabla c_i(x)\|_x^*\le\frac{U_H^c}{(1-\beta)^2}\|y-x\|_x,\quad \forall x,y\in\Omega(\delta_f,\delta_c)\text{ with }\|y-x\|_x\le\beta,\ 1\le i\le m,\label{gci-loc-Lip}
\end{eqnarray}
where $\Omega(\delta_f,\delta_c)$ is given in Assumption~\ref{asp:lowbd-knownfeas}, and $U_g^c$, $U_H^c$ are defined in \eqref{g-loc-upbd} and \eqref{H-loc-upbd}, respectively.
\end{lemma}

\begin{proof}
Fix any $x,y\in\Omega(\delta_f,\delta_c)$ with $\|y-x\|_x\le\beta$ and any $1\le i\le m$. Let $z_t=x+t(y-x)$ for all $t\in[0,1]$. It then follows that $\|z_t-x\|_x \le\beta$ and $z_t\in\Omega(\delta_f,\delta_c)$. By these, \eqref{g-loc-upbd},  \eqref{H-loc-upbd}, \eqref{ineq:2-local-pnorm} and \eqref{ineq:2-local-norm}, one has  
\[
\|\nabla c_i(z_t)\|_{z_t}^*\leq U^c_g,\quad
\|\nabla^2 c_i(z_t)\|_{z_t}^*\leq U^c_H, \quad \|v\|_{z_t} \le (1-\beta)^{-1}\|v\|_x, \quad \|v\|^*_x \le (1-\beta)^{-1}\|v\|^*_{z_t}, \quad \forall v\in\bR^n, t\in[0,1].
\]
By virtue of these and \eqref{M-norm}, we obtain
\begin{align*}
|c_i(y)-c_i(x)|&=\left|\int_{0}^{1}\nabla c_i(z_t)^T(y-x)\mathrm{d}t\right| \le\int_{0}^{1}
\|\nabla c_i(z_t)\|_{z_t}^*\|y-x\|_{z_t}\mathrm{d}t \le \frac{U_g^c}{1-\beta}\|y-x\|_x, \\
\|\nabla c_i(y)-\nabla c_i(x)\|_x^*&=\left\|\int_{0}^1\nabla^2 c_i(z_t)(y-x)\mathrm{d}t\right\|_x^*
\leq \int_{0}^1\|\nabla^2 c_i(z_t)(y-x)\|_x^*\mathrm{d}t 
\le\frac{1}{1-\beta} \int_{0}^1\|\nabla^2 c_i(z_t)(y-x)\|^*_{z_t}\mathrm{d}t \\ 
& \le \frac{1}{1-\beta} \int_{0}^1\|\nabla^2 c_i(z_t)\|^*_{z_t}\|y-x\|_{z_t}\mathrm{d}t \le \frac{U_H^c}{(1-\beta)^2}\|y-x\|_x. 
\end{align*}
Hence, \eqref{ci-loc-Lip} and \eqref{gci-loc-Lip} hold. 
\end{proof}

We are now ready to prove Lemma~\ref{lem:level-set-augmented-lagrangian-func}.
\begin{proof}[{Proof of Lemma~\ref{lem:level-set-augmented-lagrangian-func}}]

The proofs of statements (i) and (ii) follow similar arguments to those used in proving \cite[Lemma~4.1]{HL21al}

(iii) Fix $x,y\in\Omega(\delta_f,\delta_c)$ with $\|y-x\|_x\le\beta$ and $1\le i\le m$.  By this, \eqref{Hes-loc-Lip}, \eqref{H-loc-upbd}, {\eqref{per-eq-constr}}, \eqref{LkH}, \eqref{ci-loc-Lip} and $\|c(z_\epsilon)\|\le 1$, 
one has
\begin{align}
&\|\tilde{c}_i(y)\nabla^2 c_i(y) - \tilde{c}_i(x)\nabla^2 c_i(x)\|_{x}^*=\|\tilde{c}_i(y)(\nabla^2 c_i(y)-\nabla^2 c_i(x))+(\tilde{c}_i(y)- \tilde{c}_i(x))\nabla^2 c_i(x)\|_{x}^*\nn\\
&\le|c_i(y)-c_i(z_\epsilon)|\|\nabla^2 c_i(y)-\nabla^2 c_i(x)\|_x^* + |c_i(y)- c_i(x)| \|\nabla^2 c_i(x)\|_x^*\le (1+U^c) L_H^c\|y-x\|_x + \frac{U^c_gU_H^c}{1-\beta}\|y-x\|_x\nn \\
&= \left[(1+U^c)L_H^c+\frac{U^c_gU_H^c}{1-\beta}\right]\|y-x\|_x. \label{ineq:c-Hes-loc}
\end{align}
In addition, by \eqref{g-loc-upbd}, \eqref{ineq:2-local-norm} and \eqref{gci-loc-Lip}, one has
\begin{align}
&\|\nabla c_i(y)\nabla c_i(y)^T-\nabla c_i(x)\nabla c_i(x)^T\|_{x}^*=\|\nabla c_i(y)(\nabla c_i(y)-\nabla c_i(x))^T+ (\nabla c_i(y)-\nabla c_i(x))\nabla c_i(x)^T\|_{x}^*\nn \\
&\le\|\nabla c_i(y)\|_x^* \|\nabla c_i(y)-\nabla c_i(x)\|_x^* +\|\nabla c_i(x)\|_x^* \|\nabla c_i(y)-\nabla c_i(x)\|_x^*\nn \\
&\le\left(\frac{1}{1-\beta} \|\nabla c_i(y)\|_y^* + \|\nabla c_i(x)\|_x^*\right) \|\nabla c_i(y)-\nabla c_i(x)\|_x^*\le\frac{(2-\beta)U_g^cU_H^c}{(1-\beta)^3}\|y-x\|_x. \label{ineq:c-grad-loc}
\end{align}
In view of \eqref{def:ALfunc} and the fact that $\nabla \tilde{c}=\nabla c$ and $\nabla^2 \tilde{c}_i = \nabla^2 c_i$, $1\le i\le m$, we see that
\begin{equation}\label{eq:hessian-augmented-Lagrangian}
\nabla^2_{xx} \cL(x,\lambda^k;\rho_k) = \nabla^2 f(x) +\sum_{i=1}^m\lambda^k_{i}\nabla^2 c_i(x) + \rho_k \sum_{i=1}^m\left(\nabla c_i(x)\nabla c_i(x)^T +  \tilde{c}_i(x)\nabla^2 c_i(x)\right),
\end{equation}
which implies that 
\begin{align*}
\|\nabla^2_{xx}\cL(y,\lambda^k;\rho_k)-&\nabla^2_{xx}\cL(x,\lambda^k;\rho_k)\|_x^* \le 
\|\nabla^2 f(y)-\nabla^2 f(x)\|^*_x+\sum_{i=1}^m|\lambda^k_{i}| \|\nabla^2 c_i(y)-\nabla^2 c_i(x)\|^*_x \\
&\qquad\qquad+ \rho_k \sum_{i=1}^m\left(\|\nabla c_i(y)\nabla c_i(y)^T-\nabla c_i(x)\nabla c_i(x)^T\|^*_x +  \|\tilde{c}_i(y)\nabla^2 c_i(y)-\tilde{c}_i(x)\nabla^2 c_i(x)\|^*_x\right).
\end{align*}
Statement (iii) then follows from this, \eqref{ineq:c-Hes-loc} and  \eqref{ineq:c-grad-loc}.
	
(iv) Notice from \eqref{def:ALfunc} and $\nabla \tilde{c}=\nabla c$ that $\nabla_{x} \cL(x,\lambda^k;\rho_k) = \nabla f(x) +\nabla c(x)\lambda^k + \rho_k \nabla c(x)\tilde{c}(x)$. Also, one can see from \eqref{nearly-feas-level-set}, {\eqref{per-eq-constr}} and $\|c(z_\epsilon)\|\le 1$ that $\|\tilde{c}(x)\|\le 2 + \delta_c$ for any $x\in\cS(\delta_f,\delta_c)$. 
In addition, by \eqref{M-norm}, one can observe that $\|\nabla c_i(x)\nabla c_i(x)^T\|^*_x =(\|\nabla c_i(x)\|^*_x)^2$ for all $i$. {In view of these} and \eqref{eq:hessian-augmented-Lagrangian}, we have
\begin{align*}
&{\|\nabla^2_{x} \cL(x,\lambda^k;\rho_k)\|^*_x \leq \|\nabla f(x)\|_x^*+\sum_{i=1}^m|\lambda^k_{i}|\|\nabla c_i(x)\|^*_x + \rho_k \sum_{i=1}^m |\tilde{c}_i(x)|\|\nabla c_i(x)\|^*_x,}\\
&\|\nabla^2_{xx} \cL(x,\lambda^k;\rho_k)\|^*_x \leq \|\nabla^2 f(x)\|^*_x +\sum_{i=1}^m|\lambda^k_{i}|\|\nabla^2 c_i(x)\|^*_x + \rho_k \sum_{i=1}^m\left((\|\nabla c_i(x)\|^*_x)^2 + |\tilde{c}_i(x)|\|\nabla^2 c_i(x)\|^*_x\right),
\end{align*}
which, together with  \eqref{g-loc-upbd}, \eqref{H-loc-upbd} and the fact that $\|\tilde{c}(x)\|\le 2 + \delta_c$ for any $x\in\cS(\delta_f,\delta_c)$, implies that
\begin{align*}
&{U_{k,g} = \sup_{x\in\cS(\delta_f,\delta_c)}\|\nabla^2_{x}\cL(x,\lambda^k;\rho_k)\|_{x}^*\leq U_g^f + \|\lambda^k\|_1U_g^c + \rho_k\sqrt{m}(2+\delta_c)U_g^c,}\\	
&U_{k,H} = \sup_{x\in\cS(\delta_f,\delta_c)}\|\nabla^2_{xx}\cL(x,\lambda^k;\rho_k)\|_{x}^* \leq 
U^f_H+\|\lambda^k\|_1 U^c_H+\rho_k(m(U^c_g)^2+\sqrt{m}(2 + \delta_c)U^c_H).
\end{align*}
Hence, statement (iv) holds.
\end{proof}

We now prove Theorem~\ref{thm:output-alg1}.

\begin{proof}[{Proof of Theorem~\ref{thm:output-alg1}}]
Suppose that Algorithm \ref{alg:2nd-order-AL-nonconvex} terminates at some iteration $k$, that is, {$\mu_k\le\epsilon/(2\vartheta^{1/2}+2)$} and $\|c(x^{k+1})\|\le\epsilon$ hold. By {$\mu_k\le\epsilon/(2\vartheta^{1/2}+2)$}, $\tilde{\lambda}^{k+1}=\lambda^k+\rho_k\tilde{c}(x^{k+1})$, $\nabla \tilde{c}=\nabla c$, and the second relation in \eqref{algstop:1st-order}, one has
{
\begin{align}
\|\nabla f(x^{k+1})+\nabla c(x^{k+1})\tilde{\lambda}^{k+1}+\mu_k\nabla B(x^{k+1})\|_{x^{k+1}}^*
&=\|\nabla f(x^{k+1})+\nabla \tilde{c}(x^{k+1})(\lambda^k+\rho_k \tilde{c}(x^{k+1}))+\mu_k\nabla B(x^{k+1})\|_{x^{k+1}}^* \nn \\
&=\|\nabla_x\cL_{\mu_k}(x^{k+1},\lambda^k;\rho_k)\|_{x^{k+1}}^*\le\mu_k. \label{ineq:1st-barAL}
\end{align}
}
This along with {$\mu_k>0$} yields that 
\[
\|(\nabla f(x^{k+1})+\nabla c(x^{k+1})\tilde{\lambda}^{k+1})/\mu_k+\nabla B(x^{k+1})\|_{x^{k+1}}^*\le 1.
\] 
By this, $x^{k+1}\in\rmint\cK$ and Lemma~\ref{lem:ppty-bar}(v), one has {$(\nabla f(x^{k+1})+\nabla c(x^{k+1})\tilde{\lambda}^{k+1})/\mu_k\in\cK^*$}, which implies that \eqref{apx-1st-2} holds for $(x^{k+1},\tilde{\lambda}^{k+1})$. We next prove that \eqref{apx-1st-3} holds for $(x^{k+1},\tilde{\lambda}^{k+1})$ with $\epsilon_1=\epsilon$. Indeed, by \eqref{ineq:1st-barAL}, {$\mu_k\le\epsilon/(2\vartheta^{1/2}+2)$}, $x^{k+1}\in\rmint\cK$ and Lemma~\ref{lem:ppty-bar}(i), one has
\begin{align*}
\|\nabla f(x^{k+1})+\nabla c(x^{k+1})\tilde{\lambda}^{k+1}\|_{x^{k+1}}^*
&\le\|\nabla f(x^{k+1})+\nabla c(x^{k+1})\tilde{\lambda}^{k+1}+\mu_k\nabla B(x^{k+1})\|_{x^{k+1}}^*+\mu_k\|\nabla B(x^{k+1})\|_{x^{k+1}}^* \\
&\le\mu_k + \mu_k\vartheta^{1/2}=\epsilon/2<\epsilon,
\end{align*}
and hence \eqref{apx-1st-3} holds for $(x^{k+1},\tilde{\lambda}^{k+1})$ with $\epsilon_1=\epsilon$. In view of these, $\|c(x^{k+1})\|\le\epsilon$ and $x^{k+1}\in\rmint\cK$, we conclude that $x^{k+1}$ is a deterministic $\epsilon$-FOSP of problem~\eqref{model:equa-cnstr}.
	
In addition, {recall from \eqref{algstop:2nd-order} that }
\[
\lambda_{\min}(M_{k+1}^T\nabla^2_{xx}\cL_\mu(x^{k+1},\lambda^k;\rho_k)M_{k+1})\ge-\sqrt{\mu_k}
\]
holds with probability at least $1-\delta$, which implies that  $\hd^TM_{k+1}^T\nabla^2_{xx}\cL_\mu(x^{k+1},\lambda^k;\rho_k)M_{k+1}\hd \ge -{\sqrt{\mu_k}}\|\hd\|^2$ holds for all $\hd\in\bR^n$ with probability at least $1-\delta$. Substituting $\hd=M_{k+1}^{-1}\nabla^2 B(x^{k+1})^{-1/2}d$ in this inequality and using \eqref{Mt}, $\tilde{\lambda}^{k+1}=\lambda^k+\rho_k\tilde{c}(x^{k+1})$, $\nabla \tilde{c}=\nabla c$ and $\nabla^2 \tilde{c}_i=\nabla^2 c_i$, $1\le i\le m$, we obtain that with probability at least $1-\delta$, it holds that 
{
\begin{align*}
&d^T\nabla^2 B(x^{k+1})^{-1/2}\Big(\nabla^2 f(x^{k+1})+\sum_{i=1}^m\tilde{\lambda}_i^{k+1}\nabla^2 c_i(x^{k+1})+\rho_k\nabla c(x^{k+1})\nabla c(x^{k+1})^T+\mu_k\nabla^2 B(x^{k+1})\Big)\nabla^2 B(x^{k+1})^{-1/2}d \\
 &\ge -\sqrt{\mu_k}\|M_{k+1}^{-1}\nabla^2 B(x^{k+1})^{-1/2}d\|^2 \overset{\eqref{Mt}}=-\sqrt{\mu_k}\|d\|^2,\quad \forall d\in\bR^n,
\end{align*}
}
which together with {$\mu_k\le\epsilon/(2\vartheta^{1/2}+2)\in(0,1)$} and $\vartheta \geq 1$ implies that
{
\begin{align*}
&d^T\nabla^2 B(x^{k+1})^{-1/2}\left(\nabla^2 f(x^{k+1})+\sum_{i=1}^m\tilde{\lambda}_i^{k+1}\nabla^2 c_i(x^{k+1})\right)\nabla^2 B(x^{k+1})^{-1/2}d \ge -(\sqrt{\mu_k}+\mu_k)\|d\|^2\\
&\ge -2\sqrt{\mu_k}\|d\|^2 = -2\sqrt{\epsilon/(2\vartheta^{1/2}+2)}\|d\|^2\ge -\sqrt\epsilon\|d\|^2, \quad \forall d\in \cC(x^{k+1}),
\end{align*}
}
where $\cC(\cdot)$ is defined in \eqref{c-cone}. Hence, with probability at least $1-\delta$, the relation \eqref{apx-2nd} holds for $(x^{k+1},\tilde{\lambda}^{k+1})$ with $\epsilon_2=\sqrt{\epsilon}$. In view of this and the fact that $x^{k+1}$ is a deterministic $\epsilon$-FOSP of \eqref{model:equa-cnstr}, we conclude that the output $x^{k+1}$ is an $(\epsilon,\sqrt{\epsilon})$-SOSP of \eqref{model:equa-cnstr} with probability at least $1-\delta$.
\end{proof}


We next provide a proof for Theorem~\ref{thm:out-itr-cmplxity-1}.

\begin{proof}[{Proof of Theorem~\ref{thm:out-itr-cmplxity-1}}]
Notice from \eqref{def:delta0c-rhobar1} that ${\rho}_{\epsilon,1}\ge 2\rho_0$, which along with \eqref{omega-tolerance} and \eqref{T-epsilon-g} implies that
\begin{equation}\label{T-epsilong}
K_{\epsilon} \ \overset{\eqref{T-epsilon-g}}{=} \ {\left\lceil\log\epsilon/\log \omega \right\rceil} \ \overset{\eqref{omega-tolerance}}{=} \ \left\lceil\log 2/\log{r}\right\rceil  \ 
\le \ \log (\rho_{\epsilon,1}\rho_0^{-1})/\log{r}+1.
\end{equation}
Since $\{\rho_k\}$ is either unchanged or increased by a ratio $r$ as $k$ increases, it follows from \eqref{T-epsilong} that 
\begin{equation}\label{rho:upper-bound-Tepsilong}
\max_{0 \le k \le K_{\epsilon}} \rho_k \le{r}^{K_{\epsilon}}\rho_0\overset{\eqref{T-epsilong}}{\le}{r}^{\frac{\log(\rho_{\epsilon,1}\rho_0^{-1})}{\log{r}}+1}\rho_0={r}{\rho}_{\epsilon,1}.
\end{equation}
In addition, observe from Algorithm~\ref{alg:2nd-order-AL-nonconvex} that $\rho_k>2\gamma$ and $\|\lambda^k\|\le \Lambda$. 
By these,  \eqref{Lxk-uppbnd}, and Lemma \ref{tech-1}(ii) with $(x,\lambda,\rho)=(x^{k+1},\lambda^k,\rho_k)$, we obtain that
\begin{equation}\label{bd-ck}
\|\tilde{c}(x^{k+1})\|\le\sqrt{\frac{2(\fh-\fl+\gamma)}{\rho_k-2\gamma}+\frac{\|\lambda^k\|^2}{(\rho_k-2\gamma)^2}}+\frac{\|\lambda^k\|}{\rho_k-2\gamma}\le \sqrt{\frac{2(\fh-\fl+\gamma)}{\rho_k-2\gamma}+\frac{\Lambda^2}{(\rho_k-2\gamma)^2}}+\frac{\Lambda}{\rho_k-2\gamma}.
\end{equation}
On the other hand, notice from $\|c(z_\epsilon)\|\le \epsilon/2$ and {\eqref{per-eq-constr}} that
\begin{equation}\label{cbd-tc}
\|c(x^{k+1})\|\le \|\tilde{c}(x^{k+1})\|+\|c(z_\epsilon)\|\le  \|\tilde{c}(x^{k+1})\| + \epsilon/2.
\end{equation}
	
We now prove that $\overline{K}_{\epsilon}$ is finite. Suppose for contradiction that $\overline{K}_{\epsilon}$ is infinite. By this and \eqref{number-outer-iteration}, one has that $\|c(x^{k+1})\|>\epsilon$ for all $k\ge K_{\epsilon}$. This along with \eqref{cbd-tc} implies that $\|\tilde{c}(x^{k+1})\|>\epsilon/2$ for all $k\ge K_\epsilon$. It then follows that $\|\tilde{c}(x^{k+1})\|>\alpha\|\tilde{c}(x^k)\|$ must hold for infinitely many $k$'s, which, together with the update scheme on $\{\rho_k\}$, further implies $\rho_{k+1}={r}\rho_k$ holds for infinitely many $k$'s. Using this and the monotonicity of $\{\rho_k\}$, we see that $\rho_k\to\infty$ as $k\to\infty$. This along with \eqref{bd-ck} yields that $\|\tilde{c}(x^{k+1})\|\to0$ as $k\to\infty$, which leads to a contradiction with the fact that $\|\tilde{c}(x^{k+1})\|>\epsilon/2$ for all $k\ge K_\epsilon$. Hence, $\overline{K}_{\epsilon}$ is finite. In addition, notice from {$\mu_k=\max\{\epsilon,\omega^k\}/(2\vartheta^{1/2}+2)$} and \eqref{T-epsilon-g} that {$\mu_k=\epsilon/(2\vartheta^{1/2}+2)$} for all $k\ge K_{\epsilon}$. Combining this with the termination criterion of Algorithm \ref{alg:2nd-order-AL-nonconvex} and the definition of $\overline{K}_{\epsilon}$, we conclude that Algorithm \ref{alg:2nd-order-AL-nonconvex} with {$\mu_k=\max\{\epsilon,\omega^k\}/(2\vartheta^{1/2}+2)$} must terminate at iteration $\overline{K}_{\epsilon}$.  
	
We next prove \eqref{outer-iteration-cmplxty} and that $\rho_k\le{r}{\rho}_{\epsilon,1}$ holds for $0 \le k \le \overline{K}_{\epsilon}$ by considering two separate cases below.
	
Case 1) $\|c(x^{K_{\epsilon}+1})\|\le\epsilon$. It then follows from \eqref{number-outer-iteration} that $\overline{K}_{\epsilon}=K_{\epsilon}$, and thus \eqref{outer-iteration-cmplxty} holds due to \eqref{T-epsilong}. In addition, by $\overline{K}_{\epsilon}=K_{\epsilon}$ and \eqref{rho:upper-bound-Tepsilong}, one has that $\rho_k\le{r}{\rho}_{\epsilon,1}$ holds for $0 \le k \le \overline{K}_{\epsilon}$ as well.

Case 2) $\|c(x^{K_{\epsilon}+1})\|>\epsilon$. It then follows from \eqref{number-outer-iteration} that $\overline{K}_{\epsilon}>K_{\epsilon}$, and moreover, $\|c(x^{k+1})\|>\epsilon$ for all $K_{\epsilon} \le k \le \overline{K}_{\epsilon}-1$. This along with \eqref{cbd-tc} implies that
\begin{equation}\label{tc-lbd}
\|\tilde{c}(x^{k+1})\|>\epsilon/2,\quad \forall K_\epsilon\le k\le \overline{K}_\epsilon-1.
\end{equation}
By this, $\|\lambda^k\|\le\Lambda$, \eqref{def:delta0c-rhobar1}, \eqref{Lxk-uppbnd},
 and Lemma \ref{tech-1}(iv) with $(x,\lambda,\rho,\tdc)=(x^{k+1},\lambda^k,\rho_k,\epsilon/2)$, one has
\begin{align}
\rho_k& < 8(\fh-\fl+\gamma)\epsilon^{-2}+4\|\lambda^k\|\epsilon^{-1}+2\gamma \nn\\
&\le 8(\fh-\fl+\gamma)\epsilon^{-2}+4\Lambda\epsilon^{-1}+2\gamma
\overset{\eqref{def:delta0c-rhobar1}}{\le}{\rho}_{\epsilon,1},\quad \forall K_{\epsilon} \le k \le \overline{K}_{\epsilon}-1.\label{bound-rhok}
\end{align}
In view of this, \eqref{rho:upper-bound-Tepsilong}, and the fact $\rho_{\overline{K}_{\epsilon}}\le {r}\rho_{ \overline{K}_{\epsilon}-1}$, we obtain that $\rho_k\le{r}{\rho}_{\epsilon,1}$ holds for $0 \le k\le \overline{K}_{\epsilon}$. It remains to prove \eqref{outer-iteration-cmplxty}. To this end, let 
\[
\bK =\{k:\rho_{k+1}={r}\rho_{k}, K_{\epsilon}\le k\le \overline{K}_{\epsilon}-2\}.
\]
By \eqref{bound-rhok} and the update scheme of $\rho_k$, one has
${r}^{|\bK|}\rho_{K_{\epsilon}}=\max_{K_{\epsilon}\le k \le \overline{K}_{\epsilon}-1}
\rho_k\le {\rho}_{\epsilon,1}$, 
which along with $\rho_{K_{\epsilon}}\ge\rho_0$ implies that
\begin{equation}\label{bound-cU}
|\bK|\le \log({\rho}_{\epsilon,1}\rho_{K_{\epsilon}}^{-1})/ \log{r} \le \log({\rho}_{\epsilon,1}\rho_0^{-1})/ \log{r}.
\end{equation}
Let $\{k_1,k_2,\ldots,k_{|\bK|}\}$ denote all the elements of $\mathcal{\bK}$ arranged in ascending order, and let $k_0=K_{\epsilon}$ and $k_{|\bK|+1}=\overline{K}_{\epsilon}-1$. 
We next derive an upper bound for $k_{j+1}-k_{j}$ for $j=0,1,\ldots,|\bK|$. Using the definition of $\bK$, we see that $\rho_{k}=\rho_{k'}$ for $k_j< k,k'\le k_{j+1}$. By this and the update scheme of $\rho_k$, one can see that
\begin{equation}\label{ineq:sequence-cxk}
\|\tilde{c}(x^{k+1})\|\le\alpha\|\tilde{c}(x^{k})\|,\quad \forall k_j<k< k_{j+1}.
\end{equation}
In addition, by \eqref{def:delta0c-rhobar1xxx}, \eqref{bd-ck} and $\rho_k\ge\rho_0$, one has $\|\tilde{c}(x^{k+1})\|\le\delta_{c,1}$ for $0 \le k \le \overline{K}_{\epsilon}$. Using this and \eqref{tc-lbd}, we obtain that
\begin{equation}\label{bound-cxk}
\epsilon/2<\|\tilde{c}(x^{k+1})\|\le\delta_{c,1},\quad \forall K_{\epsilon} \le k \le \overline{K}_{\epsilon}-1.
\end{equation}
Now, we notice that either $k_{j+1}-k_j=1$ or $k_{j+1}-k_j>1$. In the latter case, one can apply \eqref{ineq:sequence-cxk} with $k=k_{j+1}-1,\ldots,k_j+1$ along with \eqref{bound-cxk} to deduce that
\[
\epsilon/2<\|\tilde{c}(x^{k_{j+1}})\|\le\alpha \|\tilde{c}(x^{k_{j+1}-1})\|\le\cdots\le \alpha^{k_{j+1}-k_j-1}\|\tilde{c}(x^{k_{j}+1})\|\le \alpha^{k_{j+1}-k_j-1}\delta_{c,1}, \quad \forall j=0,1,\ldots,|\bK|.
\]
Combining the two cases, we deduce that
\begin{equation}\label{ineq:subsequence-rho-same}
k_{j+1}-k_{j}\le |\log(\epsilon(2\delta_{c,1})^{-1})/\log \alpha|+1,\quad \forall j=0,1,\ldots,|\bK|.
\end{equation}
Summing up these inequalities, and using \eqref{T-epsilong}, \eqref{bound-cU},  $k_0=K_{\epsilon}$ and $k_{|\bK|+1}=\overline{K}_{\epsilon}-1$, we have
\begin{eqnarray}
&\hspace{-.5in}\overline{K}_{\epsilon} = 1+k_{|\bK|+1}=1+k_0+\sum_{j=0}^{|\bK|}(k_{j+1}-k_{j})\overset{\eqref{ineq:subsequence-rho-same}}{\le}1+K_{\epsilon}+(|\bK|+1)	\left(\left|\frac{\log(\epsilon(2\delta_{c,1})^{-1})}{\log \alpha}\right|+1\right) \nonumber \\[6pt]
&\le2+\frac{\log({\rho}_{\epsilon,1}\rho_0^{-1})}{\log{r}}+\left(\frac{\log({\rho}_{\epsilon,1}\rho_0^{-1})}{\log{r}}+1\right)	\left(\left|\frac{\log(\epsilon(2\delta_{c,1})^{-1})}{\log \alpha}\right|+1\right) =1+ \left(\frac{\log({\rho}_{\epsilon,1}\rho_0^{-1})}{\log{r}}+1\right)	\left(\left|\frac{\log(\epsilon(2\delta_{c,1})^{-1})}{\log \alpha}\right|+2\right), \nn
\end{eqnarray}
where the second inequality is due to \eqref{T-epsilong} and \eqref{bound-cU}. Hence, \eqref{outer-iteration-cmplxty} holds as well in this case. 
\end{proof}

We next prove Theorem~\ref{thm:total-iter-cmplxity}. Before proceeding, we recall from Lemma~\ref{lem:level-set-augmented-lagrangian-func} and the discussions in Section \ref{sbsc:NCGBAL} that the  subproblem {$\min_x\cL_{\mu_k}(x,\lambda^k;\rho_k)$} satisfies Assumptions~\ref{asp:NCG-cmplxity}(b) and \ref{asp:NCG-cmplxity}(c) with $(F(\cdot),\cS,\Omega,L_H^F,U^F_g,U^F_H)=(\cL(\cdot,\lambda^k;\rho_k),\cS(\delta_f,\delta_c),\Omega(\delta_f,\delta_c),L_{k,H},U_{k,g},U_{k,H})$. Moreover, in view of the fact that $\|\lambda^k\| \le \Lambda$, one can see from \eqref{LkH} and Lemma~\ref{lem:level-set-augmented-lagrangian-func}(iv) that there exist some constants ${L}_{1}$, ${L}_{2}$, ${U}_{1}$ and ${U}_{2}$, depending only on  $f$, $c$, $B$, $\beta$, $\Lambda$, $m$, $\bdelta_{f}$ and $\bdelta_{c}$,  such that
\begin{equation}\label{ineq:bound-norm-hessian-Liphessian}
L_{k,H}\le {L}_{1}+\rho_k{L}_{2},\quad U_{k,H}\le{U}_{1}+\rho_k{U}_{2}.
\end{equation}

We are now ready to prove Theorem~\ref{thm:total-iter-cmplxity}.

\begin{proof}[{Proof of Theorem~\ref{thm:total-iter-cmplxity}}]
Let $T_k$ and $N_k$ denote the number of iterations and fundamental operations performed by Algorithm~\ref{alg:NCG} at outer iteration $k$ of Algorithm \ref{alg:2nd-order-AL-nonconvex}, respectively. It then follows from Theorem \ref{thm:out-itr-cmplxity-1} that the total number of iterations of Algorithm~\ref{alg:NCG} performed in Algorithm \ref{alg:2nd-order-AL-nonconvex} is $\sum_{k=0}^{\overline{K}_{\epsilon}}T_k$, and moreover, the total number of Cholesky factorizations and other fundamental operations performed by Algorithm~\ref{alg:NCG} in Algorithm~\ref{alg:2nd-order-AL-nonconvex} are $\sum_{k=0}^{\overline{K}_{\epsilon}}T_k$ and $\sum_{k=0}^{\overline{K}_{\epsilon}}N_k$, respectively. In addition, notice from \eqref{def:delta0c-rhobar1} and Theorem \ref{thm:out-itr-cmplxity-1} that ${\rho}_{\epsilon,1}=\cO(\epsilon^{-2})$ and  $\rho_k\le{r}{\rho}_{\epsilon,1}$, which yield $\rho_k=\cO(\epsilon^{-2})$ for all $0 \leq k \leq \overline{K}_{\epsilon}$.
	
(i) Recall from Lemmas~\ref{lem:level-set-augmented-lagrangian-func}(i) and \ref{lem:level-set-augmented-lagrangian-func}(iii) that $L_{k,H}$ is a Lipschitz constant of $\nabla^2_{xx} \cL(x,\lambda^k;\rho_k)$ with respect to the local norm on an open convex neighborhood of {$\{x\in\rmint\cK:\cL_{\mu_k}(x,\lambda^k;\rho_k)\le \cL_{\mu_k}(x^{k}_{\init},\lambda^k;\rho_k)\}$}. In addition, recall from Lemma \ref{lem:level-set-augmented-lagrangian-func}(ii) that {$\inf_{x\in\rmint\cK} \cL_{\mu_k}(x,\lambda^k;\rho_k) \ge\fl-\gamma-\Lambda\bdelta_c$}. By these, 
\eqref{L-xinit},
\eqref{ineq:bound-norm-hessian-Liphessian}, and Theorem \ref{thm:NCG-iter-oper-cmplxity}(iii) with {$(\phih,\phil,L_{H}^F,\epsilon_g,\epsilon_H)=(\cL_{\mu_k}(x^{k}_{\init},\lambda^k;\rho_k),\fl-\gamma-\Lambda\bdelta_c,L_{k,H},\mu_k,\sqrt{\mu_k})$}, one has 
\begin{equation}
{
T_k =\cO((\fh-\fl+\gamma+\Lambda\delta_{c})L_{k,H}^2\mu_k^{-3/2})
\overset{\eqref{ineq:bound-norm-hessian-Liphessian}}{=}\cO(\rho_k^2\mu_k^{-3/2}) = \cO(\epsilon^{-11/2}), \label{eq:alg1-inner-iter-kthAL}}
\end{equation}
where the last equality follows from {$\mu_k\ge\epsilon/(2\vartheta^{1/2}+2)$} and $\rho_k=\cO(\epsilon^{-2})$. On the other hand, if $c$ is assumed to be affine, namely, $c(x)=Ax-b$ for some $A\in\bR^{m\times n}$ and $b\in\bR^m$, then $\nabla c(x)=A^T$ and $\nabla^2 c_i(x)=0$ for $1 \le i \le m$. Using these and \eqref{eq:hessian-augmented-Lagrangian}, we observe that  $L_{k,H}=\cO(1)$. By this and the similar arguments as for \eqref{eq:alg1-inner-iter-kthAL}, one has {$T_k=\cO(\mu_k^{-3/2})=\cO(\epsilon^{-3/2})$}. Combining these with \eqref{eq:alg1-inner-iter-kthAL} and $\overline{K}_{\epsilon}=\cO(|\log\epsilon|^2)$ (see Remark \ref{order-out-itera-penalty}), we conclude that statement (i) holds.
	
(ii) By Lemmas~\ref{lem:level-set-augmented-lagrangian-func}(i) and \ref{lem:level-set-augmented-lagrangian-func}(iv), one has 
\[
{
U_{k,H} \ge \sup\limits_{x\in\rmint\cK}\{\|\nabla^2_{xx} \cL(x,\lambda^k;\rho_k)\|_{x}^*: \cL_{\mu_k}(x,\lambda^k;\rho_k)\le \cL_{\mu_k}(x^{k}_{\init},\lambda^k;\rho_k)\}.
}
\] 
In view of this, {$\cL_{\mu_k}(x^{k}_{\init},\lambda^k;\rho_k)\le \fh$}, \eqref{ineq:bound-norm-hessian-Liphessian}, and Theorem \ref{thm:NCG-iter-oper-cmplxity}(iv) with {$(\phih,\phil,L_{H}^F,U_{H}^F,\epsilon_g,\epsilon_H)=(\cL_{\mu_k}(x^{k}_{\init},\lambda^k;\rho_k),\\\fl-\gamma-\Lambda\bdelta_c,L_{k,H},U_{k,H},\mu_k,\sqrt{\mu_k})$}, we obtain that
{
\begin{equation}\label{eq:alg1-oper-kthAL}
\begin{array}{rcl}
N_k&=&\widetilde{\cO}((\fh-\fl+\gamma+\Lambda\bdelta_c)L_{k,H}^2\mu_k^{-3/2}\min\{n,U_{k,H}^{1/2}\mu_k^{-1/4}\})\\[4pt]
&\overset{\eqref{ineq:bound-norm-hessian-Liphessian}}{=}&\widetilde{\cO}(\rho_k^2\mu_k^{-3/2}\min\{n,\rho_k^{1/2}\mu_k^{-1/4}\})\ =\ \widetilde{\cO}\left(\epsilon^{-11/2}\min\left\{n,\epsilon^{-5/4}\right\}\right),	
\end{array}
\end{equation}
}
where the last equality follows from {$\mu_k\ge\epsilon/(2\vartheta+2)$} and $\rho_k=\cO(\epsilon^{-2})$. On the other hand, if $c$ is assumed to be affine, it follows from the above discussion that $L_{k,H}=\cO(1)$. By this, $U_{k,H}\le U_1+\rho_k U_2$, and the similar arguments as for \eqref{eq:alg1-oper-kthAL}, one has {$N_k=\widetilde{\cO}(\mu_k^{-3/2}\min\{n,\rho_k^{1/2}\mu_k^{-1/4}\})=\widetilde{\cO}\left(\epsilon^{-3/2}\min\left\{n,\epsilon^{-5/4}\right\}\right)$}. Combining these with \eqref{eq:alg1-oper-kthAL} and $\overline{K}_{\epsilon}=\cO(|\log\epsilon|^2)$ (see Remark \ref{order-out-itera-penalty}), we conclude that statement (ii) holds.
\end{proof}

We next establish two technical lemmas that will be used to prove Theorem~\ref{thm:total-iter-cmplxity2}.

\begin{lemma}\label{lem:multiplier-update-without-proj}
Suppose that Assumptions \ref{asp:lowbd-knownfeas} and \ref{asp:LICQ} hold and that $\rho_0$ is sufficiently large such that $\delta_{f,1}\le\delta_f$ and $\delta_{c,1}\le\delta_c$, where $\delta_{f,1}$ and $\delta_{c,1}$ are defined in \eqref{def:delta0c-rhobar1xxx}. Let $\{(x^k,\lambda^k,\rho_k)\}$ be generated by Algorithm~\ref{alg:2nd-order-AL-nonconvex}. Suppose that
\begin{equation}\label{rho-bd2}
\rho_k\ge\max\{\Lambda^2(2\bdelta_f)^{-1},2(\fh-\fl+\gamma)\bdelta_c^{-2}+2\Lambda\bdelta_c^{-1}+2\gamma,2(U^f_g + \sqrt{m}U_g^c\Lambda + \sqrt{\vartheta} +1)(\sigma\epsilon)^{-1}\}
\end{equation}
for some $k\ge0$, where $\gamma$, $\fh$, $\fl$, $\bdelta_f$, $\bdelta_c$, ${U^f_g}$ and $U_g^c$ are given in Assumption \ref{asp:lowbd-knownfeas}, and $\sigma$ is given in \eqref{sigma}. Then it holds that $\|c(x^{k+1})\|\le\epsilon$.
\end{lemma}

\begin{proof}
Using $\|\lambda^k\|\le\Lambda$ (see step \ref{algstep:proj-multiplier} of Algorithm~\ref{alg:2nd-order-AL-nonconvex}) and \eqref{rho-bd2}, we have
\[
\rho_k\ge\max\{\|\lambda^k\|^2(2\bdelta_f)^{-1},2(\fh-\fl+\gamma)\bdelta_c^{-2}+2\|\lambda^k\|\bdelta_c^{-1}+2\gamma\}.
\]
By this, \eqref{Lxk-uppbnd}, and Lemmas~\ref{tech-1}(iii) and \ref{tech-1}(iv) with {$(x,\lambda,\mu,\rho,\tdf,\tdc)=(x^{k+1},\lambda^k,\mu_k,\rho_k,\bdelta_f,\bdelta_c)$}, one has {$f(x^{k+1})+\mu_k B(x^{k+1})\le \fh+\bdelta_f$} and $\|\tilde{c}(x^{k+1})\|\le\bdelta_c$. Also, notice from $\|c(z_\epsilon)\|\le 1$ and {\eqref{per-eq-constr}} that $\|c(x^{k+1})\|\le 1 + \|\tilde{c}(x^{k+1})\|$. These along with \eqref{nearly-feas-level-set}, $x^{k+1}\in\rmint\cK$, and {$\mu_k\in(0,\mu_0]$} yield that $x^{k+1}\in\cS(\delta_f,\delta_c)$. It then follows from \eqref{g-loc-upbd} that $\|\nabla f(x^{k+1})\|_{x^{k+1}}^*\le{U^f_g}$ and $\|\nabla c_i(x^{k+1})\|_{x^{k+1}}^*\le U^c_g$ for all $1\le i\le m$. By these, {$\mu_k\le1$}, $\|\lambda^k\|\le\Lambda$, {\eqref{per-eq-constr}} and the second relation in \eqref{algstop:1st-order}, one has
{
\begin{align}
&\rho_k\|\nabla^2 B(x^{k+1})^{-1/2}\nabla c(x^{k+1})\tilde{c}(x^{k+1})\|=\rho_k\|\nabla c(x^{k+1})\tilde{c}(x^{k+1})\|_{x^{k+1}}^* \nn \\
&\le\|\nabla f(x^{k+1})+\nabla c(x^{k+1})\lambda^k\|_{x^{k+1}}^* + \mu_k\|\nabla B(x^{k+1})\|_{x^{k+1}}^*+ \|\nabla_x \cL_{\mu_k}(x^{k+1},\lambda^k;\rho_k)\|_{x^{k+1}}^* \nn \\
&\le\|\nabla f(x^{k+1})\|_{x^{k+1}}^*+\sum_{i=1}^m|\lambda_i^k|\|\nabla c_i(x^{k+1})\|_{x^{k+1}}^* + \mu_k\sqrt{\vartheta}+\mu_k\le U^f_g + \sqrt{m}U_g^c\Lambda + \sqrt{\vartheta} +1,
\label{grad-bound}
\end{align}
}
where the first inequality follows from the triangle inequality, and the second inequality follows from $\|\nabla B(x^{k+1})\|_{x^{k+1}}^*=\sqrt{\vartheta}$ and the second relation in \eqref{algstop:1st-order}. In addition, by $x^{k+1}\in\cS(\bdelta_f,\bdelta_c)$ and \eqref{sigma}, one has \\
$\lambda_{\min}(\nabla c(x^{k+1})^T\nabla^2 B(x^{k+1})^{-1}\nabla c(x^{k+1}))\ge\sigma^2$, which along with \eqref{grad-bound} implies that 
\begin{align}
\|\tilde{c}(x^{k+1})\|&\le \|[\nabla c(x^{k+1})^T\nabla^2 B(x^{k+1})^{-1}\nabla c(x^{k+1})]^{-1}\nabla c(x^{k+1})^T\nabla^2 B(x^{k+1})^{-1/2}\| \|\nabla^2 B(x^{k+1})^{-1/2}\nabla c(x^{k+1})\tilde{c}(x^{k+1})\|\nn \\[4pt]
&=\|[\nabla c(x^{k+1})^T\nabla^2 B(x^{k+1})^{-1}\nabla c(x^{k+1})]^{-1}\|^{1/2}
\|\nabla^2 B(x^{k+1})^{-1/2}\nabla c(x^{k+1})\tilde{c}(x^{k+1})\|\nn \\[4pt]
&=\lambda_{\min}(\nabla c(x^{k+1})^T\nabla^2 B(x^{k+1})^{-1}\nabla c(x^{k+1}))^{-1/2}\|\nabla^2 B(x^{k+1})^{-1/2}\nabla c(x^{k+1})\tilde{c}(x^{k+1})\|\nn \\[4pt]
&\le(U^f_g + \sqrt{m}U_g^c\Lambda + \sqrt{\vartheta} +1)/(\sigma\rho_k).
\label{ineq:1st-order-lambda-bound}   
\end{align}
Observe from \eqref{rho-bd2} that $\rho_k\ge2(U^f_g + \sqrt{m}U_g^c\Lambda + \sqrt{\vartheta} +1)(\sigma\epsilon)^{-1}$, which along with \eqref{ineq:1st-order-lambda-bound} implies $\|\tilde{c}(x^{k+1})\|\le\epsilon/2$. Combining this with $\|c(z_\epsilon)\|\le\epsilon/2$ and {\eqref{per-eq-constr}}, we obtain $\|c(x^{k+1})\|\le\epsilon$ as desired.
\end{proof}

The next lemma establishes a stronger upper bound for $\{\rho_k\}$ than the one given in Theorem \ref{thm:out-itr-cmplxity-1}.

\begin{lemma}\label{cor:improved-outer-iteration-cmplxty}
Suppose that Assumptions \ref{asp:lowbd-knownfeas} and \ref{asp:LICQ} hold and that $\rho_0$ is sufficiently large such that $\delta_{f,1}\le\delta_f$ and $\delta_{c,1}\le\delta_c$, where $\delta_{f,1}$ and $\delta_{c,1}$ are defined in \eqref{def:delta0c-rhobar1xxx}. Let $\{\rho_k\}$ be generated by Algorithm \ref{alg:2nd-order-AL-nonconvex} and
\begin{equation}\label{bound-barrho2}
{\rho}_{\epsilon,2}:=\max\{\Lambda^2(2\bdelta_f)^{-1},2(\fh-\fl+\gamma)\bdelta_c^{-2}+2\Lambda\bdelta_c^{-1}+2\gamma,2(U^f_g + \sqrt{m}U_g^c\Lambda + \sqrt{\vartheta} +1)(\sigma\epsilon)^{-1},2\rho_0\},
\end{equation}
where $\gamma$, $\fh$, $\fl$, $\bdelta_f$, $\bdelta_c$, ${U^f_g}$ and $U_g^c$ are given in Assumption \ref{asp:lowbd-knownfeas}, and $\sigma$ is given in \eqref{sigma}. Then $\rho_k\le{r}{\rho}_{\epsilon,2}$ holds for $0 \le k\le \overline{K}_{\epsilon}$, where $\overline{K}_{\epsilon}$ is defined in \eqref{number-outer-iteration}.
\end{lemma}

\begin{proof}
Observe from \eqref{bound-barrho2} that ${\rho}_{\epsilon,2} \ge 2\rho_0$. Using this and similar arguments as for \eqref{T-epsilong}, we have $K_{\epsilon}\le\log({\rho}_{\epsilon,2}\rho_0^{-1})/\log{r}+1$, where $K_{\epsilon}$ is defined in \eqref{T-epsilon-g}. By this, the update scheme for $\{\rho_k\}$, and similar arguments as for \eqref{rho:upper-bound-Tepsilong}, one has
\begin{equation}\label{bound-rhok-Tepsilon}
\max_{0 \le k \le K_{\epsilon}} \rho_k \le {r}{\rho}_{\epsilon,2}.
\end{equation}
If $\|c(x^{K_{\epsilon}+1})\|\le\epsilon$, it follows from \eqref{number-outer-iteration} that $\overline{K}_{\epsilon}=K_{\epsilon}$, which along with \eqref{bound-rhok-Tepsilon} implies that $\rho_k\le{r}{\rho}_{\epsilon,2}$ holds for $0 \le k\le \overline{K}_{\epsilon}$. On the other hand,  if $\|c(x^{K_{\epsilon}+1})\|>\epsilon$, it follows from \eqref{number-outer-iteration} that $\|c(x^{k+1})\|>\epsilon$ for $K_{\epsilon} \le k \le \overline{K}_{\epsilon}-1$, which together with Lemma~\ref{lem:multiplier-update-without-proj} and \eqref{bound-barrho2} implies that
\[
\rho_k <\max\left\{\frac{\Lambda^2}{2\bdelta_f},\frac{2(\fh-\fl+\gamma)}{\bdelta_c^2}+\frac{2\Lambda}{\bdelta_c}+2\gamma,\frac{2(U^f_g + \sqrt{m}U_g^c\Lambda + \sqrt{\vartheta} +1)}{\sigma\epsilon}\right\}\overset{\eqref{bound-barrho2}}{\le}{\rho}_{\epsilon,2},\quad  \forall K_{\epsilon} \le k \le \overline{K}_{\epsilon}-1.
\]
Using this, \eqref{bound-rhok-Tepsilon}, and $\rho_{\overline{K}_{\epsilon}}\le{r}\rho_{\overline{K}_{\epsilon}-1}$, we also conclude that $\rho_k\le{r}{\rho}_{\epsilon,2}$ holds for $0 \le k\le \overline{K}_{\epsilon}$.
\end{proof}

We now provide a proof for Theorem~\ref{thm:total-iter-cmplxity2}.

\begin{proof}[{Proof of Theorem~\ref{thm:total-iter-cmplxity2}}]
Notice from \eqref{bound-barrho2} and Lemma \ref{cor:improved-outer-iteration-cmplxty} that ${\rho}_{\epsilon,2}=\cO(\epsilon^{-1})$ and $\rho_k \le {r}{\rho}_{\epsilon,2}$, which imply $\rho_k=\cO(\epsilon^{-1})$. The rest of the proof follows from the same arguments as for Theorem \ref{thm:total-iter-cmplxity} with $\rho_k=\cO(\epsilon^{-2})$ replaced by $\rho_k=\cO(\epsilon^{-1})$.
\end{proof}

\section{Concluding remarks}\label{sec:cr}
In this paper we proposed a Newton-CG based barrier-AL method for finding an approximate SOSP of general nonconvex conic optimization problem \eqref{model:equa-cnstr}. We also established the worst-case iteration and operation complexity bounds of the proposed method for finding an approximate SOSP of problem~\eqref{model:equa-cnstr}. In addition, we conducted preliminary numerical experiments to demonstrate the superior solution quality of our method over a well-known first-order method, SpaRSA.



There are several potential directions for future research. Firstly, conducting extensive numerical studies could provide new insights into improving the practical performance of our method. Secondly, it would be interesting to extend our method to solve a more general conic optimization problem, $\min_{x, y}\{\tf(x, y) : \tc(x, y) = 0, \ y \in \cK\}$, which includes the problem $\min_{x}\{f(x) : c(x) = 0, \ d(x) \leq 0\}$ and, more generally, the problem $\min_{x}\{f(x) : c(x) = 0, \ d(x) \in \cK\}$ as special cases. Notice the latter problem can be equivalently solved as $\min_{x, y}\{f(x) : c(x) = 0, \ d(x) - y = 0, \ y \in \cK\}$, which is a specific instance of the problem considered in this paper. Consequently, it can be suitably solved by our proposed method. 
Lastly, extending our approach to finding an approximate SOSP for nonconvex optimization problems with a general convex set constraint, beyond the convex conic constraint, remains an open question.

\vspace{-.4in}
\section*{}

\noindent {\bf Data availability:} The codes for generating the random data and implementing the algorithms
in the numerical section are available from the first author upon request. 

\vspace{-.5in}
\section*{}

{\bf Competing interests:} The third author is an editorial board member of this journal.

\bibliographystyle{abbrv}
\bibliography{references}

\section*{Appendix}

\appendix

\section{A capped conjugate gradient method}\label{appendix:capped-CG}

\begin{algorithm}[h]
\caption{A capped conjugate gradient method}
\label{alg:capped-CG}
{\footnotesize
\begin{algorithmic}
\State \noindent\textit{Input}: symmetric matrix $H\in\bR^{n\times n}$, vector $g\neq0$, damping parameter $\varepsilon\in(0,1)$, desired relative accuracy $\zeta\in(0,1)$.
\State \textit{Optional input:} scalar $U\ge0$ such that $\|H\|\le U$ (set to $0$ if not provided).
\State \textit{Outputs:} $\hat{d}$, d$\_$type.
\State \textit{Secondary outputs:} final values of $U,\kappa,\hat{\zeta},\tau,$ and $T$.
\State Set 
\begin{equation*}
\bar{H}:=H+2\varepsilon I,\quad \kappa:=\frac{U+2\varepsilon}{\varepsilon},\quad\hat{\zeta}:=\frac{\zeta}{3\kappa},\quad\tau:=\frac{\sqrt{\kappa}}{\sqrt{\kappa}+1},\quad T:=\frac{4\kappa^4}{(1-\sqrt{\tau})^2},
\end{equation*}
$y^0\leftarrow 0,r^0\leftarrow g,p^0\leftarrow -g, j\leftarrow 0$.
\If {$(p^0)^T \bar{H}p^0<\varepsilon\|p^0\|^2$}
\State Set $\hat{d}\leftarrow p^0$ and terminate with d$\_$type = NC;
\ElsIf {\ $\|Hp^0\|>U\|p^0\|\ $}
\State Set $U\leftarrow\|Hp^0\|/\|p^0\|$ and update $\kappa,\hat{\zeta},\tau, T$ accordingly;
\EndIf
\While{TRUE}
\State $\alpha_j\leftarrow (r^j)^T r^j/(p^j)^T\bar{H}p^j$; \{Begin Standard CG Operations\}
\State $y^{j+1}\leftarrow y^j+\alpha_jp^j$;
\State $r^{j+1}\leftarrow r^j+\alpha_j\bar{H}p^j$;
\State $\beta_{j+1}\leftarrow\|r^{j+1}\|^2/\|r^j\|^2$;
\State $p^{j+1}\leftarrow-r^{j+1}+\beta_{j+1}p^j$; \{End Standard CG Operations\}
\State $j\leftarrow j+1$;
\If {$\|Hp^j\|>U\|p^j\|$}
\State Set $U\leftarrow\|Hp^j\|/\|p^j\|$ and update $\kappa,\hat{\zeta},\tau,T$ accordingly;
\EndIf 
\If {\ $\|Hy^j\|>U\|y^j\|\ $}
\State Set $U\leftarrow\|Hy^j\|/\|y^j\|$ and update $\kappa,\hat{\zeta},\tau,T$ accordingly;
\EndIf      
\If {\ $\|Hr^j\|>U\|r^j\|\ $}
\State Set $U\leftarrow\|Hr^j\|/\|r^j\|$ and update $\kappa,\hat{\zeta},\tau,T$ accordingly;
\EndIf
\If {$(y^j)^T\bar{H}y^j<\varepsilon\|y^j\|^2$}
\State Set $\hat{d}\leftarrow y^j$ and terminate with d$\_$type = NC;
\ElsIf {\ $\|r^j\|\le\hat{\zeta}\|r^0\|$}
\State Set $\hat{d}\leftarrow y^j$ and terminate with d$\_$type = SOL;
\ElsIf{\ $(p^j)^T\bar{H}p^j<\varepsilon\|p^j\|^2$}
\State Set $\hat{d}\leftarrow p^j$ and terminate with d$\_$type = NC;  
\ElsIf {\ $\|r^j\|>\sqrt{T}\tau^{j/2}\|r^0\| $}
\State Compute $\alpha_j, y^{j+1}$ as in the main loop above;
\State Find $i\in\{0,\ldots,j-1\}$ such that
\[
(y^{j+1}-y^i)^T\bar{H}(y^{j+1}-y^i)<\varepsilon\|y^{j+1}-y^i\|^2;
\]
\State Set $\hat{d}\leftarrow y^{j+1}-y^i$ and terminate with d$\_$type = NC;
\EndIf
\EndWhile
\end{algorithmic}
}
\end{algorithm}

In this part we present the capped CG method proposed in \cite[Algorithm~1]{ROW20} for solving a possibly indefinite linear system~\eqref{indef-sys}. As briefly discussed in Section~\ref{sec:sbpb-solver}, the capped CG method finds either an approximate solution to \eqref{indef-sys} or a sufficiently negative curvature direction of the associated matrix $H$. More details about this method can be found in \cite[Section~3.1]{ROW20}.

The following theorem presents the iteration complexity of Algorithm~\ref{alg:capped-CG}, whose proof can be found in \cite[Theorem~A.1]{HL21al}, and thus omitted here.

\begin{theorem}[{\bf iteration complexity of Algorithm \ref{alg:capped-CG}}]\label{lem:capped-CG}
Consider applying Algorithm~\ref{alg:capped-CG} with the optional input $U=0$ to the linear system~\eqref{indef-sys} with $g\neq 0$, $\varepsilon>0$, and $H$ being an $n\times n$ symmetric matrix. Then the number of iterations of Algorithm~\ref{alg:capped-CG} is $\widetilde{\cO}(\min\{n,\sqrt{\|H\|/\varepsilon}\})$.
\end{theorem}

\section{A randomized Lanczos based minimum eigenvalue oracle} \label{appendix:meo}

\begin{algorithm}[h]
\caption{A randomized Lanczos based minimum eigenvalue oracle}
\label{pro:meo}
{\small
\noindent\textit{Input}: symmetric matrix $H\in\bR^{n\times n}$, tolerance $\varepsilon>0$, and probability parameter $\delta\in(0,1)$.\\
\noindent\textit{Output:} a sufficiently negative curvature direction $v$ satisfying $v^THv\le-\varepsilon/2$ and $\|v\|=1$; or a certificate that $\lambda_{\min}(H)\ge-\varepsilon$ with probability at least  $1-\delta$.\\
Apply the Lanczos method \cite{KW92LR} to estimate $\lambda_{\min}(H)$ starting with a random vector uniformly generated on the unit sphere, and run it for at most
\begin{equation}\label{N-iter}
N(\varepsilon,\delta):=\min\left\{n,1+\left\lceil\frac{\ln(2.75n/\delta^2)}{2}\sqrt{\frac{\|H\|}{\varepsilon}}\right\rceil\right\}
\end{equation}
iterations.
\begin{enumerate}[{\rm (i)}]
\item[(i)]
If it finds a unit vector $v$ such that $v^THv \le -\varepsilon/2$ at some iteration,  it terminates immediately and returns $v$.
\item[(ii)]Otherwise, it certifies that $\lambda_{\min}(H)\ge-\varepsilon$ holds with probability at least $1-\delta$.
\end{enumerate}
}
\end{algorithm}

In this part we present the randomized Lanczos method proposed in \cite[Section~3.2]{ROW20}, which can be used as a minimum eigenvalue oracle for Algorithm~\ref{alg:NCG}. As mentioned in Section~\ref{sec:sbpb-solver}, this oracle either outputs a sufficiently negative curvature direction of $H$ or certifies that $H$ is nearly positive semidefinite with high probability. More details about it can be found in \cite[Section~3.2]{ROW20}.

The following theorem justifies that Algorithm~\ref{pro:meo} is a suitable minimum eigenvalue oracle for Algorithm~\ref{alg:NCG}. Its proof is identical to that of \cite[Lemma~2]{ROW20} and thus omitted.

\begin{theorem}[{\bf iteration complexity of Algorithm~\ref{pro:meo}}]\label{rand-Lanczos}
Consider Algorithm~\ref{pro:meo} with tolerance $\varepsilon>0$, probability parameter $\delta\in(0,1)$, and symmetric matrix $H\in\bR^{n\times n}$ as its input. Then it either finds a sufficiently negative curvature direction $v$ satisfying $v^THv\le-\varepsilon/2$ and $\|v\|=1$ or certifies that $\lambda_{\min}(H)\ge-\varepsilon$ holds with probability at least  $1-\delta$  in at most $N(\varepsilon,\delta)$ iterations, where $N(\varepsilon,\delta)$ is defined in \eqref{N-iter}.
\end{theorem}

Notice that generally, computing $\|H\|$ in Algorithm~\ref{pro:meo} may not be cheap when $n$ is large. Nevertheless, $\|H\|$ can be efficiently estimated via a randomization scheme with high confidence (e.g., see the discussion in \cite[Appendix~B3]{ROW20}).

\end{document}